\documentclass[A4]{amsart}
\usepackage[square,sort,comma,numbers]{natbib}
\usepackage{amsmath, amssymb, amstext, amsfonts, textcomp, amsxtra, amsbsy, amsgen, amsopn, amscd, mathrsfs, amsthm, latexsym, array}
\usepackage[textwidth=16cm,textheight=22cm,centering]{geometry}

\usepackage{color}
\usepackage{graphicx}

\usepackage[all]{xy}

\usepackage{bbm, dsfont}

\newtheorem{theorem}             {Theorem}  [section]
\newtheorem{definition} [theorem] {Definition}
\newtheorem{lemma}      [theorem]{Lemma}
\newtheorem{corollary}  [theorem]{Corollary}
\newtheorem{proposition}[theorem]{Proposition}
\newtheorem{remark} [theorem] {Remark}

\numberwithin{equation}{section} \everymath{\displaystyle}

\newcommand{\Cont}{{\rm C}}
\newcommand{\Aut}{\mathcal{A}}
\newcommand{\Sob}{{\rm S}}
\newcommand{\Sch}{\mathcal{S}}
\newcommand{\sgn}{{\rm sgn}}
\newcommand{\intL}{{\rm L}}
\newcommand{\Ht}{{\rm Ht}}

\newcommand{\Nr}{{\rm Nr}}
\newcommand{\Tr}{{\rm Tr}}

\newcommand{\hol}{{\rm hol}}

\newcommand{\gp}[1]{\mathbf{#1}}
\newcommand{\GL}{{\rm GL}}
\newcommand{\PGL}{{\rm PGL}}

\newcommand{\SO}{{\rm SO}}
\newcommand{\PSO}{{\rm PSO}}
\newcommand{\SU}{{\rm SU}}

\newcommand{\grB}{\mathbf{B}}

\newcommand{\grN}{\mathbf{N}}

\newcommand{\ag}[1]{\mathbb{#1}}

\newcommand{\Z}{\mathbb{Z}}

\newcommand{\Casimir}{\mathcal{C}}

\newcommand{\Q}{\mathbb{Q}}
\newcommand{\R}{\mathbb{R}}
\newcommand{\C}{\mathbb{C}}
\newcommand{\F}{\mathbf{F}}
\newcommand{\A}{\mathbb{A}}
\newcommand{\vo}{\mathfrak{o}}
\newcommand{\vp}{\mathfrak{p}}
\newcommand{\idlN}{\mathfrak{N}}

\newcommand{\Dis}{{\rm D}}

\newcommand{\Proj}{{\rm P}}
\newcommand{\ProjP}{{\rm P}}
\newcommand{\norm}[1][\cdot]{\lvert #1 \rvert}
\newcommand{\extnorm}[1]{\left\lvert #1 \right\rvert}
\newcommand{\Norm}[1][\cdot]{\lVert #1 \rVert}
\newcommand{\extNorm}[1]{\left\lVert #1 \right\rVert}

\newcommand{\Pairing}[2]{\langle #1, #2 \rangle}

\newcommand{\rpR}{{\rm R}}
\newcommand{\Bas}{\mathcal{B}}

\newcommand{\Res}{{\rm Res}}
\newcommand{\Ind}{{\rm Ind}}

\newcommand{\Intw}{\mathcal{M}}
\newcommand{\IntwR}{\mathcal{R}}
\newcommand{\Whi}{\mathcal{W}}

\newcommand{\Cond}{\mathbf{C}}
\newcommand{\cond}{\mathfrak{c}}

\newcommand{\fin}{{\rm fin}}
\newcommand{\eis}{{\rm E}}
\newcommand{\Reis}{\mathcal{E}}
\newcommand{\eisCst}{{\rm E}_{\grN}}
\newcommand{\reg}{{\rm reg}}

\newcommand{\freg}{{\rm fr}}

\newcommand{\Vol}{{\rm Vol}}
\makeatletter
\newcommand{\rmnum}[1]{\romannumeral #1}
\newcommand{\Rmnum}[1]{\expandafter\@slowromancap\romannumeral #1@}
\makeatother

\newcommand{\Ex}{\mathcal{E}{\rm x}}

\title{Explicit Subconvexity for $\GL_2$}
\author{Han Wu}

\begin{document}

\begin{abstract}
	We make the subconvex exponent for $\mathrm{GL}_2$ cuspidal representation in the work of Michel \& Venkatesh explicit. The result depends on an effective dependence on the ``fixed'' $\mathrm{GL}_2$ representation in our former work on the subconvex bounds for twists by Hecke characters, which in turn depends on the $\mathrm{L}^4$-norm of the test function.
\end{abstract}

	\maketitle
	
	\tableofcontents

\section{Introduction}

	\subsection{Main Result}

	Let $\F$ be a number field with ring of adeles $\A$. Let $\pi$ be an automorphic cuspidal representation of $\GL_2(\A)$. This is the natural generalization of Hecke characters $\chi$ of $\F^{\times} \backslash \A^{\times}$ to the $\GL_2$ setting, hence generalization of Dirichlet $L$-functions (for $\F = \Q$) in particular. Similarly, we also have the associated $L$-function $L(s,\pi)$. While good and uniform bounds for Hecke $L$-functions $L(s,\chi)$ are so far available in the literature, in particular in the case $\F=\Q$, no bounds for $L(s,\pi)$ of similar quality are known. In particular, if $\omega$ denotes the central character of $\pi$ and if $\omega$ \emph{varies with} $\pi$, the known subconvex bounds for $L(1/2,\pi)$ are of poor quality, especially for the level aspect. For example:
\begin{itemize}
	\item[(1)] Over $\F=\Q$, for a Maass form $f$ of level $q$ and (necessarily even) primitive central character $\omega$, subconvexity in the $q$-aspect was solved by Duke, Friedlander and Iwaniec \cite[Theorem 2.4]{DFI02} with a power saving $q^{1/23041}$ from the convex bound $q^{1/4}$.
	\item[(2)] In the same setting, the above method was simplified and generalized by Blomer, Harcos and Michel \cite{BHM07GL2} with an improvement on the power saving $q^{1/1889}$ \cite[Theorem 2]{BHM07GL2}.
	\item[(3)] Recently, Blomer-Khan \cite[Theorem 1]{BK18} improves the above power saving to $q^{1/128}$ in the hybrid case when $q$ is prime. (In fact, a general uniform subconvex bound is obtained. However, its effectiveness depends on the unspecified polynomial dependence of the usual conductor in Ivi\'c's bound \cite[Corollary 2]{Iv01}.)
\end{itemize}
	Even though none of these bounds is uniform with respect to the analytic conductor $\Cond(\pi)$ of $\pi$, they have various applications. In the same paper \cite{DFI02}, some properties of the class group of a quadratic field are derived from the above mentioned subconvexity \cite[Theorem 2.6 \& 2.7]{DFI02}. The improved version was an important ingredient of a cubic analogue of Duke's equidistribution result by Einsiedler, Lindenstrauss, Michel and Venkatesh \cite{ELMV11}.
	
	As for uniform bounds, much less is known. A uniform bound over general number fields is obtained in the celebrated paper by Michel and Venkatesh \cite{MV10} with the subconvex power saving unspecified. However, it is believed that the method of Michel and Venkatesh goes beyond the method of Duke, Friedlander and Iwaniec. We re-confirm this opinion and make their result effective in this paper. Our main tool is a further development \& adaptation to the triple product case of an improvement of the theory of regularized integrals due to Zagier \cite{Za82}, developed in our previous paper \cite[\S 2]{Wu9}. Precisely, we shall make the following assumption:
\begin{itemize}
	\item For $\pi'$ cuspidal representation of $\GL_2(\A)$ with trivial central character, spherical at all infinite places and Hecke character $\chi$ such that $\pi_{\fin}', \chi_{\fin}$ have disjoint ramification, assume
\begin{equation}
	L(1/2, \pi' \otimes \chi) \ll_{\F,\epsilon} (\Cond(\pi_{\fin}') \Cond(\chi))^{\epsilon} \Cond(\pi_{\infty}')^B \Cond(\pi_{\fin}')^A \Cond(\chi)^{\frac{1}{2} - \delta'}
\label{TwistSubAssump}
\end{equation}
for some constants $A,B > 0, 0 < \delta' < 1/2$.
\end{itemize}
	Let $\theta$ be any constant towards the Ramanujan-Petersson conjecture. Note that $\theta \leq 7/64$ is admissible by \cite{KS02, BB11}. The main result of this paper is:
\begin{theorem}
	Assuming $\delta' \leq (1-2\theta)/8$ and $A \geq 1/4$ in the assumption (\ref{TwistSubAssump}), we have for any $\epsilon > 0$
	$$ \extnorm{ L(\frac{1}{2}, \pi) } \ll_{\F, \epsilon} \Cond(\pi)^{\frac{1}{4}+\epsilon} \left( \frac{\Cond(\pi)}{\Cond(\omega)} \right)^{-\frac{1-2\theta}{40+32A}} \Cond(\omega)^{-\frac{\delta'}{20+16A}}. $$
\label{MainThm}
\end{theorem}
\begin{remark}
	The above bound separates $\Cond(\pi) / \Cond(\omega)$ from the problematic part $\Cond(\omega)$. In fact, in some applications, one does not need to vary $\omega$ but does need uniform bounds.
\end{remark}
\begin{remark}
	The assumption (\ref{TwistSubAssump}) should not be regarded as a \emph{condition}, since an effective value of $A$ is obtained in \cite[Theorem 2.1]{Wu3}. In particular, it implies that $\delta' = (1-2\theta)/8, A=5/4$ is admissible (note that ``$\pi_{\fin}', \chi_{\fin}$ have disjoint ramification'' implies $\Cond_{\fin}(\pi',\chi) = \Cond_{\fin}[\pi',\chi] = 1$). We record the numerical subconvex saving for these values:
	$$ \frac{1-2\theta}{40+32A} = \frac{1-2\theta}{80} > \frac{1}{128}, \quad \frac{\delta'}{20+16A} = \frac{1-2\theta}{320} > \frac{1}{1889}. $$
	It should even be possible to improve to $A=3/4$ once the relevant sup-norm result becomes available, see the discussion in \cite[\S 1.3]{Wu3}.
\end{remark}
\begin{remark}
	We did not specify an admissible value of $B$ in the previous work \cite{Wu3}, because it does not enter into our main bound in Theorem \ref{MainThm}. Over $\Q$, $\delta' = 1/8, A=1/2$ and $B=7/4$ is admissible by \cite[Theorem 2]{BH08}.
\end{remark}

\noindent Some immediate consequences are as follows.
\begin{corollary}
	If $\F = \ag{Q}$, then we have for any $\epsilon > 0$
	$$ \extnorm{ L(\frac{1}{2}, \pi) } \ll_{\epsilon} \Cond(\pi)^{\frac{1}{4}+\epsilon} \left( \frac{\Cond(\pi)}{\Cond(\omega)} \right)^{-\frac{1-2\theta}{56}} \Cond(\omega)^{-\frac{1}{224}}. $$
\end{corollary}
\begin{proof}
	Over $\ag{Q}$, the assumption (\ref{TwistSubAssump}) together with $\delta' = 1/8, A = 1/2$ is admissible by \cite[Theorem 2]{BH08}.
\end{proof}
\begin{remark}
	Blomer and Khan's uniform bound \cite{BK18} seems to give better bound in some aspects. In any case, our bound is valid over any number field.
\end{remark}

\begin{corollary}
	If the central character $\omega$ of $\pi$ is fixed, then we have for any $\epsilon > 0$
	$$ \extnorm{ L(\frac{1}{2}, \pi) } \ll_{\F,\epsilon} \Cond(\pi)^{\frac{1}{4}-\frac{1-2\theta}{48}+\epsilon}. $$
\label{FixCentralChar}
\end{corollary}
\begin{proof}
	The convex bound ($A=1/4, \delta'=0$) is in any case valid for (\ref{TwistSubAssump}).
\end{proof}

	\subsection{Geometric Intuition of the Method: Recall and Adaptation}
	\label{Heuristic}
	
	We recall the geometric intuition of the method, which imitates the description given just after \cite[Proposition 4.1]{Ve10}. We adapt it using our extension of regularized integral \cite[\S 2]{Wu9}.
	
	For simplicity, we assume $\F=\Q$, the central character $\omega$ of $\pi$ remains trivial, $\pi_{\infty}$ remains spherical and the usual conductor $\Cond(\pi_{\fin}) = p$ is a large varying prime. Recall the standard notations
	$$ \Gamma_0(p) := \left\{ \begin{pmatrix} a & b \\ c & d \end{pmatrix} \in \GL_2(\Z) \ \middle| \ p \mid c \right\}, \quad Y(p) := \Gamma_0(p) \backslash \PGL_2(\R) / \PSO_2(\R) = \Gamma_0(p) \backslash \ag{H}. $$
	A ($\intL^2$-normalized) new form in $\pi$ can be regarded as a function $\varphi$ on $Y(p)$. Let
	$$ \eis(s,z) := \sideset{}{_{\gamma \in \Gamma_0(1)_{\infty} \backslash \Gamma_0(1)}} \sum \Im(\gamma.z)^s, \quad \eis^*(s,z) := \Lambda(2s) \eis(s,z) $$
	be the standard spherical analytic Eisenstein series and its completion, where $\Lambda(s)$ is the complete Riemann zeta-function. The following integral represents $L(1/2,\pi)^2$
	$$ I(\varphi,p) := \int_{Y(p)} \varphi(z) \eis^*(1/2,z) \eis^*(1/2,pz) d\mu(z), \quad d\mu(z) := \frac{dxdy}{y^2}, z=x+iy $$
	It turns out that the product of the local terms of $I(\varphi,p)$ compensate the convex bound. It suffices to bound $I(\varphi,p)$. We regard $Y(p)$ as the graph of the $p$-th Hecke correspondence via
	$$ Y(p) \to Y(1) \times Y(1), \quad z \mapsto (z, pz), $$
	thus the function
	$$ \phi_p(z) := \eis^*(1/2,z) \eis^*(1/2,pz) $$
	can be regarded as the restriction to $Y(p)$ of the fixed function on $Y(1) \times Y(1)$
	$$ \phi(z_1,z_2) := \eis^*(1/2,z_1) \eis^*(1/2,z_2). $$
	$I(\varphi,p)$ is thus expected to be bounded as
	$$ \extnorm{I(\varphi,p)}^2 \leq \int_{Y(p)} \norm[\varphi(z)]^2 d\mu(z) \cdot \int_{Y(p)} \norm[\phi_p(z)]^2 d\mu(z) \to \int_{Y(1) \times Y(1)} \norm[\phi(z_1,z_2)]^2 d\mu(z_1) d\mu(z_2), $$
	and one shall apply the method of amplification to deal with the non-decreasing limit. This argument has a technical issue, namely $\phi_p(z)$ is not $\intL^2$-integrable.
	
	However, $\phi_p$ is \emph{finitely regularizable} in the sense of \cite[Definition 2.14]{Wu9}. In order to not let the complication of multiple cusps of $Y(p)$ obscure the idea, we pretend $p=1$. The skeptical reader is invited to gain the necessary information of rigorous computation from \cite[\S 6]{Wu2}. We thus propose to regularize
	$$ I(\varphi,1) = \int_{Y(1)} \varphi(z) \phi_1(z) d\mu(z), \quad \phi_1(z) := \eis^*(1/2,z)^2. $$
	The \emph{essential constant term} \cite[Definition 2.14]{Wu9} of $\phi_1(z)$ is equal to
	$$ \phi_{1,\gp{N}}^*(z) = \eisCst^*(1/2,z)^2 = 4\gamma^2 y + 4\gamma \Lambda^* y \log y + (\Lambda^*)^2 y (\log y)^2, $$
	where the constants $\gamma, \Lambda^*$ appear in the Laurent expansion
	$$ \Lambda(s) = \frac{\Lambda^*}{s-1} + \gamma + O((s-1)). $$
	We thus need to take the \emph{$\intL^2$-residue} \cite[Definition 2.20]{Wu9} of $\phi_1$ as
	$$ \Reis_1 = \Reis(\phi_1) = 4\gamma^2 \eis^{\reg}(1,z) + 4\gamma \Lambda^* \eis^{\reg, (1)}(1,z) + (\Lambda^*)^2 \eis^{\reg, (2)}(1,z), $$
	where the \emph{regularizing Eisenstein series} and its derivatives are defined as
	$$ \eis^{\reg}(s,z) := \eis(s,z) - \frac{3}{\pi} \frac{1}{s-1}, \quad \eis^{\reg,(n)}(1,z) := \left. \frac{d^n}{ds^n} \right|_{s=1} \eis^{\reg}(s,z). $$
	Choose a small prime $p_0 \ll \log p$, and let $T_0(1)$ be the level one normalized Hecke operator in Definition \ref{NormHeckeOp}. We have
	$$ T_0(1) \eis(s,z) = \lambda_0(s) \eis(s,z), \quad \lambda_0(s) = \frac{p_0^{s-1/2} + p_0^{1/2-s}}{p_0^{1/2} + p_0^{-1/2}}. $$
	Writing $\lambda_0^{(k)}(s)$ for the $k$-th derivative of $\lambda_0(s)$ and taking $n$-th derivative on both sides, we get
	$$ T_0(1) \eis^{\reg,(n)}(s,z) = \lambda_0(s) \eis^{\reg,(n)}(s,z) + \sideset{}{_{k=1}^n} \sum \binom{n}{k} \lambda_0^{(k)}(s) \eis^{\reg,(n-k)}(s,z) + \frac{3}{\pi} \frac{d^n}{ds^n} \left( \frac{\lambda_0(s)-1}{s-1} \right). $$
	Thus in the space spanned by $1, \eis^{\reg}(1,z), \eis^{\reg,(1)}(1,z), \dots, \eis^{\reg,(n)}(1,z)$, the operator $T_0(1)$ corresponds to a unipotent matrix with diagonal entries constant equal to $\lambda_0(1)=1$. In particular, we deduce\footnote{In a simpler way, we have $(T_0(1)-1)^4 \phi_{1,\gp{N}}^* = 0$, which immediately implies the rapid decay of $(T_0(1)-1)^4 \phi_1$.}
	$$ (T_0(1)-1)^{n+2} \eis^{\reg,(n)}(1,z) = 0, \quad \Rightarrow \quad (T_0(1)-1)^4 \Reis_1 = 0. $$
	It follows that
\begin{align*}
	\extnorm{I(\varphi,1)} &= \frac{1}{\norm[\lambda_{\varphi}(p_0)-1]^4} \extnorm{ \int_{Y(1)} (T_0(1)-1)^4\varphi(z) \cdot \phi_1(z) d\mu(z) } \\
	&= \frac{1}{\norm[\lambda_{\varphi}(p_0)-1]^4} \extnorm{ \int_{Y(1)} \varphi(z) \cdot (T_0(1)-1)^4\phi_1(z) d\mu(z) } \\
	&\leq \frac{1}{\norm[\lambda_{\varphi}(p_0)-1]^4} \left( \int_{Y(1)} \norm[\varphi(z)]^2 d\mu(z) \right)^{\frac{1}{2}} \cdot \left( \int_{Y(1)} \extnorm{(T_0(1)-1)^4\phi_1(z)}^2 d\mu(z) \right)^{\frac{1}{2}},
\end{align*}
	where $\lambda_{\varphi}(p_0)$ is the eigenvalue of $\varphi$ for $T_0(1)$. Since we have a non-trivial estimate of the constant $\theta \leq 7/64$ towards the Ramanujan-Petersson conjecture \cite{KS02, BB11}, $\norm[\lambda_{\varphi}(p_0)-1]^4$ is bounded from above and below by some absolute constants. The function
	$$ (T_0(1)-1)^4\phi_1(z) $$
	is now of \emph{rapid decay}, hence \emph{a fortiori} $\intL^2$-integrable. We then apply the Plancherel formula for the $\intL^2$-norm of the above function (with amplification) to get a non-trivial bound of $I(\varphi,1)$.

	\subsection{Organization of the Paper}
	
	As \cite{Wu3}, instead of a linear exhibition according to the logical order, we decide to regroup the ingredients according to their natures. Each ``proof'' of a global result in the proof of Theorem \ref{MainThm} serves as a \emph{pointer} to the relevant global or local results at a more fundamental level. In fact, the number of period formulas contained in the current paper is much more than those in (even the sum of) our previous works \cite{Wu14, Wu2, Wu3}. The transitions between local and global computations occur so often that it is too difficult to write down the argument in a linear logical way. We can only encourage the reader who really wants to understand every detail of the proof to linearise the argument by him/herself. Moreover, the current regroupment of arguments has the advantage of facilitating the possible future improvements if one seeks a better test function in our method.
	
	Precisely, we will fix the notations \& conventions and set up the precise measure/operator of regularization and amplification in \S 2.1. We then recall our extension of the theory of regularized integrals as well as its first development to the triple product case in \S 2.2. Since a big number of different triple product periods come into play in this paper, we standardize their Euler product decompositions in \S 2.3. After these preparations, we give a formal proof of the main result in \S 2.4, reducing/pointing the task to the relevant local and global estimations scattered in \S 3 and \S 4. This part is the adelization of the description given above in \S 1.2 in the general case and makes that subsection rigorous.
	
	In \S 3, we recollect all the local estimations. They serve either directly for the ``compensation of convex bound'' in \S 2.4, or for the estimation of global periods in \S 4. Along the way, we also specify the test functions via their local data in the Kirillov model.
	
	\S 4 contains all the relevant global estimations. Note that many proofs given there are again \emph{pointers} to the local estimations given in \S 3.
	
	In \S 5, we give some technical complements, which seem to be useful for the analytic theory of automorphic representations in general.
	
	Once again, this paper is NOT organized linearly. For the first reading, we highly recommend the following order of ``linearisation'': 
\begin{itemize}
	\item[(1)] \S 2.3 with ``return jumps'' to \S 3 indicated by pointers;
	\item[(2)] \S 4 with ``return jumps'' to \S 3 indicated by pointers.
\end{itemize}
	\S 2.1 and \S 2.4 should be consulted constantly whenever a notation or convention is not clear. A linear reading of \S 3 would make sense only for a second reading when the reader gets sufficiently familiar with the global steps.
	
	In \S 6, we fill a gap in our previous work by applying the proof of our main theorem without amplification, yielding a version of the Lindel\"of hypothesis for the fourth moment of $L$-functions for $\GL_2$. This part can be regarded as an introduction to our main theorem. It would be a good idea to read it before embarking the main body of the paper.

\section{Preliminaries and First Reductions}

	\subsection{Notations and Conventions}
	
		\subsubsection{Complex Analysis}
		
	If $f$ is a meromorphic function around $s=s_0$, we introduce the coefficients into its Laurent expansion
	$$ f(s) = \sideset{}{_{-\infty < k < 0}} \sum \frac{f^{(k)}(s_0)}{(-k)!} (s-s_0)^k + \sideset{}{_{k \geq 0}} \sum \frac{f^{(k)}(s_0)}{k!} (s-s_0)^k. $$
	The terms for $k<0$ form the \emph{principal part} of $f$ at $s_0$. The difference of $f$ and its principal part is denoted by $f^{\hol}$. This notation applies also to functions taking value in some Banach space/space of operators.
		
		\subsubsection{Number Theory}
		
	Throughout the paper, $\F$ is a (fixed) number field with ring of integers $\vo$ and degree $r=[\F : \ag{Q}] = r_1 + 2 r_2$, where $r_1$ resp. $r_2$ is the number of real resp. complex places. $V_{\F}$ denotes the set of places of $\F$ and for any $v \in V_{\F}$, $\F_v$ is the completion of $\F$ with respect to the absolute value $\norm_v$ corresponding to $v$. $\ag{A} = \ag{A}_{\F}$ is the ring of adeles of $\F$, while $\ag{A}^{\times}$ denotes the group of ideles. We fix a section $s_{\F}$ of the adelic norm map $\norm_{\ag{A}}: \ag{A}^{\times} \to \ag{R}_+$, hence identify $\ag{A}^{\times}$ with $\ag{R}_+ \times \ag{A}^{(1)}$, where $\ag{A}^{(1)}$ is the kernel of the adelic norm map, i.e., the subgroup of ideles with norm $1$. For example, we can take
	$$ s_{\F}: \ag{R}_+ \to \ag{A}^{\times}, \quad t \mapsto (\underbrace{t^{1/r},\dots,t^{1/r}}_{r_1 \text{ real places}}, \underbrace{t^{1/r}, \dots, t^{1/r}}_{r_2 \text{ complex places}}, 1, \dots). $$
	
	We put the standard Tamagawa measure $dx = \sideset{}{_v} \prod dx_v$ on $\ag{A}$ resp. $d^{\times}x = \sideset{}{_v} \prod d^{\times}x_v$ on $\ag{A}^{\times}$. We recall their constructions. Let $\Tr = \Tr_{\ag{Q}}^{\F}$ be the trace map, extended to $\ag{A} \to \ag{A}_{\ag{Q}}$. Let $\psi_{\ag{Q}}$ be the additive character of $\ag{A}_{\ag{Q}}$ trivial on $\ag{Q}$, restricting to the infinite place as
	$$ \ag{Q}_{\infty} = \ag{R} \to \ag{C}^{(1)}, x \mapsto e^{2\pi i x}. $$
	We put $\psi = \psi_{\ag{Q}} \circ \Tr$, which decomposes as $\psi(x) = \sideset{}{_v} \prod \psi_v(x_v)$ for $x=(x_v)_v \in \ag{A}$.  The additive Haar measure $dx_v$ on $\F_v$ is self-dual with respect to $\psi_v$. Precisely, if $\F_v = \ag{R}$, then $dx_v$ is the usual Lebesgue measure on $\ag{R}$; if $\F_v = \ag{C}$, then $dx_v$ is twice the usual Lebesgue measure on $\ag{C} \simeq \ag{R}^2$; if $v = \vp < \infty$ such that $\vo_{\vp}$ is the valuation ring of $\F_{\vp}$, then $dx_{\vp}$ gives $\vo_{\vp}$ the mass $\Dis_{\vp}^{-1/2}$, where $\Dis_{\vp} = \Dis(\F_{\vp})$ is the local component at $\vp$ of the discriminant $\Dis(\F)$ of $\F/\ag{Q}$ such that $\Dis(\F) = \sideset{}{_{\vp < \infty}} \prod \Dis_{\vp}$. We equip $\F \backslash \ag{A}$ with the quotient measure by the discrete measure on $\F$. The total mass of $\F \backslash \A$ is $1$ \cite[Ch.\Rmnum{14} Prop.7]{Lan03}. Recall the local zeta-functions: if $\F_v = \ag{R}$, then $\zeta_v(s) = \Gamma_{\ag{R}}(s) = \pi^{-s/2} \Gamma(s/2)$; if $\F_v = \ag{C}$, then $\zeta_v(s) = \Gamma_{\ag{C}}(s) = (2\pi)^{1-s} \Gamma(s)$; if $v=\vp < \infty$ then $\zeta_{\vp}(s) = (1-q_{\vp}^{-s})^{-1}$, where $q_{\vp} := \Nr(\vp)$ is the cardinality of $\vo/\vp$. We then define
	$$ d^{\times} x_v := \zeta_v(1) \frac{dx_v}{\norm[x]_v}. $$
	In particular, $\Vol(\vo_{\vp}^{\times}, d^{\times}x_{\vp}) = \Vol(\vo_{\vp}, dx_{\vp})$ for $\vp < \infty$. Put the measure $d^{\times} t = dt/\norm[t]$ on $\ag{R}_+$, where $dt$ is the usual Lebesgue measure on $\ag{R}$ restricted to $\ag{R}_+$. We equip $\ag{A}^{(1)} \simeq \ag{R}_+ \backslash \ag{A}^{\times}$ with the quotient. Consequently, $\F^{\times} \backslash \ag{A}^{(1)}$ admits the total mass \cite[Ch.\Rmnum{14} Prop.13]{Lan03}
	$$ \Vol(\F^{\times} \backslash \ag{A}^{(1)}) = \zeta_{\F}^* = \zeta_{\F}^{(-1)}(1) = \Res_{s=1} \zeta_{\F}(s), $$
	where $\zeta_{\F}(s) := \sideset{}{_{\vp < \infty}} \prod \zeta_{\vp}(s)$ is the Dedekind zeta-function of $\F$.
	
	For any automorphic representation $\pi$, $L(s,\pi)$ denotes the usual $L$-function of $\pi$ without components at infinity, $\Lambda(s,\pi)$ denotes its completion with components at infinity. The local component $L_{\vp}(s,\pi_{\vp})$ at a finite place $\vp$ takes $\Dis_{\vp}$ into account, so that $L_{\vp}(s,\mathbbm{1}_{\vp}) = \Dis_{\vp}^{s/2} \zeta_{\vp}(s)$, where $\mathbbm{1}$ is the trivial representation. We define the complete Dedekind zeta-function to be
	$$ \Lambda_{\F}(s) := \Lambda(s, \mathbbm{1}) = \Dis_{\F}^{s/2} \cdot \left( \sideset{}{_{v \mid \infty}} \prod \zeta_v(s) \right) \cdot \zeta_{\F}(s), $$
	so that it satisfies the functional equation $\Lambda_{\F}(s) = \Lambda_{\F}(1-s)$.
	
\begin{remark}
	Although the star $*$ is ambiguous (for example the complete Eisenstein series below also uses it), it is conventional in the literature. Hence we keep it for $\zeta_{\F}^*, \Lambda_{\F}^*$.
\end{remark}
		
		\subsubsection{Automorphic Representation, Spectral Theory}
		
	We will work on algebraic groups $\GL_2$ and $\PGL_2$ over $\F$, the latter being the quotient of $\GL_2$ by its center over $\ag{A}$ or $\F_v$ in the category of abstract groups. We put the \emph{hyperbolic measure} instead of the Tamagawa measure on $\GL_2$. We recall its definition. We pick the standard maximal compact subgroup $\gp{K} = \sideset{}{_v} \prod \gp{K}_v$ of $\GL_2(\ag{A})$ by defining
	$$ \gp{K}_v = \left\{ \begin{matrix} \SO_2(\ag{R}) & \text{if } \F_v = \ag{R} \\ \SU_2(\ag{C}) & \text{if } \F_v = \ag{C} \\ \GL_2(\vo_{\vp}) & \text{if } v = \vp < \infty \end{matrix} \right. , $$
and equip it with the Haar probability measure $d\kappa_v$. We define the following one-parameter algebraic subgroups of $\GL_2(\F_v)$
	$$ \gp{Z}_v = \gp{Z}(\F_v) = \left\{ z(u) := \begin{pmatrix} u & 0 \\ 0 & u \end{pmatrix} \ \middle| \ u \in \F_v^{\times} \right\}, $$
	$$ \gp{N}_v = \gp{N}(\F_v) = \left\{ n(x) := \begin{pmatrix} 1 & x \\ 0 & 1 \end{pmatrix} \ \middle| \ x \in \F_v \right\}, $$
	$$ \gp{A}_v = \gp{A}(\F_v) = \left\{ a(y) := \begin{pmatrix} y & 0 \\ 0 & 1 \end{pmatrix} \ \middle| \ y \in \F_v^{\times} \right\}, $$
and equip them with the Haar measures on $\F_v^{\times}, \F_v, \F_v^{\times}$ respectively. The hyperbolic Haar measure $dg_v$ on $\GL_2(\F_v)$ is the push-forward of the product measure $d^{\times}u \cdot dx \cdot d^{\times}y / \norm[y]_v \cdot d\kappa_v$ under the Iwasawa decomposition map
	$$ \gp{Z}_v \times \gp{N}_v \times \gp{A}_v \times \gp{K}_v \to \GL_2(\F_v), \quad (z(u), n(x), a(y), \kappa) \mapsto z(u) n(x) a(y) \kappa. $$
	Similarly, the hyperbolic Haar measure $d\bar{g}_v$ on $\PGL_2(\F_v)$ is the push-forward of the product measure $dx \cdot d^{\times}y / \norm[y]_v \cdot d\kappa_v$ under the composition map
	$$ \gp{N}_v \times \gp{A}_v \times \gp{K}_v \to \GL_2(\F_v) \to \PGL_2(\F_v), \quad (n(x), a(y), \kappa) \mapsto [n(x) a(y) \kappa]. $$
	We put the product measure $d\bar{g} := \sideset{}{_v} \prod d\bar{g}_v$ on $\PGL_2(\ag{A})$. Define and equip the quotient space
	$$ [\PGL_2] := \gp{Z}(\ag{A}) \GL_2(\F) \backslash \GL_2(\ag{A}) = \PGL_2(\F) \backslash \PGL_2(\ag{A}) $$
with the quotient measure by the discrete measure on $\PGL_2(\F)$.

	At a finite place $\vp$, we have the following congruence subgroups of $\gp{K}_{\vp}$ for $n \in \Z_{\geq 0}$
	$$ \gp{K}_0[\vp^n] := \left\{ \begin{pmatrix} a & b \\ c & d \end{pmatrix} \in \gp{K}_{\vp} \ \middle| \ c \equiv 0 \pmod{\vp^n} \right\}, $$
	$$ \gp{K}_1[\vp^n] := \left\{ \begin{pmatrix} a & b \\ c & d \end{pmatrix} \in \gp{K}_{\vp} \ \middle| \ d-1, c \equiv 0 \pmod{\vp^n} \right\}. $$
If $\pi_{\vp}$ is an admissible representation of $\GL_2(\F_{\vp})$, its \emph{logarithmic conductor} $\cond(\pi_{\vp})$ is the smallest integer $n$ so that $\pi_{\vp}$ admits a non-zero vector invariant by $\gp{K}_1[\vp^n]$; while its \emph{conductor} is $\Cond(\pi_{\vp}) := q_{\vp}^{\cond(\pi_{\vp})}$. At $v \mid \infty$, if the local $L$-function has the form $L_v(s,\pi_v) = \sideset{}{_{i=1}^d} \prod \Gamma_{\R}(s+\mu_i)$, then we define $\Cond(\pi_v) = \sideset{}{_{i=1}^d} \prod (1+\norm[\mu_i])$. Note that this definition makes sense for automorphic representations of $\GL_d$ for any integer $d \geq 1$, in particular for Hecke characters $\chi$. The global analytic conductor is defined to be 
	$$ \Cond(\pi) := \sideset{}{_v} \prod \Cond(\pi_v), \quad \text{resp.} \quad \Cond(\chi) := \sideset{}{_v} \prod \Cond(\chi_v). $$
	
	The product $\gp{B} := \gp{Z} \gp{N} \gp{A}$ is a Borel subgroup of $\GL_2$. We have the \emph{height function} $\Ht$ resp. $\Ht_v$ on $\GL_2(\ag{A})$ resp. $\GL_2(\F_v)$ associated with $\gp{B}$ defined by
	$$ \Ht_v \left( \begin{pmatrix} t_1 & x \\ 0 & t_2 \end{pmatrix} \kappa \right) := \extnorm{\frac{t_1}{t_2}}_v, \quad \forall t_1, t_2 \in \F_v^{\times}, x \in \F_v; $$
	$$ \Ht(g) = \sideset{}{_v} \prod \Ht_v(g_v), \quad g = (g_v)_v \in \GL_2(\ag{A}). $$
	
	We fix a Hecke character $\omega$ of $\F^{\times} \backslash \ag{A}^{\times}$, and identify it with a unitary character of $\gp{Z}(\ag{A})$ in the obvious way. Let $\intL^2(\GL_2, \omega)$ denote the (Hilbert) space of Borel measurable functions $\varphi$ satisfying
	$$ \left\{ \begin{matrix} \varphi(z \gamma g) = \omega(z)\varphi(g), \quad \forall \gamma \in \GL_2(\F), z \in \gp{Z}(\ag{A}), g \in \GL_2(\ag{A}), \\ \int_{[\PGL_2]} \norm[\varphi(g)]^2 d\bar{g} < \infty. \end{matrix} \right. $$
	Let $\intL_0^2(\GL_2, \omega)$ denote the subspace of $\varphi \in \intL^2(\GL_2, \omega)$ such that its \emph{constant term}
	$$ \varphi_{\gp{N}}(g) := \int_{\F \backslash \ag{A}} \varphi(n(x)g) dx = 0, \quad \text{a.e. } \bar{g} \in [\PGL_2]. $$
	Note that $\intL_0^2(\GL_2, \omega)$ is a closed subspace of $\intL^2(\GL_2, \omega)$. Obviously $\GL_2(\ag{A})$ acts on $\intL_0^2(\GL_2, \omega)$ resp. $\intL^2(\GL_2, \ag{A})$, giving rise to a unitary representation $\rpR_0$ resp. $\rpR$. We also write $\rpR_{\omega}$ instead of $\rpR$ if the central character $\omega$ is to be emphasized. $\rpR_0$ decomposes as a direct sum of unitary irreducible representations of $\GL_2(\A)$. The ortho-complement of $\rpR_0$ in $\rpR$ is the orthogonal sum of the one-dimensional spaces
	$$ \ag{C} \left( \xi \circ \det \right) : \quad \xi \text{ Hecke character such that } \xi^2 = \omega $$
and $\rpR_c$, which can be identified as a direct integral representation over the unitary dual of $\F^{\times} \backslash \ag{A}^{\times} \simeq \ag{R}_+ \times (\F^{\times} \backslash \ag{A}^{(1)} )$. Precisely, let $s \in \ag{C}$ and $\chi$ be a unitary character of $\F^{\times} \backslash \ag{A}^{(1)}$ regarded as a unitary character of $\F^{\times} \backslash \ag{A}^{\times}$ via trivial extension. We associate a representation $\pi_{\chi}(s)$ of $\GL_2(\ag{A})$ on the following Hilbert space $V_{\chi}(s)$ of functions via right regular translation
	$$ \left\{ \begin{matrix} f\left( \begin{pmatrix} t_1 & x \\ 0 & t_2 \end{pmatrix} g \right) = \chi(t_1) \omega\chi^{-1}(t_2) \extnorm{\frac{t_1}{t_2}}_{\ag{A}}^{\frac{1}{2}+s} f(g), \quad \forall t_1,t_2 \in \ag{A}^{\times}, x \in \ag{A}, g \in \GL_2(\ag{A}) ; \\ \int_{\gp{K}} \norm[f(\kappa)]^2 d\kappa < \infty . \end{matrix} \right. $$
	The representation $\pi_{\chi}(s)$ is unitary if $s \in i\R$. To any $f \in V_{\chi}(0)$, one can associate a \emph{flat section} $f_s \in V_{\chi}(s)$ determined by $f_s \mid_{\gp{K}} = f \mid_{\gp{K}}$. One constructs an intertwining operator from $V_{\chi}(s)$ to the ortho-complement of $\rpR_0$ in $\rpR$, called the Eisenstein series
	$$ \eis(s,f)(g) := \sideset{}{_{\gamma \in \gp{B}(\F) \backslash \GL_2(\F)}} \sum f_s(\gamma g) $$
convergent for $\Re s > 1/2$ and admitting a meromorphic continuation to $s \in \ag{C}$. The image of $V_{\chi}(i\tau)$ with $\tau \in \R$ under the construction of Eisenstein series form an direct integral decomposition of $\rpR_c$ with Plancherel measure $d\tau/4\pi$. We also denote
	$$ \eis^*(s,f) = \Lambda(1+2s, \omega^{-1}\chi^2) \eis(s,f), \quad \eis^{\sharp}(s,f) = L(1+2s, \omega^{-1}\chi^2) \eis(s,f). $$

		\subsubsection{Some Specialties}
		
	Denote by $\pi$ a varying cuspidal representation of $\GL_2(\A)$, whose central character $\omega$ is also varying. Whittaker models/functions are taken with respect to the fixed standard additive character $\psi$ or $\psi_v$.

	Due to the varying central character $\omega$, we need to choose the measure formed by Hecke operators which regularizes the product of Eisenstein series in a way different from the one described in \S \ref{Heuristic}. Precisely, choose a finite place $\vp_0$ at which $\F, \pi$ are unramified. In particular, $\omega_{\vp_0}$ is unramified. Write 
	$$ \varpi_0 := \varpi_{\vp_0}, \quad q_0 := q_{\vp_0}, \quad \gp{K}_0 := \gp{K}_{\vp_0}, \quad \omega_0 := \omega_{\vp_0} $$ 
	for simplicity of notations and denote 
	$$ \alpha_0 := \omega_{0}(\varpi_0). $$
	We can assume
\begin{equation}
	q_0 \ll \left( \log \Cond(\pi_{\fin}) \right)^2.
\label{Smallp0}
\end{equation}
\begin{definition}
	Let $\mathcal{A}(\GL_2, \omega^{-1})$ be the space of smooth automorphic functions with central character $\omega^{-1}$ of moderate growth, i.e. smooth functions $f: \GL_2(\A) \to \C$ satisfying
\begin{itemize}
	\item[(1)] For any $\gamma \in \GL_2(\F), z \in \A^{\times}$,
	$$ f\left( \gamma g \begin{pmatrix} z & 0 \\ 0 & z \end{pmatrix} \right) = \omega(z)^{-1} f(g). $$
	\item[(2)] There is $A > 0$ so that for $g$ lying in a/any Siegel domain
	$$ \extnorm{f(g)} \ll \Ht(g)^A. $$
\end{itemize}
	There is a natural extension $\rpR_{\omega^{-1}}^{\mathcal{A}}$ of $\rpR_{\omega^{-1}}$, from $\intL^2(\GL_2, \omega^{-1})$ to $\mathcal{A}(\GL_2, \omega^{-1})$, given by the same formula
	$$ \rpR_{\omega^{-1}}^{\mathcal{A}}(g_0)f(g) := f(gg_0). $$
	Define the Hecke operators for $n \in \ag{N}$
	$$ T_0(n) := \int_{\gp{K}_0^2} \rpR_{\omega^{-1}}^{\mathcal{A}}(\kappa_1 a(\varpi_0^n) \kappa_2) d\kappa_1 d\kappa_2, $$
	where the measure on $\gp{K}_0$ is normalized so that the total mass is equal to $1$.
	$T_0(1)$ acts on the spherical vector in $\pi(\norm_{\vp_0}^{1/2+s},\omega_0^{-1} \norm_{\vp_0}^{-(1/2+s)})$ resp. $\pi(\omega_0^{-1}\norm_{\vp_0}^{1/2+s},\norm_{\vp_0}^{-(1/2+s)})$ as multiplication by
	$$ \lambda_0(s) := \frac{q_0^{-1/2}}{1+q_0^{-1}} \left( q_0^{-(1/2+s)} + \alpha_0^{-1} q_0^{1/2+s} \right), \quad \text{resp.} \quad \tilde{\lambda}_0(s) := \frac{q_0^{-1/2}}{1+q_0^{-1}} \left( \alpha_0^{-1} q_0^{-(1/2+s)} + q_0^{1/2+s} \right). $$
	Define the \emph{operator/measure of regularization}
	$$ \sigma_0 := (T_0(1) - \lambda_0(0))^2 (T_0(1) - \tilde{\lambda}_0(0))^2. $$
\label{NormHeckeOp}
\end{definition}

\begin{proposition}
	The above Hecke operators satisfy the following recurrence relation
	$$ \left( \sideset{}{_{n=0}^{\infty}} \sum T_0(n)X^n \right) \cdot \left( 1 - (1+q_0^{-1})T_0(1)X + \alpha_0^{-1} q_0^{-1} X^2 \right) = T_0(0) - q_0^{-1}T_0(1) X. $$
	Moreover, there exist constants $a_j^{(n)}(q_0) \in \C$ satisfying
	$$ \sideset{}{_{j=0}^{\lfloor n/2 \rfloor}} \sum \extnorm{a_j^{(n)}(q_0)} \leq 1, \quad a_0^{(n)}(q_0) = \left( \frac{q_0}{1+q_0} \right)^{n-1}, $$
	so that we have
	$$ T_0(1)^n = \sideset{}{_{j=0}^{\lfloor n/2 \rfloor}} \sum a_j^{(n)}(q_0) T_0(n-2j). $$
	In particular, we have
	$$ T_0(1)^2 = \frac{q_0}{q_0+1} T_0(2) + \frac{\alpha_0^{-1}}{q_0+1} T_0(0); $$
	$$ T_0(2)^2 = \frac{q_0}{q_0+1} T_0(4) + \frac{\alpha_0^{-1}}{q_0+1} T_0(2) + \left( \frac{\alpha_0^{-1}}{q_0+1} \right)^2 T_0(0). $$
\label{HeckeRel}
\end{proposition}
\begin{proof}
	For $n \geq 1$, we have the $\gp{K}_{\vp_0}$-coset decomposition
	$$ \gp{K}_0[\vp_0^n] \backslash \gp{K}_{\vp_0} = \sideset{}{_{\alpha \in \vo/\vp_0^n}} \bigsqcup \gp{K}_0[\vp_0^n] \begin{pmatrix} 1 & 0 \\ \alpha & 1 \end{pmatrix} \sqcup \sideset{}{_{\beta \in \vp_0/\vp_0^n}} \bigsqcup \gp{K}_0[\vp_0^n] \begin{pmatrix} 1 & 0 \\ \beta & 1 \end{pmatrix} \begin{pmatrix} 0 & 1 \\ 1 & 0 \end{pmatrix}. $$
	We have the following relations
	$$ a(\varpi_0^n) \begin{pmatrix} 1 & 0 \\ \alpha & 1 \end{pmatrix} a(\varpi_0) = \begin{pmatrix} \varpi_0^{n+1} & 0 \\ \alpha \varpi_0 & 1 \end{pmatrix} = a(\varpi^{n+1}) \begin{pmatrix} 1 & 0 \\ \alpha \varpi_0 & 1 \end{pmatrix} \in a(\varpi^{n+1}) \gp{K}_{\vp_0}, $$
	$$ a(\varpi_0^n) \begin{pmatrix} 1 & 0 \\ \beta & 1 \end{pmatrix} \begin{pmatrix} 0 & 1 \\ 1 & 0 \end{pmatrix} a(\varpi_0) = \begin{pmatrix} \varpi_0 & 0 \\ 0 & \varpi_0 \end{pmatrix} a(\varpi_0^{n-1}) \begin{pmatrix} 1 & 0 \\ \beta \varpi_0^{-1} & 1 \end{pmatrix} \in \begin{pmatrix} \varpi_0 & 0 \\ 0 & \varpi_0 \end{pmatrix} a(\varpi_0^{n-1}) \gp{K}_0. $$
	From these relations, we readily deduce
	$$ T_0(n) T_0(1) = \frac{q_0}{1+q_0} T_0(n+1) + \frac{\alpha_0^{-1}}{1+q_0}T_0(n-1), \quad \forall n \geq 1. $$
	Together with the obvious relation $T_0(0)T_0(1) = T_0(1)$, we prove the ``moreover'' part by induction. Multiplying the above equation by $X^n$ and summing over $n \geq 1$, we get the desired recurrence relation.
\end{proof}

\begin{proposition}
	Let $f_2 \in \mathcal{A}(\GL_2, \omega^{-1})$ and $f_1 \in \mathcal{A}(\GL_2, \omega^{-1})$ resp. $f_1 \in \mathcal{A}(\GL_2, \omega)$. Suppose $f_1$ is of rapid decay, i.e., we replace the condition (2) in Definition \ref{NormHeckeOp} by
\begin{itemize}
	\item[(2')] For any $A > 0$ and for $g$ lying in a/any Siegel domain, we have
	$$ \extnorm{f_1(g)} \ll_A \Ht(g)^{-A}. $$
\end{itemize}
	We define the adjoint resp. dual operator $T_0(n)^*$ resp. $T_0(n)^{\vee}$ of $T_0(n)$ by
	$$ \int_{[\PGL_2]} f_1(g) \overline{T_0(n)f_2(g)} dg = \int_{[\PGL_2]} T_0(n)^*f_1(g) \overline{f_2(g)} dg, $$
	$$ \text{resp.} \quad \int_{[\PGL_2]} f_1(g) T_0(n)f_2(g) dg = \int_{[\PGL_2]} T_0(n)^{\vee}f_1(g) f_2(g) dg. $$
	Then we have the relation
	$$ T_0(n)^* = T_0(n)^{\vee} = \alpha_0^{-n}T_0(n). $$
\label{DualHeckeOp}
\end{proposition}
\begin{remark}
	Consequently, we have the formulas for the dual resp. adjoint of $\sigma_0$
	$$ \sigma_0^* = (\alpha_0^{-1}T_0(1) - \overline{\lambda_0(0)})^2 (\alpha_0^{-1}T_0(1) - \overline{\tilde{\lambda}_0(0)})^2, $$
	$$ \text{resp.} \quad \sigma_0^{\vee} = (\alpha_0^{-1}T_0(1) - \lambda_0(0))^2 (\alpha_0^{-1}T_0(1) - \tilde{\lambda}_0(0))^2. $$
\end{remark}
\begin{proof}
	This is a consequence of the following simple calculation
\begin{align*}
	\int_{[\PGL_2]} f_1(g) \overline{T_0(n)f_2(g)} dg &= \int_{\gp{K}_0^2} \int_{[\PGL_2]} f_1(g) \overline{f_2(g\kappa_1 a(\varpi_0^n) \kappa_2)} dg d\kappa_1 d\kappa_2 \\
	&= \int_{\gp{K}_0^2} \int_{[\PGL_2]} f_1(g \kappa_2 a(\varpi_0^{-n}) \kappa_1^{-1}) \overline{f_2(g)} dg d\kappa_1 d\kappa_2 \\
	&= \int_{\gp{K}_0^2} \int_{[\PGL_2]} f_1(g \kappa_1 a(\varpi_0^{n}) \kappa_2 \begin{pmatrix} \varpi_0^{-n} & 0 \\ 0 & \varpi_0^{-n} \end{pmatrix}) \overline{f_2(g)} dg d\kappa_1 d\kappa_2 \\
	&= \int_{[\PGL_2]} \alpha_0^{-n} T_0(n)f_1(g) \overline{f_2(g)} dg.
\end{align*}
	The case of dual operator is similar.
\end{proof}

	Finally, we construct the amplifiers as follows. For some $K>0$ to be optimized later, let 
\begin{equation}
	S=S(K)=\{ \vp: \vp \neq \vp_0, K < \Nr(\vp) \leq 2K, \pi \text{ and } \F \text{ are unramified at } \vp \}, \quad S^*=S \cup \{ \vp_0 \}.
\label{AmpPlace}
\end{equation}
	For $\vp \in S$, write $q_{\vp} := \Nr(\vp), \alpha_{\vp} := \omega_{\vp}(\varpi_{\vp})$ and let $\lambda_{\pi}(\vp^n)$ be the eigenvalue of the $n$-th Hecke operator 
	$$ T(\vp^n) := \int_{\gp{K}_{\vp}^2} \rpR_{\omega}(\kappa_1 a(\varpi_{\vp}^{-n}) \kappa_2) d\kappa_1 d\kappa_2 $$ 
	on the spherical vector in $\pi_{\vp}$, and define an operator/measure (amplifier)
\begin{equation}
	\sigma = \sigma(S,\pi) = \sum_{\vp \in S} a_{\vp} T(\vp^{n_{\vp}}),
\label{Amplifier}
\end{equation}
where $n_{\vp} \in \{ 1,2 \}$ is chosen so that (recall $\lambda_{\pi}(\vp)^2 = \frac{q_{\vp}}{q_{\vp}+1} \lambda_{\pi}(\vp^2) + \frac{\alpha_{\vp}^{-1}}{q_{\vp}+1} $ from Proposition \ref{HeckeRel})
	$$ q_{\vp}^{n_{\vp}/2} \norm[\lambda_{\pi}(\vp^{n_{\vp}})] = \max( q_{\vp}^{1/2} \norm[\lambda_{\pi}(\vp)], q_{\vp} \norm[\lambda_{\pi}(\vp^2)]) \geq 1/2 $$
	$$ \Rightarrow \quad \sideset{}{_{\vp \in S}} \sum q_{\vp}^{n_{\vp}/2} \norm[\lambda_{\pi}(\vp^{n_{\vp}})] \gg \norm[S] \gg_{\epsilon} K^{1-\epsilon}, $$
and we take
	$$ a_{\vp} = q_{\vp}^{n_{\vp}/2} \overline{\sgn (\lambda_{\pi}(\vp^{n_{\vp}}))} = q_{\vp}^{n_{\vp}/2} \cdot \overline{\lambda_{\pi}(\vp^{n_{\vp}})} / \norm[\lambda_{\pi}(\vp^{n_{\vp}})]. $$

	\subsection{Extension of Zagier's Regularized Integral}
	
	\begin{definition}
	Let $\omega$ be a unitary character of $\F^{\times} \backslash \ag{A}^{\times}$. Let $\varphi$ be a continuous function on $\GL_2(\F) \backslash \GL_2(\ag{A})$ with central character $\omega$. We call $\varphi$ \emph{finitely regularizable} if there exist characters $\chi_i: \F^{\times} \R_+ \backslash \ag{A}^{\times} \to \ag{C}^{(1)}$, $\alpha_i \in \ag{C}, n_i \in \ag{N}$ and continuous functions $f_i \in \Ind_{\gp{B}(\ag{A}) \cap \gp{K}}^{\gp{K}} (\chi_i, \omega \chi_i^{-1})$, such that for any $M \gg 1$
\begin{equation} 
	\varphi(n(x)a(y)k) = \varphi_{\gp{N}}^*(n(x)a(y)k) + O(\norm[y]_{\ag{A}}^{-M}), \text{ as } \norm[y]_{\ag{A}} \to \infty, 
\label{DiffRD}
\end{equation}
	where we have written the \emph{essential constant term}
	$$ \varphi_{\gp{N}}^*(n(x)a(y)k)=\varphi_{\gp{N}}^*(a(y)k)=\sideset{}{_{i=1}^l} \sum \chi_i(y) \norm[y]_{\ag{A}}^{\frac{1}{2}+\alpha_i} \log^{n_i} \norm[y]_{\ag{A}} f_i(k). $$
	We call $\Ex(\varphi)=\{ \chi_i \norm^{\frac{1}{2}+\alpha_i}: 1 \leq i \leq l \}$ the \emph{exponent set} of $\varphi$, and define
	$$ \Ex^+(\varphi) = \{ \chi_i \norm^{\frac{1}{2}+\alpha_i} \in \Ex(\varphi): \Re \alpha_i \geq 0 \}; \quad \Ex^-(\varphi) = \{ \chi_i \norm^{\frac{1}{2}+\alpha_i} \in \Ex(\varphi): \Re \alpha_i < 0 \}. $$
	The space of finitely regularizable functions with central character $\omega$ is denoted by $\Aut^{\freg}(\GL_2,\omega)$.
\label{FinRegFuncDef}
\end{definition}

\begin{definition}
	In the case $\omega^{-1}\xi^2(t)=\norm[t]_{\ag{A}}^{i\mu}$ for some $\mu \in \ag{R}$, we introduce the \emph{regularizing Eisenstein series} for $f \in V_{\xi,\omega\xi^{-1}}$ and $s$ in a neighborhood of $(1-i\mu)/2$
	$$ \eis^{\reg}(s,f)(g) = \eis(s,f)(g) - \frac{\Lambda_{\F}(1-2s-i\mu)}{\Lambda_{\F}(1+2s+i\mu)} \int_{\gp{K}} f(\kappa) \xi^{-1}(\det \kappa) d\kappa \cdot \xi(\det g) \norm[\det g]_{\ag{A}}^{-\frac{i\mu}{2}}. $$
	It is holomorphic at $s=(1-i\mu)/2$.
\label{RegEisDef}
\end{definition}

\begin{definition}
	Let $\omega$ be a unitary character of $\F^{\times} \backslash \ag{A}^{\times}$. The $\intL^2$-\emph{residual space} of central character $\omega$, denoted by $\Reis(\GL_2,\omega)$, is the direct sum of the vector spaces $\Reis^+(\GL_2,\omega)$ resp. $\Reis^{\reg}(\GL_2,\omega)$ spanned by functions
	$$ \frac{\partial^n}{\partial s^n} \eis(s,f), \text{if } s \neq \frac{1-i\mu}{2} \quad \text{resp.} \quad \eis^{\reg,(n)}(\frac{1-i\mu}{2},f) := \frac{\partial^n}{\partial s^n} \eis^{\reg}(\frac{1-i\mu}{2},f) $$
	where $s \in \ag{C}, \Re s > 0$ and for some unitary character $\xi$ of $\F^{\times} \backslash \ag{A}^{\times}$, $f \in V_{\xi,\omega\xi^{-1}}^{\infty}$ and $\mu$ as above.
\label{L2Res}
\end{definition}

\begin{definition}
	Let $\varphi \in \Aut^{\freg}(\GL_2,\mathbbm{1})$. We define its regularized integral as
	$$ \int_{[\PGL_2]}^{\reg} \varphi(g) dg := \int_{[\PGL_2]} \left( \varphi(g) - \Reis(g) \right) dg, $$
	where, for definiteness, we take
\begin{equation}
	\Reis(\varphi) = \sum_{\substack{\Re \alpha_j > 0 \\ \alpha_j \neq \frac{1}{2}+i\mu_j}} \frac{\partial^{n_j}}{\partial s^{n_j}}\eis(\alpha_j, f_j) + \sum_{\substack{\Re \alpha_j > 0 \\ \alpha_j = \frac{1}{2}+i\mu_j}} \frac{\partial^{n_j}}{\partial s^{n_j}}\eis^{\reg}(\alpha_j,f_j).
\label{ReisDef}
\end{equation}
\label{RegDef}
\end{definition}

\noindent In the case $\omega = 1$ and for any $\varphi \in \Aut^{\freg}(\GL_2,1)$, the integral
	$$ R(s,\varphi) := \int_{\ag{A}^{\times} \times \gp{K}} (\varphi_{\gp{N}} - \varphi_{\gp{N}}^*)(a(y)\kappa) \norm[y]_{\ag{A}}^{s-1/2} d^{\times}y d\kappa $$
is convergent for any $\Re s \gg 1$, and has meromorphic continuation to all $s \in \C$. The main result of \cite[\S 2]{Wu9} is the following equality
	$$ \int_{[\PGL_2]}^{\reg} \varphi(g) dg = \frac{1}{\Vol([\PGL_2])} \left( \Res_{s=1/2} R(s,\varphi) + \sideset{}{_{\substack{ \alpha_i = 1 \\ n_i = 0 }}} \sum f_i(1) \right). $$
	
	In view of the inclusion (\cite[Remark 2.19]{Wu9})
	$$ \Aut^{\freg}(\GL_2,\omega_1) \cdot \Aut^{\freg}(\GL_2,\omega_2) \subset \Aut^{\freg}(\GL_2,\omega_1 \omega_2), $$
we can consider the following bilinear form. Let $\pi_j, j=1,2$ be two principal series representations with central character $\omega_j$ satisfying $\omega_1 \omega_2 = 1$. Let $V_j$ be the vector space of $\pi_j$ realized in the induced model from $\grB(\ag{A})$ with subspace of smooth vectors $V_j^{\infty}$. We then get a bilinear form
	$$ V_1^{\infty} \times V_2^{\infty} \to \ag{C}, \quad (f_1,f_2) \mapsto \int_{[\PGL_2]}^{\reg} \eis(f_1)(g) \eis(f_2)(g) dg, $$
where $\eis(f_j)$ should be suitably regularized if $\pi_j$ is at a position which creates a pole/zero for the relevant Eisenstein series. We succeeded in \cite[Theorem 3.5]{Wu9} to identify this bilinear form in the induced model. In order to present the result, we need to introduce some extra notations. Precisely, if we identify for any $s \in \ag{C}$ the space of functions $\pi_s$ with $H$, where
	$$ \pi_s := \Ind_{\gp{B}(\ag{A})}^{\GL_2(\ag{A})} (\norm_{\ag{A}}^s, \norm_{\ag{A}}^{-s}), \quad H := \Ind_{\gp{B}(\ag{A}) \cap \gp{K}}^{\gp{K}} 1, $$
then we can regard the intertwining operator $\Intw_s : \pi_s \to \pi_{-s}$ as a map from $H$ to itself. Using the \emph{flat section} map $H \to \pi_s, f \mapsto f_s$, we mean
	$$ (\Intw_s f_s)(a(y)\kappa) =: \norm[y]_{\ag{A}}^{\frac{1}{2}-s} (\Intw_s f)(\kappa), \quad \text{i.e.,} \quad \Intw_s f_s = \left( \Intw_s f \right)_{-s}. $$
	Let $e_0 \in H$ be the constant function taking value $1$. Define
	$$ \ProjP_{\gp{K}}: H \to \ag{C}, \quad f \mapsto \int_{\gp{K}} f(\kappa) d\kappa, $$
where $d\kappa$ is the probability Haar measure on $\gp{K}$. We obtain a map from $H$ to itself
	$$ \widetilde{\Intw}_s := \Intw_s \circ (I - \ProjP_{\gp{K}}e_0), $$
where $I$ is the identity map. Since $\Intw_s$ is ``diagonalizable'', we obtain the Taylor expansion as operators
	$$ \Intw_sf = \sum_{n=0}^{\infty} \frac{s^n}{n!} \Intw_0^{(n)}f, \quad \text{resp.} \quad \widetilde{\Intw}_{1/2+s} f = \sum_{n=0}^{\infty} \frac{s^n}{n!} \widetilde{\Intw}_{1/2}^{(n)}f. $$
\begin{theorem}{(\cite[Theorem 3.5]{Wu9})}
	The regularized integral of the product of two unitary Eisenstein series is computed as:
\begin{itemize}
	\item[(1)] If $\pi_1 = \pi(\xi_1,\xi_2), \pi_2 = \pi(\xi_1^{-1}, \xi_2^{-1})$ resp. $\pi_2 = \pi(\xi_2^{-1}, \xi_1^{-1})$ and $\xi_1 \neq \xi_2$, then
	$$ \int_{[\PGL_2]}^{\reg} \eis(0,f_1) \eis(0,f_2) = \frac{2\lambda_{\F}^{(0)}(0)}{\lambda_{\F}^{(-1)}(0)} \ProjP_{\gp{K}}(f_1f_2) - \ProjP_{\gp{K}}(\Intw_0^{(1)}f_1 \cdot \Intw_0 f_2), \quad \text{resp.} $$
	$$ \frac{\lambda_{\F}^{(0)}(0)}{\lambda_{\F}^{(-1)}(0)} (\ProjP_{\gp{K}}(f_1 \Intw_0 f_2) + \ProjP_{\gp{K}}(f_2 \Intw_0 f_1)) - \ProjP_{\gp{K}}(\Intw_0^{(1)}f_1 \cdot f_2). $$
	\item[(2)] If $\pi_1 = \pi(\xi,\xi), \pi_2 = \pi(\xi^{-1},\xi^{-1})$, then
\begin{align*}
	&\quad \int_{[\PGL_2]}^{\reg} \eis^{(1)}(0,f_1) \eis^{(1)}(0,f_2) = \frac{4\lambda_{\F}^{(2)}(0)}{\lambda_{\F}^{-1}(0)} \ProjP_{\gp{K}}(f_1f_2) + \frac{4\lambda_{\F}^{(2)}(0)}{\lambda_{\F}^{-1}(0)} \ProjP_{\gp{K}}(f_1 \cdot \Intw_0^{(1)} f_2) \\
	&\quad + \frac{\lambda_{\F}^{(0)}(0)}{\lambda_{\F}^{-1}(0)} \ProjP_{\gp{K}}(\Intw_0^{(1)}f_1 \cdot \Intw_0^{(1)}f_2) - \frac{1}{3} \ProjP_{\gp{K}}(\Intw_0^{(3)}f_1 \cdot f_2) - \ProjP_{\gp{K}}(\Intw_0^{(2)}f_1 \cdot \Intw_0^{(1)}f_2).
\end{align*}
\end{itemize}
	Here we have written (\cite[(2.2)]{Wu9})
	$$ \lambda_{\F}(s) := \frac{\Lambda_{\F}(-2s)}{\Lambda_{\F}(2+2s)} = \frac{\lambda_{\F}^{(-1)}(0)}{s} + \sum_{n=0}^{\infty} \frac{s^n}{n!} \lambda_{\F}^{(n)}(0). $$
\label{RIPEisUnitary}
\end{theorem}

\noindent In \cite{Wu2}, we have used and extended the above theory to a special case of regularized triple product of Eisenstein series, which includes a generalization of $R(s,\varphi)$ to $R(s,\varphi;f)$ for finitely regularizable $\varphi \in \Cont^{\infty}(\GL_2,\omega)$ and $f \in \pi(\xi_1,\xi_2)$ with $\xi_1\xi_2\omega=1$. Namely, it is the analytic continuation of (see \cite[Proposition 2.5]{Wu2})
	$$ R(s,\varphi;f) := \int_{\F^{\times} \backslash \A^{\times}} \int_{\gp{K}} (\varphi_{\gp{N}}-\varphi_{\gp{N}}^*)(a(y)\kappa) f(\kappa) \xi_1(y) \norm[y]_{\A}^{s-1/2} d\kappa d^{\times}y. $$
This will be further extended to other relevant cases in this paper in \S \ref{FERTPF}. For the moment, we simply record \cite[Theorem 2.7]{Wu2} for convenience of the reader.

\begin{theorem}
	Let $f_j \in \pi(1,1), j=1,2,3$. Then for any $n \in \ag{N}$
	$$ \int_{[\PGL_2]}^{\reg} \eis^*(0,f_1) \cdot \eis^*(0,f_2) \cdot \eis^{\reg,(n)}(\frac{1}{2},f_3) $$
is the sum of
	$$ \left( \frac{\partial^n R}{\partial s^n} \right)^{\hol} \left( \frac{1}{2}, \eis^*(0,f_1) \cdot \eis^*(0,f_2); f_3 \right) $$
and a weighted sum with coefficients depending only on $\lambda_{\F}(s)$ of
	$$ \ProjP_{\gp{K}}(\Intw_0^{(l)}f_1 \cdot f_2) \ProjP_{\gp{K}}(f_3), \quad 0 \leq l \leq 3; $$
	$$ \ProjP_{\gp{K}}(f_1\cdot f_2 \cdot \widetilde{\Intw}_{1/2}^{(l)} f_3), \quad 0 \leq l \leq \max(2,n) \quad \& \quad l = n+3; $$
	$$ \ProjP_{\gp{K}}((f_1 \Intw_0 f_2 + f_2 \Intw_0 f_1) \cdot \widetilde{\Intw}_{1/2}^{(l)} f_3), \quad 0 \leq l \leq \max(1,n) \quad \& \quad l = n+2; $$
	$$ \ProjP_{\gp{K}}(\Intw_0 f_1\cdot \Intw_0 f_2 \cdot \widetilde{\Intw}_{1/2}^{(l)} f_3), \quad 0 \leq l \leq n \quad \& \quad l = n+1. $$
\label{RegTripEis}
\end{theorem}

	\subsection{Explicit Decomposition of Periods}
	
	Let $(\pi_j, V_j)$ be three irreducible automorphic representations of $\GL_2(\A)$ whose central characters $\omega_j$ satisfy $\sideset{}{_{j=1}^3} \prod \omega_j = 1$. The space of $\GL_2(\A)$-invariant functionals $\ell: V_1 \times V_2 \times V_3 \to \C$ has dimension at most $1$. If one of $\pi_j$ is an Eisenstein series, then one can apply the Rankin-Selberg method to unfold $\ell$, yielding a decomposition of $\ell$ as a product of local $\GL_2(\F_v)$-invariant functionals $\ell_v$. We will deal with situations where at least two of $\pi_j$ are Eisenstein series. Hence the global functional has multiple decompositions into local terms. It is necessary to specify which decomposition we use in the argument. We give a list of such specification below, which are best suited for our purpose of estimation in most situations. Only occasionally should we need to change the local trilinear forms. The relevant tool is given in \S 5.2.

		\subsubsection{On $\pi \times \pi(1,1) \times \pi(1,\omega^{-1})$.}

	For $\varphi \in \pi, f_2 \in \pi(1,1), f_3 \in \pi(1,\omega^{-1})$, we take the explicit decomposition of period:
\begin{equation}
	\int_{[\PGL_2]} \varphi \cdot \eis^*(0,f_2) \cdot \eis^{\sharp}(0,f_3) = L(\frac{1}{2},\pi)^2 \cdot \prod_v \ell_v(W_{\varphi,v}, f_{2,v}, f_{3,v})
\label{RSbeforeCS}
\end{equation}
where the local trilinear forms are defined by
\begin{align*}
	&\ell_v(\cdots) = \int_{\F_v^{\times} \times \gp{K}_v} W_{\varphi,v}(a(y)\kappa) W_{f_2,v}^*(a(-y)\kappa) f_{3,v}(\kappa) \norm[y]_v^{-\frac{1}{2}} d^{\times}y d\kappa, \quad v \mid \infty; \\
	&\ell_{\vp}(\cdots) = \frac{L(1,\omega_{\vp})}{L(1/2,\pi_{\vp})^2} \int_{\F_{\vp}^{\times} \times \gp{K}_{\vp}} W_{\varphi,\vp}(a(y)\kappa) W_{f_2,\vp}^*(a(-y)\kappa) f_{3,\vp}(\kappa) \norm[y]_{\vp}^{-\frac{1}{2}} d^{\times}y d\kappa, \quad \vp < \infty;
\end{align*}
so that $\ell_{\vp}=1$ for all but finitely many $\vp$.

		\subsubsection{On $\pi' \times \overline{\pi(1,\omega^{-1})} \times \pi(1,\omega^{-1})$.}

	Let $\pi'$ be a cuspidal representation of $\PGL_2$. For $\varphi \in \pi', f_3, f_3' \in \pi(1,\omega^{-1})$, we take the explicit decomposition of period:
\begin{equation}
	\int_{[\PGL_2]} \varphi \cdot \overline{\eis^{\sharp}(0,f_3)} \cdot \eis^{\sharp}(0,f_3') = L(\frac{1}{2},\pi') L(\frac{1}{2},\pi' \otimes \omega) \cdot \prod_v \ell_v(W_{\varphi,v}, \overline{f_{3,v}}, f_{3,v}')
\label{CuspPerD}
\end{equation}
where the local trilinear forms are defined by
\begin{align*}
	&\ell_v(\cdots) = \int_{\F_v^{\times} \times \gp{K}_v} W_{\varphi,v}(a(y)\kappa) \overline{W_{f_3,v}(a(y)\kappa)} f_{3,v}'(\kappa) \norm[y]_v^{-\frac{1}{2}} d^{\times}y d\kappa, \quad v \mid \infty; \\
	&\ell_{\vp}(\cdots) = \frac{L(1,\omega_{\vp})}{L(1/2,\pi_{\vp})L(1/2,\pi_{\vp} \otimes \omega_{\vp})} \int_{\F_{\vp}^{\times} \times \gp{K}_{\vp}} W_{\varphi,\vp}(a(y)\kappa) \overline{W_{f_3,\vp}^*(a(y)\kappa)} f_{3,\vp}'(\kappa) \norm[y]_{\vp}^{-\frac{1}{2}} d^{\times}y d\kappa, \quad \vp < \infty;
\end{align*}
so that $\ell_{\vp}=1$ for all but finitely many $\vp$.

		\subsubsection{On $\pi(\xi \norm_{\ag{A}}^s, \xi \norm_{\ag{A}}^{-s}) \times \overline{\pi(1,\omega^{-1})} \times \pi(1,\omega^{-1})$.}

	Let $\xi$ be a character of $\F^{\times} \backslash \ag{A}^{(1)}$, trivially extended to a Hecke character. Let $\Phi \in \pi(\xi,\xi^{-1})$, to which is associated a flat section $\Phi_s \in \pi(\xi \norm_{\ag{A}}^s, \xi \norm_{\ag{A}}^{-s})$. For $f_3, f_3' \in \pi(1,\omega^{-1})$ and $\Re s \in (-1/2,1/2)$, we have by \cite[Proposition 2.5]{Wu2}
\begin{align}
	&\quad \int_{[\PGL_2]} \eis(s,\Phi) \cdot \left( \overline{\eis^{\sharp}(0,f_3)} \cdot \eis^{\sharp}(0,f_3') - \Reis(\eis^{\sharp}(0,f_3) \eis^{\sharp}(0,f_3')) \right) \nonumber \\
	&= \frac{L(\frac{1}{2}+s, \xi)^2 L(\frac{1}{2}+s, \xi\omega) L(\frac{1}{2}+s, \xi\omega^{-1})}{L(1+2s, \xi^2)} \prod_v \ell_v(s; \Phi_v, \overline{f_{3,v}}, f_{3,v}') \label{EisPerD}
\end{align}
where the local factors are defined by
\begin{align*}
	&\ell_v(\cdots) = \int_{\F_v^{\times} \times \gp{K}_v} \overline{W_{f_3,v}(a(y)\kappa)} W_{f_3',v}(a(y)\kappa) \Phi_v(\kappa)\xi_v(y) \norm[y]_v^{s-\frac{1}{2}} d^{\times}y d\kappa, \quad v \mid \infty; \\
	&\ell_{\vp}(\cdots) = \frac{L(1+2s, \xi_{\vp}^2)}{L(\frac{1}{2}+s, \xi_{\vp})^2 L(\frac{1}{2}+s, \xi_{\vp}\omega_{\vp}) L(\frac{1}{2}+s, \xi_{\vp}\omega_{\vp}^{-1})} \cdot \\
	&\quad \quad \int_{\F_{\vp}^{\times} \times \gp{K}_{\vp}} \overline{W_{f_3,\vp}^*(a(y)\kappa)} W_{f_3',\vp}^*(a(y)\kappa) \Phi_{\vp}(\kappa) \xi_{\vp}(y) \norm[y]_{\vp}^{s-\frac{1}{2}} d^{\times}y d\kappa, \quad \vp < \infty;
\end{align*}
so that $\ell_{\vp}=1$ for all but finitely many $\vp$. As a $\GL_2(\F_v)$-invariant trilinear form, $\ell_v$ is not always convenient for our purpose of estimation. We shall also need $\tilde{\ell}_v(\Phi_v, \overline{f_{3,v}}, f_{3,v}')$ defined by
\begin{align*}
	&\tilde{\ell}_v(\cdots) = \int_{\F_v^{\times} \times \gp{K}_v} W_{\Phi_{s,v}}(a(y)\kappa) \overline{W_{f_3,v}(a(y)\kappa)} f_{3,v}'(\kappa) \norm[y]_v^{-\frac{1}{2}} d^{\times}y d\kappa, \quad v \mid \infty; \\
	&\tilde{\ell}_{\vp}(\cdots) = \frac{L(1, \omega_{\vp})}{L(\frac{1}{2}+s, \xi_{\vp}) L(\frac{1}{2}-s, \xi_{\vp}^{-1}) L(\frac{1}{2}+s, \xi_{\vp}\omega_{\vp}) L(\frac{1}{2}-s, \xi_{\vp}^{-1}\omega_{\vp})} \cdot \\
	&\quad \quad \int_{\F_{\vp}^{\times} \times \gp{K}_{\vp}} W_{\Phi_{s,\vp}}^*(a(y)\kappa) \overline{W_{f_3,\vp}^*(a(y)\kappa)} f_{3,\vp}'(\kappa) \norm[y]_{\vp}^{-\frac{1}{2}} d^{\times}y d\kappa, \quad \vp < \infty;
\end{align*}
so that $\tilde{\ell}_{\vp}=1$ for all but finitely many $\vp$.

		\subsubsection{One-dimensional Projection of $\overline{\pi(1,\omega^{-1})} \times \pi(1,\omega^{-1})$. }

	Let $\chi$ be a quadratic Hecke character unramified at every finite place (i.e., a quadratic class group character). For $f_3, f_3' \in \pi(1,\omega^{-1})$, we take the explicit decomposition of the regularized integral (c.f. \cite[\S 2.3]{Wu9})
\begin{align}
	&\quad \int_{[\PGL_2]}^{\reg} \overline{\eis^{\sharp}(0,f_3)} \cdot \eis^{\sharp}(0,f_3') \cdot \chi \circ \det \nonumber \\
	&= \Res_{s=\frac{1}{2}} \frac{L(\frac{1}{2}+s, \chi)^2 L(\frac{1}{2}+s, \chi\omega) L(\frac{1}{2}+s, \chi\omega^{-1})}{\zeta_{\F}(1+2s)} \prod_v \ell_v(s; \overline{f_{3,v}}, f_{3,v}',\chi_v) \label{OneDimProj}
\end{align}
where the local factors are defined by
\begin{align*}
	&\ell_v(\cdots) = \int_{\F_v^{\times} \times \gp{K}_v} \overline{W_{f_3,v}(a(y)\kappa)} W_{f_3',v}(a(y)\kappa) \chi_v(y) \norm[y]_v^{s-\frac{1}{2}} d^{\times}y d\kappa, \quad v \mid \infty; \\
	&\ell_{\vp}(\cdots) = \frac{\zeta_{\vp}(1+2s)}{L(\frac{1}{2}+s, \chi_{\vp})^2 L(\frac{1}{2}+s, \chi_{\vp}\omega_{\vp}) L(\frac{1}{2}+s, \chi_{\vp}\omega_{\vp}^{-1})} \cdot \\
	&\quad \quad \int_{\F_{\vp}^{\times} \times \gp{K}_{\vp}} \overline{W_{f_3,\vp}^*(a(y)\kappa)} W_{f_3',\vp}^*(a(y)\kappa) \chi_{\vp}(y) \norm[y]_{\vp}^{s-\frac{1}{2}} d^{\times}y d\kappa, \quad \vp < \infty;
\end{align*}
so that $\ell_{\vp}=1$ for all but finitely many $\vp$, and all $\ell_v$ are holomorphic at $s=1/2$.
\begin{remark}
	The pole at $s=1/2$ of the global $L$-factor in (\ref{OneDimProj}) has order equal to $0$ if $\chi \neq 1, \omega \neq \chi$; $2$ if $\chi \neq 1, \omega = \chi$ or $\chi = 1, \omega \neq \chi$; $4$ if $\chi=\omega=1$. Hence if $\Cond(\omega)$ is sufficiently large (depending only on $\F$), (\ref{OneDimProj}) is non-vanishing only if $\chi=1$. We also note $L^{(k)}(1,\chi\omega^{\pm 1}) \ll_{\F,\epsilon} \Cond(\omega)^{\epsilon}$ for $0 \leq k \leq 4$.
\label{OneDimProjRk}
\end{remark}

		\subsubsection{On Regularized $\pi(1,1) \times \overline{\pi(1,1)} \times \pi(\norm_{\ag{A}}^{1/2+s}, \norm_{\ag{A}}^{-(1/2+s)})$}
		
	For $\tilde{f}_2, f_2 \in \pi(1,1), \tilde{f}_3, f_3 \in \pi(1,\omega^{-1})$, let $f \in \{ (\tilde{f}_3 \cdot \overline{f_3}) \mid_{\gp{K}}, (\IntwR_0 \tilde{f}_3 \cdot \overline{\IntwR_0 f_3}) \mid_{\gp{K}} \}$, regarded as an element in $\pi(1,1)$. Write $\widetilde{\eis}_2^* = \eis^*(0,\tilde{f}_2)$ and $\widetilde{\eis}_3^{\sharp} = \eis^{\sharp}(0,\tilde{f}_3)$. We take the explicit decomposition of the extended Rankin-Selberg integral (c.f. \cite[Proposition 2.5]{Wu2})
\begin{equation}
	R(\frac{1}{2}+s, \widetilde{\eis}_2^* \overline{\eis_2^*}; f) = \frac{\zeta_{\F}(1+s)^4}{\zeta_{\F}(2+2s)} \cdot \prod_v \ell_v(s; \tilde{f}_{2,v}, \overline{f_{2,v}}; f_v)
\label{Regu1PerD}
\end{equation}
where the local factors are defined by
\begin{align*}
	&\ell_v(\cdots) = \int_{\F_v^{\times} \times \gp{K}_v} W_{\tilde{f}_2,v}^*(a(y)\kappa) \overline{W_{f_2,v}^*(a(y)\kappa)} f_v(\kappa) \norm[y]_v^s d^{\times}y d\kappa, \quad v \mid \infty; \\
	&\ell_{\vp}(\cdots) = \frac{\zeta_{\vp}(2+2s)}{\zeta_{\vp}(1+s)^4} \cdot \int_{\F_{\vp}^{\times} \times \gp{K}_{\vp}} W_{\tilde{f}_2,\vp}^*(a(y)\kappa) \overline{W_{f_2,\vp}^*(a(y)\kappa)} f_{\vp}(\kappa) \norm[y]_{\vp}^s d^{\times}y d\kappa, \quad \vp < \infty;
\end{align*}
so that $\ell_{\vp}=1$ for all but finitely many $\vp$, and all $\ell_v$ are holomorphic at $s=0$.

		\subsubsection{On Regularized $\pi(1,1) \times \overline{\pi(1,1)} \times \pi(\omega^{-1}\norm_{\ag{A}}^{1/2+s}, \omega\norm_{\ag{A}}^{-(1/2+s)})$ resp. $\pi(\omega\norm_{\ag{A}}^{1/2+s}, \omega^{-1}\norm_{\ag{A}}^{-(1/2+s)})$}
		
	Let notations be as in the previous case. Consider $f = (\IntwR_0 \tilde{f}_3 \cdot \overline{f_3}) \mid_{\gp{K}}$ resp. $(\tilde{f}_3 \cdot \overline{\IntwR_0 f_3}) \mid_{\gp{K}}$, regarded as an element in $\pi(\omega^{-1},\omega)$ resp. $\pi(\omega,\omega^{-1})$. We take
\begin{equation}
	R(\frac{1}{2}+s, \widetilde{\eis}_2^* \overline{\eis_2^*}; f) = \frac{L(1+s,\omega^{\mp})^4}{L(2+2s,\omega^{\mp 2})} \cdot \prod_v \ell_v(s; \tilde{f}_{2,v}, \overline{f_{2,v}}; f_v)
\label{Regu2PerD}
\end{equation}
where the local factors are defined by
\begin{align*}
	&\ell_v(\cdots) = \int_{\F_v^{\times} \times \gp{K}_v} W_{\tilde{f}_2,v}^*(a(y)\kappa) \overline{W_{f_2,v}^*(a(y)\kappa)} f_v(\kappa) \omega_v^{\mp}(y)\norm[y]_v^s d^{\times}y d\kappa, \quad v \mid \infty; \\
	&\ell_{\vp}(\cdots) = \frac{L_{\vp}(2+2s, \omega_{\vp}^{\mp 2})}{L_{\vp}(1+s, \omega_{\vp}^{\mp})^4} \cdot \int_{\F_{\vp}^{\times} \times \gp{K}_{\vp}} W_{\tilde{f}_2,\vp}^*(a(y)\kappa) \overline{W_{f_2,\vp}^*(a(y)\kappa)} f_{\vp}(\kappa) \omega_{\vp}^{\mp}(y)\norm[y]_{\vp}^s d^{\times}y d\kappa, \quad \vp < \infty;
\end{align*}
so that $\ell_{\vp}=1$ for all but finitely many $\vp$, and all $\ell_v$ are holomorphic at $s=0$.

		\subsubsection{On Regularized $\pi(1,\omega^{-1}) \times \overline{\pi(1,\omega^{-1})} \times \pi(\norm_{\ag{A}}^{1/2+s}, \norm_{\ag{A}}^{-(1/2+s)})$}
		
	For $\tilde{f}_2, f_2 \in \pi(1,1), \tilde{f}_3, f_3 \in \pi(1,\omega^{-1})$, let $f = (\tilde{f}_2 \cdot \overline{f_2}) \mid_{\gp{K}}$, regarded as an element in $\pi(1,1)$. Write $\widetilde{\eis}_3^{\sharp} = \eis^{\sharp}(0,\tilde{f}_3)$. We take the explicit decomposition of the extended Rankin-Selberg integral (c.f. \cite[Proposition 2.5]{Wu2})
\begin{equation}
	R(\frac{1}{2}+s, \widetilde{\eis}_3^{\sharp} \overline{\eis_3^{\sharp}}; f) = \frac{\zeta_{\F}(1+s)^2 L(1+s,\omega^{-1}) L(1+s, \omega)}{\zeta_{\F}(2+2s)} \cdot \prod_v \ell_v(s; \tilde{f}_{3,v}, \overline{f_{3,v}}; f_v)
\label{Regu3PerD}
\end{equation}
where the local factors are defined by
\begin{align*}
	&\ell_v(\cdots) = \int_{\F_v^{\times} \times \gp{K}_v} W_{\tilde{f}_3,v}^*(a(y)\kappa) \overline{W_{f_3,v}^*(a(y)\kappa)} f_v(\kappa) \norm[y]_v^s d^{\times}y d\kappa, \quad v \mid \infty; \\
	&\ell_{\vp}(\cdots) = \frac{\zeta_{\vp}(2+2s)}{\zeta_{\vp}(1+s)^2 L_{\vp}(1+s, \omega_{\vp}^{-1}) L_{\vp}(1+s, \omega_{\vp})} \cdot \\
	&\quad \quad \int_{\F_{\vp}^{\times} \times \gp{K}_{\vp}} W_{\tilde{f}_3,\vp}^*(a(y)\kappa) \overline{W_{f_3,\vp}^*(a(y)\kappa)} f_{\vp}(\kappa) \norm[y]_{\vp}^s d^{\times}y d\kappa, \quad \vp < \infty;
\end{align*}
so that $\ell_{\vp}=1$ for all but finitely many $\vp$, and all $\ell_v$ are holomorphic at $s=0$.

	\subsection{Proof of Main Result: First Reduction}
	
	In this subsection, we outline the ``big steps'' of the proof, which are global in nature. Basically, there is a main global triple product period (see Proposition \ref{MainLocEst} (2) below), which serves as the pivot: we bound it from below in terms of targeted $L$-value $\norm[L(1/2,\pi)]^2$, and we bound it from above with some method of amplification. The lower bound follows directly from the local lower bounds, which will be given in \S 3, together with the explicit choice of test functions. The upper bound is achieved by the upper bounds of several other global periods, whose proofs will be given in \S 4. These global periods also decompose as products of local terms. Their upper bounds follow from the corresponding local bounds, which are again given in \S 3.

	Take $\varphi \in \pi, f_2 \in \pi(1,1)$ and $f_3 \in \pi(1,\omega^{-1})$, which will be specified in Section \ref{ChoicesNA} \& \ref{ChoicesA}. In particular, the local Whittaker functions $W_{\varphi,v}$ will be specified at each $v \mid \infty$, while $W_{\varphi,\vp}$ will be specified up to a scalar at $\vp < \infty$. We can always adjust these scalars to make the global norm $\Norm[\varphi]=1$. Recall the relation of the local and global norms \cite[Lemma 2.10]{Wu14}
	$$ 1 = \Norm[\varphi]^2 = 2 L(1,\pi,\mathrm{Ad}) \prod_{v \mid \infty} \frac{\zeta_v(2)}{\zeta_v(1)} \Norm[W_{\varphi,v}]^2 \prod_{\vp < \infty} \frac{\zeta_{\vp}(2)\Norm[W_{\varphi,\vp}]^2}{L(1,\pi_{\vp} \times \bar{\pi}_{\vp})}. $$
	
\begin{proposition}
	(1) With our choice of test vectors and $\ell_v$ defined in (\ref{RSbeforeCS}), we have
	$$ \prod_{v \mid \infty} \frac{\extnorm{ \ell_v(W_{\varphi,v}, f_{2,v}, f_{3,v}) }}{\Norm[W_{\varphi,v}]} \sqrt{\frac{\zeta_v(1)}{\zeta_v(2)}} \prod_{\vp < \infty} \frac{\extnorm{ \ell_{\vp}(W_{\varphi,\vp}, f_{2,\vp}, f_{3,\vp}) }}{\Norm[W_{\varphi,\vp}]} \sqrt{\frac{L_{\vp}(1,\pi_{\vp} \times \bar{\pi}_{\vp})}{\zeta_{\vp}(2)}} \gg_{\F,\epsilon} \Cond(\pi)^{-\frac{1}{2}-\epsilon}. $$
	(2) Consequently, we have
	$$ \extnorm{\int_{[\PGL_2]} \varphi \cdot \eis^*(0,f_2) \cdot \eis^{\sharp}(0,f_3)} \gg_{\F,\epsilon} \Cond(\pi)^{-\frac{1}{2}-\epsilon} \extnorm{L \left( \frac{1}{2},\pi \right)}^2. $$
\label{MainLocEst}
\end{proposition}
\begin{proof}
	(1) follows from Lemma \ref{LocLowerBdNA} and \ref{LocLowerBdA}. (2) follows from (1) and (\ref{RSbeforeCS}) by inserting $L(1,\pi,\mathrm{Ad}) \ll_{\epsilon} \Cond(\pi)^{\epsilon}$ (see \cite{HL94} and \cite[Lemma 3]{BH10}).
\end{proof}

	Write $\eis_2^* = \eis^*(0,f_2)$ resp. $\eis_3^{\sharp} = \eis^{\sharp}(0,f_3)$ for simplicity of notation.
\begin{lemma}
	Recall the operator $\sigma_0$ in Definition \ref{NormHeckeOp}. It has the following basic properties.
\begin{itemize}
	\item[(1)] $\varphi$ is an eigenvector of the dual operator $\sigma_0^{\vee} = (\alpha_0^{-1} T_0(1) - \lambda_0(0))^2 (\alpha_0^{-1} T_0(1) - \tilde{\lambda}_0(0))^2$ (see Proposition \ref{DualHeckeOp}) with eigenvalue $R_0$ satisfying $\norm[R_0] \asymp 1$.
	\item[(2)] $\sigma_0 (\eis_2^* \eis_3^{\sharp}) \in \mathcal{A}(\GL_2, \omega^{-1})$ is of rapid decay (see Proposition \ref{DualHeckeOp}).
\end{itemize}
\label{RDRegProd}
\end{lemma}
\begin{proof}
	(1) follows trivially from MacDonald's formula \cite[Theorem 4.6.6]{Bu98}. For (2), we only treat the case $\omega \neq 1$. Note that $\sigma_0$ annihilates
\begin{align*}
	(\eis_{2,\gp{N}}^* \cdot \eis_{3,\gp{N}}^{\sharp})(a(y)\kappa) &= \Lambda_{\F}^* L(1,\omega) \cdot \left\{ \norm[y]_{\ag{A}} \log \norm[y]_{\ag{A}} f_3(\kappa) + 2^{-1} \lambda_{\F}^{(1)}(-1/2) \norm[y]_{\ag{A}} f_3(\kappa) \right. \\
	&\left. + \omega(y) \norm[y]_{\ag{A}} \log \norm[y]_{\ag{A}} \Intw_0f_3(\kappa) + 2^{-1} \lambda_{\F}^{(1)}(-1/2) \omega(y) \norm[y]_{\ag{A}} \Intw_0 f_3(\kappa) \right\},
\end{align*}
	since we have, if we denote by $f_s$ resp. $\tilde{f}_s$ a flat spherical section in $\pi(\norm_{\vp_0}^{1/2+s},\omega_0^{-1} \norm_{\vp_0}^{-(1/2+s)})$ resp. $\pi(\omega_0^{-1}\norm_{\vp_0}^{1/2+s},\norm_{\vp_0}^{-(1/2+s)})$ and $f_0' := (\partial f_s / \partial s) \mid_{s=0}$,
	$$ T_0(1) (f_0, f_0') = (f_0, f_0') \begin{pmatrix} \lambda_0(0) & \lambda_0'(0) \\ & \lambda_0(0) \end{pmatrix} \quad \text{resp.} \quad T_0(1) (\tilde{f}_0, \tilde{f}_0') = (\tilde{f}_0, \tilde{f}_0') \begin{pmatrix} \tilde{\lambda}_0(0) & \tilde{\lambda}_0'(0) \\ & \tilde{\lambda}_0(0) \end{pmatrix}. $$
	Hence $\sigma_0(\eis_2^* \eis_3^{\sharp})$ is \emph{finitely regularizable} \cite[Definition 2.14]{Wu9} with \emph{essential constant term} equal to $0$, thus of rapid decay, \emph{a fortiori} square integrable.
\end{proof}	
\noindent By (\ref{RSbeforeCS}) and Proposition \ref{MainLocEst}, we are reduced to bounding
\begin{align}
	\int_{[\PGL_2]} \varphi \cdot \eis_2^* \cdot \eis_3^{\sharp} &= \left( R_0 \sum_{\vp \in S} q_{\vp}^{n_{\vp}/2} \norm[\lambda_{\pi}(\vp^{n_{\vp}})] \right)^{-1} \int_{[\PGL_2]} \sigma_0^{\vee}\sigma(\varphi) \cdot \eis_2^* \cdot \eis_3^{\sharp} \nonumber \\
	&= \left( R_0 \sum_{\vp \in S} q_{\vp}^{n_{\vp}/2} \norm[\lambda_{\pi}(\vp^{n_{\vp}})] \right)^{-1} \int_{[\PGL_2]} \varphi \cdot \sigma^{\vee} \sigma_0(\eis_2^* \eis_3^{\sharp}). \label{SumBeforeCS}
\end{align}
	We apply Cauchy-Schwartz inequality and unfold the square as (recall $q_{\vp} := \Nr(\vp), \alpha_{\vp} := \omega_{\vp}(\varpi_{\vp})$)
\begin{align}
	&\quad \extnorm{\int_{[\PGL_2]} \varphi \cdot \sigma^{\vee}\sigma_0(\eis_2^* \eis_3^{\sharp})}^2 \ll_{\F} \int_{[\PGL_2]} \extnorm{ \sum_{\vp \in S} a_{\vp} T(\vp^{n_{\vp}}). \sigma_0(\eis_2^* \eis_3^{\sharp}) }^2 \nonumber \\
	&= \sum_{\vp_1,\vp_2 \in S} a_{\vp_1} \overline{a_{\vp_2}} \int_{[\PGL_2]} T(\vp_1^{n_{\vp_1}}) T(\vp_2^{n_{\vp_2}})^* \sigma_0(\eis_2^* \eis_3^{\sharp}) \cdot \overline{ \sigma_0(\eis_2^* \eis_3^{\sharp}) } \nonumber \\
	&= \sideset{}{_{\vp \in S}} \sum \norm[a_{\vp}]^2 \alpha_{\vp}^{-n_{\vp}} \int_{[\PGL_2]} T(\vp^{n_{\vp}})^2 \sigma_0^*\sigma_0(\eis_2^* \eis_3^{\sharp}) \cdot \overline{ \eis_2^* \eis_3^{\sharp} } \nonumber \\
	&\quad + \sideset{}{_{\substack{\vp_1,\vp_2 \in S \\ \vp_1 \neq \vp_2}}} \sum a_{\vp_1} \overline{a_{\vp_2}} \alpha_{\vp_2}^{-n_{\vp_2}} \int_{[\PGL_2]} T(\vp_1^{n_{\vp_1}}) T(\vp_2^{n_{\vp_2}}) \sigma_0^*\sigma_0(\eis_2^* \eis_3^{\sharp}) \cdot \overline{ \eis_2^* \eis_3^{\sharp} }. \label{C-S}
\end{align}
	From Proposition \ref{HeckeRel}, we deduce that
	$$ \sigma_0^*\sigma_0 = \sideset{}{_{j=0}^8} \sum b_j T_0(j) $$
	for some $b_j \in \C$ with $\norm[b_j] \ll 1$. Hence the last two lines of (\ref{C-S}) are equal to
\begin{align}
	&\quad \sideset{}{_{j=0}^8} \sum \sideset{}{_{\vp \in S}} \sum b_j \norm[a_{\vp}]^2 \alpha_{\vp}^{-n_{\vp}} \int_{[\PGL_2]}^{\reg} T(\vp^{n_{\vp}})^2 T_0(j)^*(\eis_2^* \eis_3^{\sharp}) \cdot \overline{(\eis_2^* \eis_3^{\sharp}) } \nonumber \\
	&\quad + \sideset{}{_{j=0}^8} \sum \sideset{}{_{\substack{\vp_1,\vp_2 \in S \\ \vp_1 \neq \vp_2}}} \sum b_j a_{\vp_1} \overline{a_{\vp_2}} \alpha_{\vp_2}^{-n_{\vp_2}} \int_{[\PGL_2]}^{\reg} T(\vp_1^{n_{\vp_1}}) T(\vp_2^{n_{\vp_2}}) T_0(j)^*(\eis_2^* \eis_3^{\sharp}) \cdot \overline{(\eis_2^* \eis_3^{\sharp}) } \nonumber \\
	&= \sideset{}{_{j=0}^8} \sum \sideset{}{_{\vp \in S}} \sum b_j \norm[a_{\vp}]^2 \alpha_{\vp}^{-n_{\vp}} \cdot \int_{[\PGL_2]}^{\reg} D_{j,\vp}(\eis_2^* \eis_3^{\sharp}) \cdot \overline{(\eis_2^* \eis_3^{\sharp}) } \nonumber \\
	&\quad + \sideset{}{_{j=0}^8} \sum \sideset{}{_{\substack{\vp_1,\vp_2 \in S \\ \vp_1 \neq \vp_2}}} \sum b_j a_{\vp_1} \overline{a_{\vp_2}} \alpha_{\vp_2}^{-n_{\vp_2}} \int_{[\PGL_2]}^{\reg} a(\varpi_0^{-j} \varpi_{\vp_1}^{-n_{\vp_1}} \varpi_{\vp_2}^{-n_{\vp_2}}).(\eis_2^* \eis_3^{\sharp}) \cdot \overline{(\eis_2^* \eis_3^{\sharp}) }, \label{SumAfterCS}
\end{align}
	where the diagonal operators (see the last formulas in Proposition \ref{HeckeRel}) are given by
	$$ D_{j,\vp} := \left\{ \begin{matrix} \frac{q_{\vp}}{q_{\vp}+1}a(\varpi_0^{-j} \varpi_{\vp}^{-2}) - \frac{\alpha_{\vp}^{-1}}{q_{\vp}+1} a(\varpi_0^{-j}) & \text{if } n_{\vp} = 1 \\ \frac{q_{\vp}}{q_{\vp}+1}a(\varpi_0^{-j} \varpi_{\vp}^{-4}) - \frac{\alpha_{\vp}^{-1}}{q_{\vp}+1} a(\varpi_0^{-j} \varpi_{\vp}^{-2}) - \frac{\alpha_{\vp}^{-2}}{q_{\vp}(q_{\vp}+1)} a(\varpi_0^{-j}) & \text{if } n_{\vp} = 2 \end{matrix} \right. $$
	and we have applied the invariance property of the regularized integrals \cite[Proposition 2.27 (2)]{Wu9}. Introduce 
\begin{equation} 
	\mathfrak{t}=\varpi_0^{-j} \varpi_{\vp}^{-kn_{\vp}}, \quad k= \left\{ \begin{matrix} 1,2 & \text{if } n_{\vp} = 1 \\ 0,1,2 & \text{if } n_{\vp}=2 \end{matrix} \right. ; \quad \text{or} \quad \mathfrak{t} = \varpi_0^{-j} \varpi_{\vp_1}^{-n_{\vp_1}} \varpi_{\vp_2}^{-n_{\vp_2}}, 
\label{TransIndex}
\end{equation}
	which are understood as finite ideles. We re-arrange the (prototype) regularized integral as
\begin{align}
	&\quad \int_{[\PGL_2]}^{\reg} a(\mathfrak{t}).\left( \eis_2^* \eis_3^{\sharp} \right) \cdot \overline{ \left( \eis_2^* \eis_3^{\sharp} \right) } \nonumber \\
	&= \int_{[\PGL_2]} \left( a(\mathfrak{t}).\eis_2^* \cdot \overline{\eis_2^*} - \Reis(a(\mathfrak{t}).\eis_2^* \overline{\eis_2^*}) \right) \cdot \left( a(\mathfrak{t}).\eis_3^{\sharp} \cdot \overline{\eis_3^{\sharp}} - \Reis(a(\mathfrak{t}).\eis_3^{\sharp} \overline{\eis_3^{\sharp}}) \right) \label{RegTerm} \\
	&+ \int_{[\PGL_2]}^{\reg} a(\mathfrak{t}).\eis_2^* \cdot \overline{\eis_2^*} \cdot \Reis(a(\mathfrak{t}).\eis_3^{\sharp} \overline{\eis_3^{\sharp}}) + \int_{[\PGL_2]}^{\reg} a(\mathfrak{t}).\eis_3^{\sharp} \cdot \overline{\eis_3^{\sharp}} \cdot \Reis(a(\mathfrak{t}).\eis_2^* \overline{\eis_2^*}) \label{ReguTerm} \\
	&- \int_{[\PGL_2]}^{\reg} \Reis(a(\mathfrak{t}).\eis_2^* \overline{\eis_2^*}) \cdot \Reis(a(\mathfrak{t}).\eis_3^{\sharp} \overline{\eis_3^{\sharp}}). \label{DgnTerm}
\end{align}

\begin{definition}
	For $\mathfrak{t}$ given in (\ref{TransIndex}), we define
	$$ \Norm[\mathfrak{t}] = \left\{ \begin{matrix} n_{\vp_1} + n_{\vp_2} & \text{if } \vp_1 \neq \vp_2, \\ k n_{\vp_1} & \text{if } \vp_1 = \vp_2; \end{matrix} \right. \quad n(\mathfrak{t}) = (m_{\vp})_{\vp < \infty} \text{ with } m_{\vp}= \left\{ \begin{matrix} j & \text{if } \vp = \vp_0, \\ n_{\vp_1} & \text{if } \vp = \vp_1 \neq \vp_2, \\ n_{\vp_2} & \text{if } \vp = \vp_2 \neq \vp_1, \\ kn_{\vp} & \text{if } \vp = \vp_1 = \vp_2, \\ 0 & \text{otherwise}. \end{matrix} \right. $$ 
	For $\pi'$ an automorphic representation, we write $\pi' \leq \mathfrak{t}$ if $\cond(\pi_{\vp}') \leq m_{\vp}$ for all $\vp < \infty$.
\end{definition}
\begin{remark}
	We will refer to (\ref{RegTerm}) resp. (\ref{ReguTerm}) resp. (\ref{DgnTerm}) as the \emph{regular term} resp. \emph{regularized term} resp. \emph{degenerate term}. The effect of the measure of regularization and the one of amplification will be treated in the same way, explaining why we put them together into $\mathfrak{t}$. The smallness of $q_0$ (\ref{Smallp0}), implying that any power of it is $\ll_{\epsilon} \Cond(\pi)^{\epsilon}$ hence negligible, allows us to basically ignore the contribution at $\vp_0$ in the estimation.
\end{remark}

\noindent In the sum (\ref{SumAfterCS}), we call the terms with $\vp_1 = \vp_2$ the \emph{diagonal terms}, while the ones with $\vp_1 \neq \vp_2$ the \emph{off-diagonal terms}. The number of diagonal terms is $O(\norm[S])$. The number of off-diagonal terms is $O(\norm[S]^2)$. We shall apply Lemma \ref{RegTermBd}, \ref{ReguTermBd1}, \ref{ReguTermBd2} and \ref{DgnTermBd}. For simplicity of notations, denote
	$$ V_1 := (\frac{\Cond(\pi)}{\Cond(\omega)})^{-\frac{1}{2}+\theta} \Cond(\omega)^{-\delta'}, \quad V_2 := (\frac{\Cond(\pi)}{\Cond(\omega)})^{-\frac{1}{2}} \Cond(\omega)^{-\frac{1-2\theta}{8}}. $$
	The diagonal terms are bounded by (recall $\norm[a_{\vp}] \asymp K^{1/2}$)
\begin{align*}
	&\quad (\Cond(\pi)K)^{\epsilon} \sideset{}{_{j=0}^8} \sum \sideset{}{_{\vp \in S}} \sum K^{n_{\vp}} \left\{ \max(V_1 K^{2n_{\vp}(A+1/2)}, V_2 K^{2n_{\vp} \left( \frac{1}{2} - \frac{1-2\theta}{8} \right)  }, K^{-2n_{\vp}} ) \right. \\
	&\qquad \left. + K^{-1} \max(V_1 K^{n_{\vp}(A+1/2)}, V_2 K^{n_{\vp} \left( \frac{1}{2} - \frac{1-2\theta}{8} \right) }, K^{-n_{\vp}}) + K^{-2} \max(V_1, V_2, 1) \right\} \\
	&\ll (\Cond(\pi)K)^{\epsilon} \norm[S] \cdot \max(V_1 K^{4(A+1)}, V_2 K^{4 \left( 1 - \frac{1-2\theta}{8} \right)}, 1).
\end{align*}
	Similarly, the off-diagonal terms are bounded by
\begin{align*}
	&\quad (\Cond(\pi)K)^{\epsilon} \sum_{j=0}^8 \sideset{}{_{\substack{\vp_1,\vp_2 \in S \\ \vp_1 \neq \vp_2}}} \sum K^{\frac{n_{\vp_1} + n_{\vp_2}}{2}} \max( V_1 K^{(n_{\vp_1} + n_{\vp_2})(A+1/2)}, V_2 K^{(n_{\vp_1} + n_{\vp_2}) \left( \frac{1}{2} - \frac{1-2\theta}{8} \right) }, K^{-(n_{\vp_1} + n_{\vp_2})}) \\
	&\ll (\Cond(\pi)K)^{\epsilon} \norm[S]^2 \max(V_1 K^{4(A+1)}, V_2 K^{4 \left( 1 - \frac{1-2\theta}{8} \right) }, K^{-1}).
\end{align*}
	By (2.6), we then deduce
	$$ \extnorm{ \int_{[\PGL_2]} \varphi \cdot \eis_2^* \cdot \eis_3^{\sharp} } \ll_{\F,\epsilon} (\Cond(\pi)K)^{\epsilon} \max(\sqrt{V_1} K^{2(A+1)}, \sqrt{V_2} K^{2 \left( 1 - \frac{1-2\theta}{8} \right) }, K^{-1/2}). $$
	Assuming $\delta' \leq (1-2\theta)/8$ and $A \geq 1/4$, we deduce that (\ref{SumBeforeCS}) is bounded by
	$$ \left( \frac{\Cond(\pi)}{\Cond(\omega)} \right)^{-\frac{1-2\theta}{20+16A}} \Cond(\omega)^{-\frac{\delta'}{10+8A}}, \quad \text{with} \quad K = \left( \frac{\Cond(\pi)}{\Cond(\omega)} \right)^{\frac{1-2\theta}{10+8A}} \Cond(\omega)^{\frac{\delta'}{5+4A}}. $$

\section{Local Choices and Estimations}

	\subsection{Non-Archimedean Places}
	
		\subsubsection{Choices and Main Bounds}
		\label{ChoicesNA}

	At $\vp < \infty$, we choose test vectors as \cite[\S 3.6.2	]{MV10}. Precisely, we choose $W_{\varphi,\vp}$ to be a new vector of $\pi_{\vp}$ in the Whittaker model; $f_{2,\vp}$ to be the spherical function taking value $1$ at $1$ in the induced model of $\pi(1,1)$; $f_{3,\vp}$ whose restriction to $\gp{K}_{\vp}$ is
	$$ \begin{pmatrix} a & b \\ c & d \end{pmatrix} \mapsto \Vol(\gp{K}_0[\vp^{\cond(\pi_{\vp})}])^{-1/2} \omega_{\vp}^{-1}(d) \mathbbm{1}_{\gp{K}_0[\vp^{\cond(\pi_{\vp})}]}. $$
\begin{lemma}
	If $\cond(\pi_{\vp}) > 0$, then we have, with $\ell_{\vp}$ defined in (\ref{RSbeforeCS}) and absolute implicit constant,
	$$ \extnorm{ \ell_{\vp}(W_{\varphi,\vp}, f_{2,\vp}, f_{3,\vp}) } \gg \Cond(\pi_{\vp})^{-1/2} \cdot \frac{\Norm[W_{\varphi,\vp}]}{\sqrt{L(1, \pi_{\vp} \times \bar{\pi}_{\vp})}}. $$
\label{LocLowerBdNA}
\end{lemma}
\begin{proof}
	The proof is exactly the same as \cite[\S 3.6.2]{MV10}, except that we take into account various $L$-factors. We also remark that the necessary formula of $W_{\varphi,\vp}$ can be found in \cite[Table 1]{FMP17}, and that we are following the style of \cite{Wu14}, i.e., without specific normalization for $W_{\varphi,\vp}$.
\end{proof}
\begin{lemma}
(1) Let $\pi_{\vp}'$ be a unitary spherical representation with trivial central character and spectral parameter $\leq \theta$. Let $W'=W_{\varphi',\vp}$ be the Whittaker function of a spherical vector in $\pi_{\vp}'$. Let $\ell_{\vp}$ be defined in (\ref{CuspPerD}). With absolute implicit constant, we have
	$$ \extnorm{ \ell_{\vp}(W', \overline{f_{3,\vp}}, f_{3,\vp}) } \ll \Cond(\pi_{\vp})^{-\frac{1}{2}} (\Cond(\pi_{\vp}) / \Cond(\omega_{\vp}))^{\theta} \cdot \frac{\Norm[W']}{\sqrt{L(1, \pi_{\vp}' \times \overline{\pi_{\vp}'})}}. $$
(2) Let $\Phi_{\vp}$ be the spherical function of $\pi(\xi_{\vp}, \xi_{\vp}^{-1})$ taking value $1$ on $\gp{K}_{\vp}$ and $\tau \in \ag{R}$. Let $\ell_{\vp}$ be defined in (\ref{EisPerD}). With absolute implicit constant, we have for any $\epsilon > 0$
	$$ \extnorm{ \ell_{\vp}(i\tau; \Phi_{\vp}, \overline{f_{3,\vp}}, f_{3,\vp}) } \ll_{\epsilon} \Cond(\pi_{\vp})^{-\frac{1}{2}+\epsilon}. $$
(3) Let $\ell_{\vp}$ be defined in (\ref{OneDimProj}). With absolute implicit constant, we have for any $k \in \ag{N}$ and $\epsilon > 0$
	$$ \extnorm{ \ell_{\vp}^{(k)}(1/2;\overline{f_{3,\vp}},f_{3,\vp},1) } \ll_{k,\epsilon} \Cond(\pi_{\vp})^{\epsilon}. $$
\label{LocUpperBdNA}
\end{lemma}
\begin{proof}
	(1) We drop the subscript $\vp$ for simplicity of notations. Let $e_0$ be a unitary new vector in $\pi(1,\omega^{-1})$. We also write $W_3^*$ resp. $W_0^*$ for $W_{f_3}^*$ resp. $W_{e_0}^*$. As \cite[\S 3.6.2]{MV10}, the integral part of $\ell_{\vp}$ is equal to
	$$ \Vol(\gp{K}_0[\vp^{\cond(\pi)}])^{1/2} \int_{\F^{\times}} W'(a(y)) \overline{W_3^*(a(y))} \norm[y]^{-1/2} d^{\times}y. $$
	If $\cond(\omega) > 0$, then $f_3 = a(\varpi^{-n}).e_0$ for $n=\cond(\pi) - \cond(\omega)$ by Proposition \ref{LocBaseRel} (3). Using \cite[Table 1]{FMP17}, we can evaluate the above integral as, with $(\alpha, \alpha^{-1})$ the Satake parameter of $\pi'$,
	$$ \Vol(\gp{K}_0[\vp^{\cond(\pi)}])^{1/2} W'(1) L(1/2,\pi') \left( \frac{\alpha^{n+1} - \alpha^{-(n+1)}}{\alpha - \alpha^{-1}} - \frac{\alpha^n - \alpha^{-n}}{\alpha - \alpha^{-1}} q^{-\frac{1}{2}} \right), $$
and conclude by $q^{-\theta} \leq \norm[\alpha] \leq q^{\theta}$.

\noindent If $\cond(\omega)=0$ with $\alpha_1 = \omega(\varpi)$ (we can assume $n>0$ since the case $n=0$ is easy), then $f_3 \asymp L(1,\omega) (a(\varpi^{-n}).e_0 - \alpha_1 q^{-1/2} a(\varpi^{-(n-1)}).e_0)$ by Proposition \ref{LocBaseRel} (4). $a(\varpi^{-n}).e_0$ contributes to the integral as the product of
	$$ L(1/2,\pi')L(1/2,\pi'\otimes \omega) \cdot \Vol(\gp{K}_0[\vp^{\cond(\pi)}])^{1/2} W'(1) \quad \text{and} $$
	$$ \left\{ \frac{\alpha^{n+1}-\alpha^{-(n+1)}}{\alpha - \alpha^{-1}} - (1+\alpha_1)q^{-\frac{1}{2}} \frac{\alpha^n-\alpha^{-n}}{\alpha - \alpha^{-1}} + \alpha_1 q^{-1}\frac{\alpha^{n-1}-\alpha^{-(n-1)}}{\alpha - \alpha^{-1}}  \right\}, $$
while the second term contributes less. We conclude.

\noindent (2) Since $s=i\tau \in i\ag{R}$, Proposition \ref{TransFLocRS} tells us that $\ell_{\vp}$ and $\tilde{\ell}_{\vp}$ have the same size. But $\tilde{\ell}_{\vp}$ is of the same shape as $\ell_{\vp}$ in the cuspidal case above. Hence our bounds are the same as (1) with $\theta=0$.

\noindent (3) By Proposition \ref{TransFLocRS}, $\ell_{\vp}^{(k)}$ is of size $\Cond(\omega_{\vp})^{1/2+\epsilon}$ times $\tilde{\ell}_{\vp}^{(k)}$. Arguing as in (1), replacing $\alpha$ with $\alpha q^s$  ($s$ around $1/2$), where $\alpha = \xi(\varpi)$, we obtain and conclude by
	$$ \extnorm{ \tilde{\ell}_{\vp}^{(k)}(1/2;\overline{f_{3,\vp}},f_{3,\vp},1) } \ll_{\epsilon} \Cond(\pi_{\vp})^{-\frac{1}{2}} (\Cond(\pi_{\vp}) / \Cond(\omega_{\vp}))^{1/2+\epsilon}. $$
\end{proof}

		\subsubsection{Bounds for Regularization and Amplification}

	We restrict to a finite place $\vp \in S^*$ defined in (\ref{AmpPlace}). Let $n_{\vp} \in \ag{Z}$ and $1 \leq \norm[n_{\vp}] \leq 2$ (see (\ref{Amplifier})) if $\vp \in S$; $1 \leq \norm[n_{\vp}] \leq 8$ if $\vp=\vp_0$. Recall $f_{3,\vp}$ is $\gp{K}_{\vp}$-invariant.
\begin{lemma}
	(1) Let $\pi_{\vp}'$ be a unitary representation with trivial central character and $\cond(\pi_{\vp}') \leq \norm[n_{\vp}]$. Let $W'=W_{\varphi',\vp}$ run over an orthogonal basis of $\gp{K}_{\vp} \cap a(\varpi_{\vp}^{n_{\vp}}) \gp{K}_{\vp} a(\varpi_{\vp}^{-n_{\vp}})$-invariant vectors in the Whittaker model of $\pi_{\vp}'$, with different $\gp{K}_{\vp}$-types. For $\ell_{\vp}$ defined in (\ref{CuspPerD}) we have the estimation
	$$ \extnorm{ \ell_{\vp}(W',\overline{f_{3,\vp}}, a(\varpi_{\vp}^{n_{\vp}}).f_{3,\vp}) } \ll q^{-\frac{\norm[n_{\vp}]}{2}} \frac{\Norm[W']}{\sqrt{L(1,\pi_{\vp}' \times \overline{\pi_{\vp}'})}}. $$
(2) Let $\xi_{\vp}$ be a character of $\F_{\vp}^{\times}$ with $\cond(\xi_{\vp}) \leq \norm[n_{\vp}]/2$. Let $\Phi_{\vp}$ run over an orthogonal basis of $\gp{K}_{\vp} \cap a(\varpi_{\vp}^{n_{\vp}}) \gp{K}_{\vp} a(\varpi_{\vp}^{-n_{\vp}})$-invariant vectors in $\pi(\xi,\xi^{-1})$, with different $\gp{K}_{\vp}$-types. For $\ell_{\vp}$ defined in (\ref{EisPerD}) and $\tau \in \ag{R}$ we have the estimation
	$$ \ell_{\vp}(i\tau; \Phi_{\vp}, \overline{f_{3,\vp}}, a(\varpi_{\vp}^{n_{\vp}}).f_{3,\vp}) \ll q^{-\frac{\norm[n_{\vp}]}{2}}. $$
\label{LocAmpBd}
\end{lemma}
\begin{proof}
	(1) We drop the subscript $\vp$ for simplicity of notations. We first notice that the case $n > 0$ can be transformed into the case $n < 0$. In fact, by invariances we have
	$$ \ell_{\vp}(W',\overline{f_3}, a(\varpi^n).f_3) = \ell_{\vp}(w.W',\overline{w.f_3}, wa(\varpi^n).f_3) = \ell_{\vp}(w.W',\overline{f_3}, a(\varpi^{-n}).f_3) \omega(\varpi)^{-n}, $$
and $w.W'$ runs over a basis with $\gp{K} \cap a(\varpi^n) \gp{K}_{\vp} a(\varpi^{-n})$-invariance replaced by $\gp{K} \cap a(\varpi^{-n}) \gp{K}_{\vp} a(\varpi^n)$-invariance. We shall decompose $a(\varpi^n).f_3$ resp. $W'$ according to \emph{basis 3} resp. \emph{basis 1}, defined in Section \ref{BaseGNV}. By Proposition \ref{LocBaseRel} (4-2), $a(\varpi^n).f_3$ is a linear combination of $q^{-(n-k)/2}D_k$ for $0 \leq k \leq n$, with coefficients of size $\asymp 1$. If $W'$ is of level $m \leq n$, then we have
	$$ \int_{\gp{K}} W'(a(y)\kappa) \overline{W_3^*(a(y)\kappa)} q^{-(n-k)/2}D_k(\kappa) d\kappa \asymp q^{-n/2} W'(a(y)) \overline{W_3^*(a(y))} \mathbbm{1}_{m \geq k}, $$
since ($\intL^1$-normalized) $D_k$ induces the orthogonal projection onto the subspace of $\gp{K}_0[\vp^k]$-invariant vectors. By Proposition \ref{LocBaseRel} (1)+(2)+(3)+(4-2), $W'$ is a linear combination of $a(\varpi^{-l}).W_0'$ for $m-2 \leq l+\cond(\pi') \leq m$, with coefficients of absolute value $\leq 1$, where $W_0'$ is a $\intL^2$-normalized new vector in the Whittaker model of $\pi'$. We are reduced to computing
\begin{equation}
	\int_{\F^{\times}} a(\varpi^{-l}).W_0'(a(y)) \overline{W_3^*(a(y))} \norm[y]^{-\frac{1}{2}} d^{\times}y.
\label{FinInt}
\end{equation}
We write $\alpha = \omega(\varpi)$ and use \cite[Table 1]{FMP17} distinguishing several cases:

\noindent (\rmnum{1}) $\pi'$ is spherical with Satake parameters $\alpha_1,\alpha_2$ ($\alpha_1 \alpha_2 = 1$). (\ref{FinInt}) is equal to ($l \in \{ 0,1,2 \}$)
	$$ \frac{L(1/2,\pi')L(1/2,\pi' \otimes \omega)}{\sqrt{L(1,\pi' \times \overline{\pi'})}} \cdot \left( \frac{1-\alpha^{l+1}}{1-\alpha} - (\alpha_1+\alpha_2) \alpha q^{-\frac{1}{2}} \frac{1-\alpha^l}{1-\alpha} + \alpha^2 q^{-1} \frac{1-\alpha^{l-1}}{1-\alpha} \right). $$

\noindent (\rmnum{2}) $\pi' \simeq \mathrm{St}_{\chi}$ with $\alpha' = \chi(\omega) \in \{ \pm 1 \}$. (\ref{FinInt}) is equal to ($l \in \{ 0,1 \}$)
	$$ \frac{L(1/2,\pi')L(1/2,\pi' \otimes \omega)}{\sqrt{L(1,\pi' \times \overline{\pi'})}} \cdot \left( \frac{1-\alpha^{l+1}}{1-\alpha} - \alpha' \alpha q^{-1} \frac{1-\alpha^l}{1-\alpha} \right). $$

\noindent (\rmnum{3}) $\cond(\pi') = 2$, which in our case implies $L(s,\pi')=1$. (\ref{FinInt}) is equal to ($l=0$) $1$.

\noindent In conclusion, (\ref{FinInt}) does not create increase or decrease in terms of $q^n$ and we are done.

\noindent (2) Proposition \ref{TransFLocRS} tells us that $\ell_{\vp}$ is of the same size as $\tilde{\ell}_{\vp}$, which can be bounded the same way as in the cuspidal case above.
\end{proof}

		\subsubsection{Main Bounds in The Regularized Terms}
		
	We restrict to a finite place $\vp \notin S$.
\begin{lemma}
\begin{itemize}
	\item[(1)] For $\ell_{\vp}$ in (\ref{Regu1PerD}) and $f_{\vp} \in \{ (f_{3,\vp} \cdot \overline{f_{3,\vp}}) \mid_{\gp{K}_{\vp}}, (\IntwR_0 f_{3,\vp} \cdot \overline{ \IntwR_0 f_{3,\vp} }) \mid_{\gp{K}} \}$, we have
	$$ \ell_{\vp}(s; f_{2,\vp}, \overline{f_{2,\vp}}; f_{\vp}) = 1. $$
	\item[(2)] For $\ell_{\vp}$ in (\ref{Regu2PerD}) and $f_{\vp} \in \{ (\IntwR_0 f_{3,\vp} \cdot \overline{f_{3,\vp}}) \mid_{\gp{K}_{\vp}}, (f_{3,\vp} \cdot \overline{ \IntwR_0 f_{3,\vp} }) \mid_{\gp{K}} \}$, $\ell_{\vp}(s; f_{2,\vp}, \overline{f_{2,\vp}}; f_{\vp})$ is a constant satisfying
	$$ \extnorm{\ell_{\vp}(s; f_{2,\vp}, \overline{f_{2,\vp}}; f_{\vp})} \left\{ \begin{matrix} =0 & \text{if } \cond(\omega_{\vp}) \neq 0,  \\ \asymp \Cond(\pi_{\vp})^{-1} & \text{if } \cond(\omega_{\vp}) = 0. \end{matrix} \right. $$
	\item[(3)] For $\ell_{\vp}$ in (\ref{Regu3PerD}) and $f_{\vp} = (f_{2,\vp} \cdot \overline{f_{2,\vp}}) \mid_{\gp{K}_{\vp}} = 1$, we have for any $k \in \ag{N}$ and $\epsilon > 0$
	$$ \extnorm{ \ell_{\vp}^{(k)}(0;f_{3,\vp}, \overline{f_{3,\vp}};1) } \ll_{k,\epsilon} \Cond(\pi_{\vp})^{\epsilon}. $$
\end{itemize}
\label{ReguNAMainBd}
\end{lemma}
\begin{proof}
	(1) $f_{2,\vp}$ being $\gp{K}_{\vp}$-invariant, we get
	$$ \ell_{\vp}(s; f_{2,\vp}, \overline{f_{2,\vp}}; f_{\vp}) = \frac{\zeta_{\vp}(1+s)^4}{\zeta_{\vp}(2+2s)} \cdot \int_{\F_{\vp}^{\times}} \extnorm{W_{2,\vp}^*(a(y))}^2 \norm[y]_{\vp}^s d^{\times}y \cdot \int_{\gp{K}_{\vp}} f_{\vp}(\kappa) d\kappa = \int_{\gp{K}_{\vp}} f_{\vp}(\kappa) d\kappa. $$
	We get the desired equality since $f_{3,\vp}$ is a unitary vector and $\IntwR_0$ is unitary.
	
\noindent (2) A similar argument as in (1) gives
	$$ \ell_{\vp}(s; f_{2,\vp}, \overline{f_{2,\vp}}; f_{\vp}) = \left\{ \begin{matrix} 0 & \text{if } \cond(\omega_{\vp}) \neq 0,  \\ \int_{\gp{K}_{\vp}} f_{\vp}(\kappa) d\kappa & \text{if } \cond(\omega_{\vp}) = 0. \end{matrix} \right. $$
	Drop the subscript $\vp$ for simplicity. Assume $\cond(\omega)=0$. Take the case $f=(\IntwR_0 f_3 \cdot \overline{f_3}) \mid_{\gp{K}}$ for example. Write $n=\cond(\pi)$. By choice, $\int_{\gp{K}} \overline{f_3(\kappa)} \kappa d\kappa$ is equal to $\Vol(\gp{K}_0[\vp^n])^{1/2}$ times the orthogonal projection onto the $\gp{K}_0[\vp^n]$-invariant subspace, hence
	$$ \int_{\gp{K}} f(\kappa) d\kappa = \Vol(\gp{K}_0[\vp^n])^{1/2} \IntwR_0 f_3(1). $$
	Proposition \ref{LocBaseRel} (4-2), together with the observation $e_k(1)=0$ for $k \geq 1$ with notations in that proposition, then gives
	$$ \IntwR_0 f_3(1) = q^{-n/2}(1+q^{-1})^{-1/2} \IntwR_0 e_0(1) = q^{-n/2}(1+q^{-1})^{-1/2}, $$
where $e_0$ is the spherical function in $\pi(1,\omega^{-1})$ taking value $1$ on $\gp{K}$.

\noindent (3) This is in fact the same as Lemma \ref{LocUpperBdNA} (3).
\end{proof}	
	
		\subsubsection{Bounds in The Regularized Terms for Regularization and Amplification}
		
	We restrict to a finite place $\vp \in S^*$ and write $t_{\vp} = \varpi_{\vp}^{-n}$ with $n \in \{ n_{\vp}, -n_{\vp} \}$ or $\{ k_1-k_2, k_2-k_1 \}$.
\begin{lemma}
\begin{itemize}
	\item[(1)] For $\ell_{\vp}$ in (\ref{Regu1PerD}) and $k \in \ag{N}$, we have
	$$ \extnorm{\ell_{\vp}^{(k)}(0; a(t_{\vp}).f_{2,\vp}, \overline{f_{2,\vp}}; a(t_{\vp}).f_{3,\vp} \overline{f_{3,\vp}})} \ll_k (\norm[n]+1)^{k+3} q^{-n} (\log q)^k, $$
	$$ \extnorm{\ell_{\vp}^{(k)}(0; a(t_{\vp}).f_{2,\vp}, \overline{f_{2,\vp}}; \IntwR_0 (a(t_{\vp}).f_{3,\vp}) \overline{\IntwR_0f_{3,\vp}})} \ll_k (\norm[n]+1)^{k+3} q^{-n} (\log q)^k. $$
	\item[(2)] For $\ell_{\vp}$ in (\ref{Regu2PerD}) and $k \in \ag{N}$, we have
	$$ \extnorm{\ell_{\vp}^{(k)}(0; a(t_{\vp}).f_{2,\vp}, \overline{f_{2,\vp}}; a(t_{\vp}).f_{3,\vp} \overline{\IntwR_0 f_{3,\vp}})} \ll_k (\norm[n]+1)^{k+3} q^{-n} (\log q)^k, $$
	$$ \extnorm{\ell_{\vp}^{(k)}(0; a(t_{\vp}).f_{2,\vp}, \overline{f_{2,\vp}}; \IntwR_0 (a(t_{\vp}).f_{3,\vp}) \overline{f_{3,\vp}})} \ll_k (\norm[n]+1)^{k+3} q^{-n} (\log q)^k. $$
	\item[(3)] For $\ell_{\vp}$ in (\ref{Regu3PerD}) and $k \in \ag{N}$, we have
	$$ \extnorm{\ell_{\vp}^{(k)}(0; a(t_{\vp}).f_{3,\vp}, \overline{f_{3,\vp}}; a(t_{\vp}).f_{2,\vp} \overline{f_{2,\vp}})} \ll_k (\norm[n]+1)^{k+3} q^{-n} (\log q)^k, $$
	$$ \extnorm{\ell_{\vp}^{(k)}(0; a(t_{\vp}).f_{3,\vp}, \overline{f_{3,\vp}}; \IntwR_0 (a(t_{\vp}).f_{2,\vp}) \overline{f_{2,\vp}})} \ll_k (\norm[n]+1)^{k+3} q^{-n} (\log q)^k, $$
	$$ \extnorm{\ell_{\vp}^{(k)}(0; a(t_{\vp}).f_{3,\vp}, \overline{f_{3,\vp}}; \IntwR_0^{(1)} (a(t_{\vp}).f_{2,\vp}) \overline{f_{2,\vp}})} \ll_k (\norm[n]+1)^{k+3} q^{-n} (\log q)^{k+1}. $$
\end{itemize}
\label{ReguAmpBd}
\end{lemma}
\begin{proof}
	We drop the subscript $\vp$ for simplicity of notations.

\noindent (1) The second inequality essentially follows from the first by replacing $\omega$ with $\omega^{-1}$, since $\IntwR_0 f_3 \cdot \omega \circ \det$ is the corresponding $f_3$. By $\gp{K}$-invariance of $\ell$ and $f_2,f_3$ we have
	$$ \ell(s; a(t).f_2, \overline{f_2}; a(t).f_3 \overline{f_3}) = \ell(s; wa(t)w.f_2, \overline{w.f_2}; wa(t)w.f_3 \overline{w.f_3}) = \omega^{-1}(t)\ell(s; a(t^{-1}).f_2, \overline{f_2}; a(t^{-1}).f_3 \overline{f_3}), $$
	hence we may assume $n > 0$. The integral part of $\ell$ has the form
	$$ \int_{\F^{\times}} \overline{W_2^*(a(y))} \norm[y]^s d^{\times}y \int_{\gp{K}} a(t).f_3(\kappa) \kappa.(a(t).W_2^*)(a(y)) d\kappa. $$
	We enter into the setting of Section \ref{BaseGNV}, distinguishing elements related to $f_2$ from those to $f_3$ by putting a ``$*$''. (For example, $e_0=f_3, W_0^*=W_2^*$.) Recall the projectors $\Proj_n$ defined in Corollary \ref{ProjCal}. We have the relations
	$$ \Proj_n = \int_{\gp{K}} q^{\frac{n}{2}}(1+q^{-1})^{\frac{1}{2}} D_n(\kappa) \kappa d\kappa, \text{ if } n \geq 1; \quad \Proj_0 = D_0. $$
	By Proposition \ref{LocBaseRel} (4-2), writing $\alpha = \omega(\varpi)$, we get
	$$ \int_{\gp{K}} a(t).e_0(\kappa) \kappa d\kappa = \alpha^n q^{-\frac{n}{2}} \left\{ \Proj_0 + \frac{1-\alpha q^{-1}}{1+q^{-1}} \sum_{l=1}^n \alpha^{-l} \Proj_l \right\}. $$
	Together with Corollary \ref{ProjCal} we obtain
\begin{align*}
	\int_{\gp{K}} a(t).e_0(\kappa) \kappa a(t).e_0^* d\kappa &= \alpha^n q^{-n} \cdot \left\{ (n+1-\frac{2n}{q+1}) e_0^* + \frac{1-\alpha q^{-1}}{1+q^{-1}} \cdot \right. \\
	&\left. \sum_{l=1}^n \alpha^{-l} \left( (n-l+1)q^{\frac{l}{2}} a(\varpi^l).e_0^* - (n-l) q^{\frac{l-1}{2}} a(\varpi^{l-1)}).e_0^* \right) \right\}.
\end{align*}
	Hence we are reduced to computing
	$$ q^{\frac{l}{2}} \int_{\F^{\times}} a(\varpi^{-l}).W_0^*(a(y)) \overline{W_0^*(a(y))} \norm[y]^s d^{\times}y = q^{-ls} \cdot \left\{ \frac{1+q^{-(1+s)}}{(1-q^{-(1+s)})^3} + \frac{l}{(1-q^{-(1+s)})^2} \right\}. $$
	If we put $A_l(s) := q^{-ls} \{ 1 + l (1-q^{-(1+s)})(1+q^{-(1+s)})^{-1} \}$ for $l \geq 1$ and $A_0(s)=1$, then we get
\begin{align*}
	\ell(s; a(t).f_2, \overline{f_2}; a(t)f_3 \overline{f_3}) &= \alpha^n q^{-n} \cdot \left\{ (n+1-\frac{2n}{q+1}) + \frac{1-\alpha q^{-1}}{1+q^{-1}} \cdot \right. \\
	&\left. \sum_{l=1}^n \alpha^{-l} \left( (n-l+1) A_l(s) - (n-l) A_{l-1}(s) \right) \right\},
\end{align*}
	from which we easily deduce the desired bound.
	
\noindent (2) The argument is quite similar to (1) above. For example for the case $a(t)f_3 \overline{\IntwR_0 f_3}$, we only need to replace $A_l(s)$ with $ A_l'(s) := \alpha^l q^{-ls} \{ 1 + l (1-\alpha q^{-(1+s)})(1+\alpha q^{-l(1+s)})^{-1}$.

\noindent (3) The argument is again similar to (1) above. For example for the case $a(t)f_2 \overline{f_2}$, we need to replace $A_l(s)$ resp. $\ell(s;\cdots)$ with
	$$ A_l''(s) = q^{-ls} \left\{ 1 + \frac{\bar{\alpha}(1-\bar{\alpha}^l)}{1-\bar{\alpha}} \frac{1-\alpha q^{-(1+s)}}{1+q^{-(1+s)}} \right\}, \quad \text{resp.} $$
\begin{align*}
	\ell(s; a(t).f_3, \overline{f_3}; a(t)f_2 \overline{f_2}) &= q^{-n} \cdot \left\{ \left( \frac{1-\alpha^{n+1}}{1-\alpha} - \frac{1+\alpha}{q+1} \frac{1-\alpha^n}{1-\alpha} \right) + \frac{1-q^{-1}}{1+q^{-1}} \cdot \right. \\
	&\left. \sum_{l=1}^n \left( \frac{1-\alpha^{n-l+1}}{1-\alpha} A_l''(s) - \frac{1-\alpha^{n-l}}{1-\alpha} A_{l-1}''(s) \right) \right\}.
\end{align*}
\end{proof}

	\subsection{Archimedean Places}

		\subsubsection{Choices and Lower Bounds}
		\label{ChoicesA}

	The choice of the local test functions at the archimedean places is the subtlest construction in \cite{MV10}. We find it convenient if we specify them in two steps with some non-vanishing condition:

\noindent (1) Let $\tilde{f}_0$ be a fixed (depending only on $\F_v = \ag{R}$ or $\ag{C}$) smooth unitary vector in $\Res_{\gp{K}_v}^{\GL_2(\F_v)} \pi(1,1)$ with support contained in a small neighborhood $U$ of $\gp{B}_v \cap \gp{K}_v$ in $\gp{K}_v$. Define $f_0 \in \pi(1,\omega_v^{-1})^{\infty}$ by requiring
	$$ \tilde{f}_0: \gp{K}_v \to \ag{C}, \quad \begin{pmatrix} a & b \\ c & d \end{pmatrix} = \kappa \mapsto \omega_v(d) f_0 (\kappa). $$
	Our test function is $f_{3,v} = a(C).f_0$, where $C \in \F_v^{\times}$ with $\norm[C] = \Cond(\pi_v)^{1+\epsilon}$. It can be easily verified that for any Sobolev norm $\Sob$ defined with differential operator on $\gp{G}$

	$$ \Sob_d(f_0) \ll \Cond(\omega_v)^d. $$

\noindent (2) There is a non-negative bump function $\phi$ on $\F_v$ with support contained in a small compact neighborhood of $0$, say in $\{ x \in \F_v: \norm[x] \leq \delta_0 \}$ such that
	$$ f_0(\kappa) = \int_{\F_v^{\times}} \Psi_0((0,t).\kappa) \omega_v(t) \norm[t]_v d^{\times}t, \quad \text{i.e.} \quad \tilde{f}_0(\kappa) = \int_{\F_v^{\times}} \phi(ct) \phi(dt-1) \norm[t]_v d^{\times}t, $$
where $\Psi_0 \in \Sch(\F_v^2)$ is defined via
	$$ \Psi_0(x,y) = \phi(x) \phi(y-1) \omega_v^{-1}(y), \quad \text{and} \quad \kappa = \begin{pmatrix} a & b \\ c & d \end{pmatrix}. $$

\begin{remark}
	Note that the Kirillov norm
	$$ \Norm[W_{3,v}]^2 = \int_{\F_v^{\times}} \extnorm{W_{3,v}(a(y))}^2 d^{\times}y = \frac{\zeta_v(1)^2}{\zeta_v(2)} \int_{\gp{K}_v} \extnorm{f_{3,v}(\kappa)}^2 d\kappa \asymp \Norm[f_{3,v}]^2 $$
is essentially the same as the induced norm. Hence we may regard $W_{3,v}$ as unitary.
\end{remark}
\begin{lemma}
	If $U \subset \gp{K}$ is a small neighborhood of $\gp{K} \cap \gp{B}$, then for $C \in \F_v^{\times}$
	$$ U_C := \{ \kappa \in \gp{K}: \kappa a(C) \in \gp{B} U \} $$
is a small neighborhood in $\gp{K}$ which shrinks to $\gp{K} \cap \gp{B}$ as $\norm[C] \to \infty$. Moreover, we have
	$$ {\rm Ht}_v(\kappa a(C)) \asymp_{U} \norm[C]_v, \quad \forall \kappa \in U_C. $$
\label{ShrinkNB}
\end{lemma}
\begin{proof}
	Let $\alpha, \beta \in \F_v$ such that $\norm[\alpha]^2 + \norm[\beta]^2 = 1$. We have
	$$ \begin{pmatrix} \alpha & \beta \\ - \bar{\beta} & \bar{\alpha} \end{pmatrix} \begin{pmatrix} C & \\ & 1 \end{pmatrix} = \begin{pmatrix} \frac{C}{\sqrt{\norm[\alpha]^2 + \norm[C \beta]^2}} & * \\ & \sqrt{\norm[\alpha]^2 + \norm[C \beta]^2} \end{pmatrix} \begin{pmatrix} \alpha & \bar{C} \beta \\ - C \bar{\beta} & \bar{\alpha} \end{pmatrix} \cdot \frac{1}{\sqrt{\norm[\alpha]^2 + \norm[C \beta]^2}}. $$
	The smallness of $U$ implies $\norm[C \beta] \leq \delta_0 \norm[\alpha]$ for some small $\delta_0 > 0$, which implies
	$$ \norm[\beta]^2 \leq \norm[\delta_0]^2 (\norm[C]^2 + \norm[\delta_0]^2)^{-1} \leq (\norm[\delta_0] / \norm[C])^2; $$
	$$ 1 \leq \norm[\alpha]^2 + \norm[C \beta]^2 = 1+(\norm[C]^2-1) \norm[\beta]^2 \leq 1+\norm[\delta_0]^2 (\norm[C]^2 -1 )(\norm[C]^2 + \norm[\delta_0]^2)^{-1} \leq 1 + \norm[\delta_0]^2. $$
	We conclude both assertions.
\end{proof}
\begin{lemma}
	(\cite[(3.43)]{MV10}) We have as $\Cond(\pi_v) \to \infty$
	$$ \extnorm{ \int_{U_C} f_{3,v}(\kappa) d\kappa } \geq \frac{1}{2} \int_{U_C} \extnorm{ f_{3,v}(\kappa) } d\kappa \asymp \norm[C]_v^{-1/2}. $$
\label{EstIntonK}
\end{lemma}
\begin{proof}
	The first inequality was explained just after \cite[(3.43)]{MV10}. For the second, we apply the second assertion of Lemma \ref{ShrinkNB} and get
	$$ \int_{U_C} \extnorm{ f_{3,v}(\kappa) } d\kappa \asymp \norm[C]_v^{-1/2} \int_{U_C} \extnorm{ f_0(\kappa') } \Ht_v(\kappa a(C)) d\kappa = \norm[C]_v^{-1/2} \int_{\gp{K}} \extnorm{f_0(\kappa)} d\kappa, $$
if we write $\kappa a(C) = b' \kappa'$.
\end{proof}

	Since $\pi_{2,v}=\pi(1,1)$ is unitary and spherical, we are in a simpler situation than \cite[\S 3.6.4 \& 3.6.5]{MV10}, i.e., we can take $f_{2,v}$ to be the spherical function in $\pi(1,1)$ taking value $1$ at $1$. Specify $W_{\varphi,v}$ by taking $W_{\varphi,v}(a(y))$ to be a fixed smooth function $\delta_v(y)$ with support in a compact neighborhood of $1$ in $\F_v^{\times}$, invariant by $\ag{C}^{1}$ if $\F_v=\ag{C}$, such that
	$$ \int_{\F_v^{\times}} \delta_v(y) W_2^*(a(-y)) \norm[y]_v^{-\frac{1}{2}} d^{\times}y \gg 1, \quad \int_{\F_v^{\times}} \norm[\delta_v(y)]^2 d^{\times}y = 1. $$
\begin{lemma}
	With the above choices, we have as $\Cond(\pi_v) \to \infty$
	$$ \extnorm{ \ell_v(W_{\varphi,v}, f_{2,v}, f_{3,v}) } \gg \norm[C]^{-1/2} = \Cond(\pi_v)^{-\frac{1}{2} - \frac{\epsilon}{2}}. $$
\label{LocLowerBdA}
\end{lemma}
\begin{proof}
	The proof is similar to \cite[\S 3.6.5]{MV10}. We drop the subscript $v$ and write $W=W_{\varphi,v}$ for simplicity. Defining a bilinear form
	$$ L(\tilde{W},\tilde{W}_2) = \int_{\F_v^{\times}} \tilde{W}(a(y)) \tilde{W}_2(a(-y)) \norm[y]_v^{-\frac{1}{2}} d^{\times}y, \quad \forall \tilde{W} \in \Whi(\pi,\psi), \tilde{W}_2 \in \Whi(\pi(1,1),\psi), $$
we have for $\varepsilon > 0$ small enough \footnote{In the case $\F_v = \ag{C}$, we should write $\kappa = \kappa_0 \kappa_1$ with $\kappa_0 \in \gp{B} \cap \gp{K}$ and $\kappa_1$ close to $1$. One thus replaces $L(W,W_2^*)$ by $L(\kappa_0.W,W_2^*)$ in the following argument.}
\begin{align*}
	&\quad \extnorm{ L(\kappa.W, \kappa.W_2^*) - L(W,W_2^*) } = \extnorm{ L(\kappa.W - W, W_2^*) } \\
	&\leq \left( \int_{\F^{\times}} \norm[(\kappa.W - W)(a(y))]^2 \norm[y]_v^{-\varepsilon} d^{\times}y \right)^{1/2} \left( \int_{\F^{\times}} \norm[W_2^*(a(y))]^2 \norm[y]_v^{-1+\varepsilon} d^{\times}y \right)^{1/2}.
\end{align*}
	For $\kappa \in U_C$ arguing as in \cite[\S 3.6.4]{MV10}, i.e., using \cite[Proposition 3.2.3]{MV10} and for any $\tilde{W} \in \Whi(\pi^{\infty})$
\begin{align*}
	\int_{\norm[y]_v \leq 1} \norm[\tilde{W}(a(y))]^2 \norm[y]_v^{-\varepsilon} d^{\times}y &\leq \Sob(\tilde{W})^{\frac{\varepsilon}{1/2-\theta}} \int_{\norm[y]_v \leq 1} \norm[\tilde{W}(a(y))]^{2-\frac{\varepsilon}{1/2-\theta}} d^{\times}y \\
	&\ll_{\epsilon} \Sob(\tilde{W})^{\frac{\varepsilon}{1/2-\theta}} \cdot \Norm[\tilde{W}]_2^{2-\frac{\varepsilon}{1-2\theta}},
\end{align*}
we see the above is bounded as (c.f. \cite[\S 2.7]{Wu14})
	$$ \ll_{\varepsilon} (\Cond(\pi)/C)^{(1-d\varepsilon)/[\F:\ag{R}]} \Cond(\pi)^{d'\varepsilon} + \Cond(\pi)/C = \Cond(\pi)^{-\epsilon (1-d\varepsilon)/[\F:\ag{R}] + d'\varepsilon} + \Cond(\pi)^{-\epsilon}, $$
for some absolute $d,d' \in \ag{N}$. If $\varepsilon$ is sufficiently small (depending on $\epsilon, d, d'$), the above tends to $0$ as $\Cond(\pi) \to \infty$. Thus
	$$ \extnorm{ \ell_v(W_{\varphi,v}, f_{2,v}, f_{3,v}) } \geq L(W,W_2^*) \cdot \extnorm{ \int_{\gp{K} \cap \gp{B} \backslash \gp{K}} f_3(\kappa) d\kappa } - o(1) \int_{\gp{K} \cap \gp{B} \backslash \gp{K}} \extnorm{f_3(\kappa)} d\kappa. $$
	We conclude by Lemma \ref{EstIntonK}.
\end{proof}

		\subsubsection{Upper Bounds}

\begin{lemma}
	Let $\Phi \in \Sch(\F_v^{\times})$ and $\chi$ be a (unitary) character of $\F_v^{\times}$ with analytic conductor $C=\Cond(\chi)$. Then for any $1/2 \leq \alpha < 1$, we have
	$$ \extnorm{ \int_{\F_v} \Phi(x) \psi_v(tx) \chi(x) dx } \ll_{\Phi, N, \alpha} \min(C^N (1+\norm[t]_v)^{-N} ,C^{1/2-\alpha} (1+\norm[t]_v)^{\alpha-1}), $$
where the dependence on $\Phi$ involves only some Schwartz norms of $\Phi$ of order depending on $N$.
\label{TateIntEst}
\end{lemma}
\begin{proof}
	This is part of \cite[Lemma 3.1.14]{MV10} or \cite[Lemma 4.1]{Wu14}. Incidentally, we find our previous proofs not clearly written for the second bound. We give a clearer version as follows. Writing $\Phi_{\alpha}(x) = \Phi(x) \norm[x]_v^{\alpha}$ and $\widehat{\Phi}_{\alpha}$ for the Fourier transform of $\Phi_{\alpha}$ w.r.t. $\psi_v$, we have
	$$ \int_{\F_v} \Phi(x) \psi_v(tx) \chi(x) dx = \gamma(\chi,\psi_v,1-\alpha)^{-1} \int_{\F_v} \widehat{\Phi}_{\alpha}(x+t) \chi^{-1}(x) \norm[x]_v^{\alpha-1} dx. $$
	The (inverse) $\gamma$ factor is bounded as $\asymp_{\epsilon} C^{1/2-\alpha}$ uniformly for $\alpha \in [1/2,1-\epsilon]$. If $\norm[t]_v \leq 1$, the integral at the RHS is bounded as
	$$ \extnorm{ \int_{\F_v} \widehat{\Phi}_{\alpha}(x+t) \chi^{-1}(x) \norm[x]_v^{\alpha-1} dx } \leq \int_{\norm[x]_v \leq 1} \norm[x]_v^{\alpha-1} dx \cdot \Norm[\widehat{\Phi}_{\alpha}]_{\infty} + \Norm[\widehat{\Phi}_{\alpha}]_1; $$
while if $\norm[t]_v \geq 1$, the integral at the RHS is bounded as
	$$ \extnorm{ \int_{\F_v} \widehat{\Phi}_{\alpha}(x+t) \chi^{-1}(x) \norm[x]_v^{\alpha-1} dx } \leq \int_{\norm[x]_v \leq \norm[t]_v/2} \norm[x]_v^{\alpha-1} dx \cdot \max_{\norm[x]_v \geq \norm[t]_v/2} \norm[\widehat{\Phi}_{\alpha}(x)] + (\norm[t]_v/2)^{\alpha-1} \Norm[\widehat{\Phi}_{\alpha}]_1. $$
	We conclude by the rapid decay of $\widehat{\Phi}_{\alpha}$, quantifiable in terms of the Schwartz norms of $\Phi$.
\end{proof}
\begin{lemma}
	Write $W_0$ for $W_{f_0}$. Uniformly in $\kappa \in U$ and $y \in \F_v^{\times}$, we have
	$$ \extnorm{W_0(a(y) \kappa)} \ll_{U, \epsilon, N} \left( \frac{\norm[y]_v}{\Cond(\omega)} \right)^{\frac{1}{2} - \epsilon} \left(1+\frac{\Cond(\omega_v)}{\norm[y]_v} \right)^{-N}. $$
\label{Whi0Bd}
\end{lemma}
\begin{proof}
	Introducing the variables
	$$ \kappa = \begin{pmatrix} a & b \\ c & d \end{pmatrix} \in U, \quad X = by+dx, $$
we can write the LHS as
\begin{align*}
	W_0(a(y) \kappa) &= \psi_v(-\frac{by}{d}) \frac{\norm[y]_v^{1/2}}{\norm[d]_v} \int_{\F_v^{\times}} d^{\times}t \int_{\F_v} \phi(\frac{c}{d}X + \frac{t}{d}y) \phi(X-1) \psi_v(\frac{X}{td}) \omega_v^{-1}(X) dX \\
	&= \psi_v(-\frac{by}{d}) \frac{\norm[y]_v^{1/2}}{\norm[d]_v} \int_{\F_v^{\times}} d^{\times}t \int_{\F_v} \phi(\frac{c}{d}X + t) \phi(X-1) \psi_v(\frac{X}{td^2/y}) \omega_v^{-1}(X) dX.
\end{align*}
	The inner integral is non-vanishing only if $X$ is close to $1$ hence $\norm[t]_v \leq \delta$ for some $\delta$ depending only on $\delta_0, U$. Applying Lemma \ref{TateIntEst} to $\Phi(X) := \phi(\frac{c}{d}X + t) \phi(X-1), \alpha = 1-\epsilon$, the inner integral is bounded by
	$$ \left( \frac{\Cond(\omega_v)}{1+\norm[y/(td^2)]_v} \right)^N \ll_{U,N} \left( \frac{\Cond(\omega)_v}{\norm[y]_v} \right)^N, \quad \text{resp.} \quad \Cond(\omega_v)^{-1/2+\epsilon} \extnorm{ \frac{y}{d^2t} }_v^{-\epsilon} \ll_U \Cond(\omega_v)^{-1/2+\epsilon} \extnorm{ \frac{t}{y} }_v^{\epsilon} $$
with implied constant depending only on the Schwartz norms of $\phi$. We conclude.
\end{proof}
\begin{remark}
	The reason for which $W_0$ satisfies a ``better'' bound than the general one satisfied by a smooth vector shows some finer aspects of the integral representation of Whittaker functions than the smooth structures.
\end{remark}
\begin{corollary}
	With assumptions as in Lemma \ref{Whi0Bd} and $\norm[C]_v$ sufficiently large (depending on $U$), we have uniformly in $\kappa \in U_C$ and $y$
	$$ \extnorm{W_0(a(y) \kappa a(C))} \ll_{U, \epsilon, M} \left( \frac{\norm[C] \norm[y]}{\Cond(\omega)} \right)^{\frac{1}{2} - \epsilon} \left( 1 + \frac{ \norm[C] \norm[y]}{\Cond(\omega)} \right)^{-M}. $$
\label{WhiUniBd}
\end{corollary}
\begin{proof}
	The general case follows from the case $C=1$ and the ``moreover'' part of Lemma \ref{ShrinkNB}. For $C=1$, we apply Lemma \ref{Whi0Bd}.
\end{proof}
\begin{lemma}
	(1) Let $\pi_v'$ be a unitary irreducible representation with trivial central character and spectral parameter $\leq \theta$. Let $W'=W_{\varphi',v}$ be the Whittaker function of a unitary vector in $\pi_v'$. Let $\ell_v$ be defined in (\ref{CuspPerD}). With absolute implicit constant, we have
	$$ \extnorm{ \ell_v(W', \overline{f_{3,v}}, f_{3,v}) } \ll_{\epsilon} \Cond(\pi_v)^{-\frac{1}{2}} (\Cond(\pi_v) / \Cond(\omega_v))^{\theta + \epsilon} \cdot \Sob_d(W') $$
for some Sobolev norm of an absolute order $d$.

\noindent (2) Let $\xi_v$ be a unitary character of $\F_v^{\times}$. Let $\Phi_v$ be a smooth function in $\pi(\xi_v,\xi_v^{-1})$. Let $\ell_v$ be defined in (\ref{EisPerD}). With absolute implicit constant, we have for $\tau \in \ag{R}$
	$$ \extnorm{ \ell_v(i\tau; \Phi_v, \overline{f_{3,v}}, f_{3,v}) } \ll_{\epsilon} \Cond(\pi_v)^{-\frac{1}{2}+\epsilon} \cdot \Sob_d(\Phi_{v,i\tau}) $$
for some Sobolev norm of an absolute order $d$.

\noindent (3) Let $\ell_v$ be defined in (\ref{OneDimProj}). With absolute implicit constant, we have for any $k \in \ag{N}$ and $\epsilon > 0$
	$$ \extnorm{ \ell_v^{(k)}(1/2;\overline{f_{3,v}},f_{3,v},1) } \ll_{k,\epsilon} \Cond(\pi_v)^{\epsilon}. $$
\label{LocUpperBdA}
\end{lemma}
\begin{proof}
	(1) The proof is the same as the one given in \cite[Corollary 3.7.1]{MV10}, using Lemma \ref{EstIntonK}, \cite[Proposition 3.2.3]{MV10} (which gives $d$) and the inequality
\begin{align*}
	\int_{\F_v} \frac{(X\norm[y]_v)^{\frac{1}{2}-\epsilon}}{(1+X\norm[y]_v)^N} \cdot \frac{\norm[y]_v^{-\theta}}{(1+\norm[y]_v)^M} d^{\times} y &= X^{\theta + \epsilon} \int_{\F_v} \frac{(X\norm[y]_v)^{\frac{1}{2}-\theta -2\epsilon}}{(1+X\norm[y]_v)^N} \cdot \frac{\norm[y]_v^{\epsilon}}{(1+\norm[y]_v)^M} d^{\times} y \\
	&\leq X^{\theta + \epsilon} \left( \int_{\F_v} \frac{\norm[y]_v^{1-2\theta -4\epsilon}}{(1+\norm[y]_v)^{2N}} d^{\times}y \right)^{\frac{1}{2}} \left( \int_{\F_v} \frac{\norm[y]_v^{2\epsilon}}{(1+\norm[y]_v)^{2M}} d^{\times}y \right)^{\frac{1}{2}}.
\end{align*}

\noindent (2) Since $\tau \in \ag{R}$, Proposition \ref{TransFLocRS} tells us that $\ell_v$ and $\tilde{\ell}_v$ are of the same size. We argue as above for $\tilde{\ell}_v$.

\noindent (3) Drop subscript $v$. Recall $f_3 = a(C).f_0$ for some $C \in \F_v, \norm[C]_v \asymp_{\epsilon} \Cond(\pi_v)^{1+\epsilon}$. Let $\Phi \in \pi(1, 1)$ be the spherical function taking value $1$ on $\gp{K}$. We find
	$$ \ell_v(s; \overline{f_{3,v}},f_{3,v},1) = \int_{\F \times \gp{K}} \extnorm{W_0(a(y)\kappa a(C))}^2 \Phi_s(a(y)\kappa) \norm[y]_v^{-1} d^{\times}y d\kappa. $$
	Writing $\kappa a(C) = b' \kappa'$, we get
	$$ \ell_v(s; \cdots) = \int_{\gp{K}_v} \int_{\F_v^{\times}} \extnorm{W_0(a(y)\kappa')}^2 \norm[y]_v^{s-1/2} d^{\times}y {\rm Ht}(\kappa a(C))^{1/2-s} d\kappa. $$
	Note that uniformly in $\kappa \in \gp{K}_v$
	$$ C^{-1} \leq {\rm Ht}(\kappa a(C)) \leq C \quad \Rightarrow \quad \extnorm{ \log (\cdots) } \leq \log C . $$
	Arguing as in the proof of Lemma \ref{LocLowerBdA}, we also have for any $k \in \ag{N}, \epsilon > 0$
\begin{align*}
	\extnorm{ \int_{\norm[y]_v \leq 1} \extnorm{W_0(a(y)\kappa')}^2 (\log \norm[y]_v)^k d^{\times}y } &\ll \Sob(\kappa'.W_0)^{2\epsilon} \int_{\norm[y]_v \leq 1} \extnorm{W_0(a(y)\kappa')}^{2-2\epsilon} d^{\times}y \cdot \sup_{\norm[y]_v \leq 1} \norm[y]_v^{\epsilon} \extnorm{ \log \norm[y]_v }^k  \\
	&\ll_{k,\epsilon} \Sob(W_0)^{2\epsilon} \cdot \Norm[W_0]_2^{2-2\epsilon} \ll \Cond(\omega_v)^{d\epsilon},
\end{align*}
\begin{align*}
	\extnorm{ \int_{\norm[y]_v \geq 1} \extnorm{W_0(a(y)\kappa')}^2 (\log \norm[y]_v)^k d^{\times}y } &\ll_{k,\epsilon} \int_{\norm[y]_v \geq 1} \extnorm{W_0(a(y)\kappa')}^2 \norm[y]_v^{\epsilon} d^{\times}y \\
	&\leq \Sob(W_0)^{\epsilon} \cdot \Norm[W_0]_2^{1-\epsilon} \ll \Cond(\omega_v)^{d'\epsilon};
\end{align*}
	for some absolute constants $d,d'$. We deduce and conclude by
	$$ \extnorm{ \ell_v^{(k)}(1/2; \cdots) } \ll_{k,\epsilon} \Cond(\omega_v)^{\epsilon} (\log C)^k. $$
\end{proof}

		\subsubsection{Main Bounds in The Regularized Terms}
		
\begin{lemma}
\begin{itemize}
	\item[(1)] For $\ell_v$ in (\ref{Regu1PerD}), $f_v \in \{ (f_{3,v} \cdot \overline{f_{3,v}}) \mid_{\gp{K}_v}, (\IntwR_0 f_{3,v} \cdot \overline{ \IntwR_0 f_{3,v} }) \mid_{\gp{K}_v} \}$ and any $k \in \ag{N}$, we have 
	$$ \extnorm{\ell_v^{(k)}(0; f_{2,v}, \overline{f_{2,v}}; f_v)} \ll_k 1. $$
	\item[(2)] For $\ell_v$ in (\ref{Regu2PerD}) and $f_v \in \{ (\IntwR_0 f_{3,v} \cdot \overline{f_{3,v}}) \mid_{\gp{K}_v}, (f_{3,v} \cdot \overline{ \IntwR_0 f_{3,v} }) \mid_{\gp{K}_v} \}$, $\ell_{v}(s; f_{2,v}, \overline{f_{2,v}}; f_v)$ is non-vanishing only if $\omega_v$ is trivial on $\vo_v^{\times}$ ($=\{ \pm 1 \}$ if $\F_v = \ag{R}$, $= \ag{C}^{(1)}$ if $\F_v = \ag{C}$). In this case we have for any $k \in \ag{N}$
	$$ \extnorm{\ell_v^{(k)}(0; f_{2,v}, \overline{f_{2,v}}; f_v)} \ll_k 1. $$
	\item[(3)] For $\ell_v$ in (\ref{Regu3PerD}) and $f_v = 1$, we have for any $k \in \ag{N}$ and $\epsilon > 0$
	$$ \extnorm{ \ell_v^{(k)}(0; f_{3,v}, \overline{f_{3,v}}; 1) } \ll_{k,\epsilon} \Cond(\pi_v)^{\epsilon}. $$
\end{itemize}
\label{ReguAMainBd}
\end{lemma}
\begin{proof}
	Drop the subscript $v$ for simplicity. Since $f_2$ is $\gp{K}$-invariant, the integral part of $\ell$ in (1) resp. (2) has the form
	$$ \int_{\F^{\times}} \extnorm{W_2^*(a(y))}^2 \cdot \norm[y]^s d^{\times}y \int_{\gp{K}} f(\kappa) d\kappa, \quad \text{resp.} \quad \int_{\F^{\times}} \extnorm{W_2^*(a(y))}^2 \cdot \omega^{\mp}(y)\norm[y]^s d^{\times}y \int_{\gp{K}} f(\kappa) d\kappa. $$
	All assertions follow taking into account
	$$ \int_{\F^{\times}} \extnorm{W_2^*(a(y))}^2 \cdot \extnorm{\log^k \norm[y]} d^{\times}y \ll_k 1, \quad \int_{\gp{K}} \extnorm{f_3(\kappa)}^2 d\kappa = \Norm[a(C).f_0]^2 = \Norm[f_0]^2 = 1, $$
	$$ \int_{\gp{K}} \extnorm{\IntwR_0 f_3(\kappa)}^2 d\kappa = \Norm[\IntwR_0 f_3] = \Norm[f_3] = 1, \quad \extnorm{\int_{\gp{K}} f_3(\kappa) \overline{\IntwR_0 f_3(\kappa)} d\kappa}^2 \leq \int_{\gp{K}} \extnorm{f_3(\kappa)}^2 d\kappa \cdot \int_{\gp{K}} \extnorm{\IntwR_0 f_3(\kappa)}^2 d\kappa. $$
	(3) is the same as Lemma \ref{LocUpperBdA} (3).
\end{proof}

\section{Global Estimations}

	\subsection{Regular Term}
	\label{MainRT}

	Applying the Fourier inversion to the first factor of the integrand in (\ref{RegTerm}), we get
\begin{align}
	(\ref{RegTerm}) &= \sum_{\pi' \leq \mathfrak{t}} \sum_{\varphi' \in \Bas(\pi')} C(\varphi') \int_{[\PGL_2]} \varphi' \cdot \overline{\eis_3^{\sharp}} \cdot a(\mathfrak{t}).\eis_3^{\sharp} + \label{CuspRegTerm} \\
	&\quad \sum_{\xi} \sum_{\Phi \in \Bas(\xi,\xi^{-1})} \int_{\ag{R}} C(i\tau, \Phi) \int_{[\PGL_2]} \eis(i\tau, \Phi) \cdot \left( a(\mathfrak{t}).\eis_3^{\sharp} \cdot \overline{\eis_3^{\sharp}} - \Reis(a(\mathfrak{t}).\eis_3^{\sharp} \overline{\eis_3^{\sharp}}) \right) + \label{EisRegTerm} \\
	&\quad \sum_{\chi: \chi^2=1} \int_{[\PGL_2]} \left( a(\mathfrak{t}).\eis_2^* \cdot \overline{\eis_2^*} - \Reis(a(\mathfrak{t}).\eis_2^* \overline{\eis_2^*}) \right) \cdot \overline{\chi \circ \det} \cdot \nonumber \\
	&\quad \int_{[\PGL_2]} \left( a(\mathfrak{t}).\eis_3^{\sharp} \cdot \overline{\eis_3^{\sharp}} - \Reis(a(\mathfrak{t}).\eis_3^{\sharp} \overline{\eis_3^{\sharp}}) \right) \cdot \chi \circ \det, \label{OneDRegTerm}
\end{align}
with the Fourier coefficients
	$$ C(\varphi') := \Pairing{a(\mathfrak{t}).\eis_2^* \cdot \overline{\eis_2^*}}{\varphi'}_{[\PGL_2]}, \quad C(i\tau, \Phi) = \Pairing{a(\mathfrak{t}).\eis_2^* \cdot \overline{\eis_2^*} - \Reis(a(\mathfrak{t}).\eis_2^* \overline{\eis_2^*})}{\eis(i\tau, \Phi)}_{[\PGL_2]}. $$
	We estimate the cuspidal contribution (\ref{CuspRegTerm}), the Eisenstein contribution (\ref{EisRegTerm}) and the one-dimensional contribution (\ref{OneDRegTerm}) one by one and get
\begin{lemma}
	The contribution of the regular term (\ref{RegTerm}) is bounded as
	$$ (\Cond(\pi) K)^{\epsilon} \max \left( (\frac{\Cond(\pi)}{\Cond(\omega)})^{-\frac{1}{2}+\theta} \Cond(\omega)^{-\delta'} K^{\Norm[\mathfrak{t}](A+\frac{1}{2})}, (\frac{\Cond(\pi)}{\Cond(\omega)})^{-\frac{1}{2}} \Cond(\omega)^{-\frac{1-2\theta}{8}} K^{\Norm[\mathfrak{t}] ( \frac{1}{2}-\frac{1-2\theta}{8} ) }, K^{-\Norm[\mathfrak{t}]} \right). $$
\label{RegTermBd}
\end{lemma}
\begin{proof}
	The estimation follows from the last line of each subsequent subsubsection.
\end{proof}

		\subsubsection{Cuspidal Contribution}

	By (\ref{CuspPerD}), Lemma \ref{LocUpperBdNA} (1), Lemma \ref{LocAmpBd} (1), Lemma \ref{LocUpperBdA} (1), together with \cite{HL94} and \cite[Lemma 3]{BH10} giving bounds for $L(1,\pi' \times \overline{\pi'})$, we get
	$$ \extnorm{ \int_{[\PGL_2]} \varphi' \cdot \overline{\eis_3^{\sharp}} \cdot a(\mathfrak{t}).\eis_3^{\sharp} } \ll_{\epsilon} \Cond(\pi)^{-\frac{1}{2}} (\Cond(\pi) / \Cond(\omega))^{\theta + \epsilon} \cdot \Sob_d(\varphi') \cdot \extnorm{L(\frac{1}{2},\pi') L(\frac{1}{2},\pi' \otimes \omega)} \cdot K^{-\frac{\Norm[\mathfrak{t}]}{2}+\epsilon}. $$
	Inserting (\ref{TwistSubAssump}), summing over $\varphi'$ and $\pi'$ like in \cite[(6.16)]{Wu14}, applying \cite[Corollary 6.7]{Wu14} and \cite[Theorem 5.4]{Wu2} (to obtain $\Sob_{d'}(a(\mathfrak{t}).\eis_2^* \cdot \overline{\eis_2^*} - \Reis(a(\mathfrak{t}).\eis_2^*\overline{\eis_2^*})) \ll \log^3 K$ for some $d' > d$), we get
	$$ (\ref{CuspRegTerm}) \ll_{\epsilon} (\Cond(\pi) K)^{\epsilon} (\Cond(\pi) / \Cond(\omega))^{-1/2+\theta} \Cond(\omega)^{-\delta'} K^{\Norm[\mathfrak{t}](A+1/2)}. $$

		\subsubsection{Eisenstein Contribution}

	By (\ref{EisPerD}), Lemma \ref{LocUpperBdNA} (2), Lemma \ref{LocAmpBd} (2), Lemma \ref{LocUpperBdA} (2), together with Siegel's lower bound, we get
\begin{align*}
	&\quad \int_{[\PGL_2]} \eis(i\tau, \Phi) \cdot \left( a(\mathfrak{t}).\eis_3^{\sharp} \cdot \overline{\eis_3^{\sharp}} - \Reis(a(\mathfrak{t}).\eis_3^{\sharp} \overline{\eis_3^{\sharp}}) \right) \\
	&\ll_{\epsilon} \Cond(\pi)^{-\frac{1}{2}+\epsilon} \cdot (1+\norm[\tau])^{\epsilon} S_d(\Phi_{i\tau}) \cdot \extnorm{L(\frac{1}{2}+i\tau, \xi)^2 L(\frac{1}{2}+i\tau, \xi\omega) L(\frac{1}{2}+i\tau, \xi\omega^{-1})} \cdot K^{-\Norm[\mathfrak{t}]/2+\epsilon}.
\end{align*}
	Inserting \cite[Theorem 1.1]{Wu2}, summing over $\Phi$ and $\xi$ like in \cite[\S 6.4]{Wu14}, applying the analogue of \cite[Corollary 6.7]{Wu14} for the fourth moment bound of Hecke $L$-functions and \cite[Theorem 5.4]{Wu2}, we get
	$$ (\ref{EisRegTerm}) \ll_{\epsilon} (\Cond(\pi) K)^{\epsilon} (\Cond(\pi) / \Cond(\omega))^{-1/2} \Cond(\omega)^{-(1-2\theta)/8} K^{\Norm[\mathfrak{t}] \{ 1/2-(1-2\theta)/8 \} }. $$

		\subsubsection{One-dimensional Contribution}

	Applying Remark \ref{OneDimProjRk} with $\omega=1$, we see that
	$$ \int_{[\PGL_2]} \left( a(\mathfrak{t}).\eis_2^* \cdot \overline{\eis_2^*} - \Reis(a(\mathfrak{t}).\eis_2^* \overline{\eis_2^*}) \right) \cdot \overline{\chi \circ \det} = \int_{[\PGL_2]}^{\reg} a(\mathfrak{t}).\eis_2^* \cdot \overline{\eis_2^*} \cdot \overline{\chi \circ \det} $$
is non-vanishing only if $\chi=1$ (this is different from the cuspidal case where $\pi_2 \otimes \chi \simeq \pi_2$ is possible for non-trivial $\chi$). If $\mathfrak{t}=\varpi_0^{-j} \varpi_{\vp_1}^{-n_{\vp_1}} \varpi_{\vp_2}^{-n_{\vp_2}}$, denote by $T(\mathfrak{t}) = T(\vp_0^j)^* T(\vp_1^{n_{\vp_1}})^* T(\vp_2^{n_{\vp_2}})^*$ the corresponding product of Hecke operators. The $\gp{K}$-invariance of regularized integrals \cite[Proposition 2.27]{Wu9} implies
	$$ \extnorm{ \int_{[\PGL_2]}^{\reg} a(\mathfrak{t}).\eis_2^* \cdot \overline{\eis_2^*} } = \extnorm{ \int_{[\PGL_2]}^{\reg} T(\mathfrak{t}).\eis_2^* \cdot \overline{\eis_2^*} } \ll_{\epsilon} K^{-\frac{\Norm[\mathfrak{t}]}{2}+\epsilon} \extnorm{ \int_{[\PGL_2]}^{\reg} \eis_2^* \cdot \overline{\eis_2^*} } \ll_{\F,\epsilon} K^{-\frac{\Norm[\mathfrak{t}]}{2}+\epsilon}. $$
	A similar argument leads to
	$$ \extnorm{ \int_{[\PGL_2]} a(\mathfrak{t}).\eis_3^{\sharp} \cdot \overline{\eis_3^{\sharp}} - \Reis(a(\mathfrak{t}).\eis_3^{\sharp} \overline{\eis_3^{\sharp}}) } \ll_{\epsilon} K^{-\frac{\Norm[\mathfrak{t}]}{2}+\epsilon} \extnorm{ \int_{[\PGL_2]}^{\reg} \eis_3^{\sharp} \cdot \overline{\eis_3^{\sharp}} }. $$
	By (\ref{OneDimProj}), Lemma \ref{LocUpperBdNA} (3), Lemma \ref{LocUpperBdA} (3), together with $\extnorm{L(1,\omega^{\pm 1})} \ll_{\epsilon} \Cond(\omega)^{\epsilon}$, we get
	$$ (\ref{OneDRegTerm}) \ll_{\epsilon} K^{-\Norm[\mathfrak{t}]} \Cond(\pi)^{\epsilon}. $$

	\subsection{Regularized Term}
	
\begin{lemma}
	In (\ref{ReguTerm}), we have
	$$ \extnorm{\int_{[\PGL_2]}^{\reg} a(\mathfrak{t}).\eis_2^* \cdot \overline{\eis_2^*} \cdot \Reis(a(\mathfrak{t}).\eis_3^{\sharp} \overline{\eis_3^{\sharp}})} \ll_{\F, \epsilon} (\Cond(\pi) K)^{\epsilon} K^{-\Norm[\mathfrak{t}]}. $$
\label{ReguTermBd1}
\end{lemma}
\begin{proof}
	(1) First suppose $\omega \neq 1$. Write $u_{\infty} = L_{\infty}(1,\omega^{-1}) / L_{\infty}(1,\omega)$ and notice that $\norm[u_{\infty}]=1$. We have
\begin{align*}
	\Reis(a(\mathfrak{t}).\eis_3^{\sharp} \overline{\eis_3^{\sharp}}) &= \extnorm{L(1,\omega)}^2 \eis^{\reg}(\frac{1}{2}, a(\mathfrak{t}).f_3 \overline{f_3}) + \extnorm{L(1,\omega^{-1})}^2 \eis^{\reg}(\frac{1}{2}, a(\mathfrak{t}).\IntwR_0 f_3 \overline{\IntwR_0 f_3}) \\
	&+ \left\{ \begin{matrix} u_{\infty} L(1,\omega^{-1})^2 \eis(\frac{1}{2}, a(\mathfrak{t}).\IntwR_0 f_3 \overline{f_3}) + \overline{u_{\infty}} L(1,\omega)^2 \eis(\frac{1}{2}, a(\mathfrak{t}).f_3 \overline{\IntwR_0  f_3}) & \text{if } \omega^2 \neq 1 \\ u_{\infty} L(1,\omega^{-1})^2 \eis^{\reg}(\frac{1}{2}, a(\mathfrak{t}).\IntwR_0 f_3 \overline{f_3}) + \overline{u_{\infty}} L(1,\omega)^2 \eis^{\reg}(\frac{1}{2}, a(\mathfrak{t}).f_3 \overline{\IntwR_0  f_3}) & \text{if } \omega^2 = 1 \end{matrix} \right..
\end{align*}
	We shall apply Proposition \ref{TPFReguTerm} to treat
	$$ \begin{matrix} \int_{[\PGL_2]}^{\reg} a(\mathfrak{t}).\eis_2^* \cdot \overline{\eis_2^*} \cdot \eis(\frac{1}{2}, a(\mathfrak{t}).\IntwR_0 f_3 \overline{f_3}) \quad \text{resp.} \quad \int_{[\PGL_2]}^{\reg} a(\mathfrak{t}).\eis_2^* \cdot \overline{\eis_2^*} \cdot \eis(\frac{1}{2}, a(\mathfrak{t}).f_3 \overline{\IntwR_0  f_3}) & \text{if } \omega^2 \neq 1, \\ \int_{[\PGL_2]}^{\reg} a(\mathfrak{t}).\eis_2^* \cdot \overline{\eis_2^*} \cdot \eis^{\reg}(\frac{1}{2}, a(\mathfrak{t}).\IntwR_0 f_3 \overline{f_3}) \quad \text{resp.} \quad \int_{[\PGL_2]}^{\reg} a(\mathfrak{t}).\eis_2^* \cdot \overline{\eis_2^*} \cdot \eis^{\reg}(\frac{1}{2}, a(\mathfrak{t}).f_3 \overline{\IntwR_0  f_3}) & \text{if } \omega^2 = 1. \end{matrix} $$
	In fact, (\ref{Regu2PerD}), Lemma \ref{ReguNAMainBd} (2), Lemma \ref{ReguAmpBd} (2) and Lemma \ref{ReguAMainBd} (2) imply
	$$ \extnorm{R^{\hol}(1/2, a(\mathfrak{t}).\eis_2^* \cdot \overline{\eis_2^*}; a(\mathfrak{t}).\IntwR_0 f_3 \overline{f_3})}, \text{ resp. } \extnorm{R^{\hol}(1/2, a(\mathfrak{t}).\eis_2^* \cdot \overline{\eis_2^*}; a(\mathfrak{t}).f_3 \overline{\IntwR_0  f_3})} \ll_{\F,\epsilon} K^{-\Norm[\mathfrak{t}]+\epsilon}. $$
	We shall apply \cite[Theorem 2.7]{Wu2} to treat
	$$ \int_{[\PGL_2]}^{\reg} a(\mathfrak{t}).\eis_2^* \cdot \overline{\eis_2^*} \cdot \eis^{\reg}(\frac{1}{2}, a(\mathfrak{t}).f_3 \overline{f_3}) \quad \text{resp.} \quad \int_{[\PGL_2]}^{\reg} a(\mathfrak{t}).\eis_2^* \cdot \overline{\eis_2^*} \cdot \eis^{\reg}(\frac{1}{2}, \IntwR_0 a(\mathfrak{t}).f_3 \overline{\IntwR_0 f_3}). $$
	Combining (\ref{Regu1PerD}), Lemma \ref{ReguNAMainBd} (1), Lemma \ref{ReguAmpBd} (1) and Lemma \ref{ReguAMainBd} (1) we get
	$$ \extnorm{R^{\hol}(1/2, a(\mathfrak{t}).\eis_2^* \cdot \overline{\eis_2^*}; a(\mathfrak{t}).f_3 \overline{f_3})}, \text{ resp. } \extnorm{R^{\hol}(1/2, a(\mathfrak{t}).\eis_2^* \cdot \overline{\eis_2^*}; a(\mathfrak{t}).\IntwR_0 f_3 \overline{\IntwR_0 f_3})} \ll_{\F,\epsilon} K^{-\Norm[\mathfrak{t}]+\epsilon}. $$
	The bounds of the remaining terms corresponding to \cite[Theorem 2.7]{Wu2} follow from
	$$ \extnorm{\Proj_{\gp{K}}(a(\mathfrak{t}).f_2 \overline{f_2})} \ll_{\epsilon} K^{-\frac{\Norm[\mathfrak{t}]}{2} + \epsilon}, \quad \extnorm{\Proj_{\gp{K}}(a(\mathfrak{t}).f_3 \overline{f_3})} \text{ resp. } \extnorm{\Proj_{\gp{K}}(a(\mathfrak{t}).\IntwR_0 f_3 \overline{\IntwR_0 f_3})} \ll_{\epsilon} K^{-\frac{\Norm[\mathfrak{t}]}{2} + \epsilon}; $$
\begin{equation}
	\extnorm{\Proj_{\gp{K}}(\widetilde{\Intw}_{\frac{1}{2}}^{(k)}\left( a(\mathfrak{t}).f_2 \right) \cdot a(\mathfrak{t}).f_3 \cdot \overline{f_3})} \text{ resp. } \extnorm{\Proj_{\gp{K}}(\widetilde{\Intw}_{\frac{1}{2}}^{(k)}\left( a(\mathfrak{t}).f_2 \right) \cdot a(\mathfrak{t}).\IntwR_0 f_3 \overline{\IntwR_0 f_3})} \ll_{\epsilon} K^{-\Norm[\mathfrak{t}] + \epsilon},
\label{DT}
\end{equation}
	where $0 \leq k \leq 3$. The bounds in the first line are easy consequences of the general matrix coefficients decay \cite[Theorem 2]{CHH88}, MacDonald's formula \cite[Theorem 4.6.6]{Bu98} and the unitarity of $\IntwR_0$. For (\ref{DT}), we first note that we can assume $n(\mathfrak{t}) \leq \vec{0}$, since with $w_{\vp}$ the Weyl element at $\vp \in S$ we have
	$$ \Proj_{\gp{K}}(\widetilde{\Intw}_{\frac{1}{2}}^{(k)}\left( a(\mathfrak{t}).f_2 \right) \cdot a(\mathfrak{t}).f_3 \cdot \overline{f_3}) = \Proj_{\gp{K}}(\widetilde{\Intw}_{\frac{1}{2}}^{(k)}\left( w_{\vp}a(\mathfrak{t})w_{\vp}^{-1}.f_2 \right) \cdot w_{\vp}a(\mathfrak{t})w_{\vp}^{-1}.f_3 \cdot \overline{f_3}). $$
	Extracting the components at $\vp \in S^*$ (\ref{AmpPlace}) and distinguishing elements related to $f_2$ from those to $f_3$ by putting a ``$*$'' (for example, $f_3^{S^*} \otimes e_{\vec{0}}=f_3, e_{\vec{0}}^{S^*} \otimes e_{\vec{0}}^*=f_2$), Proposition \ref{LocBaseRel} (4-1) shows that
	$$ a(\mathfrak{t}).f_2 = e_{\vec{0}}^{S^*} \otimes a(\mathfrak{t}).e_{\vec{0}}^* = \sideset{}{_{\vec{l} \leq -n(\mathfrak{t})}} \sum O_{\epsilon}(K^{-\frac{\Norm[-n(\mathfrak{t}) - \vec{l}]}{2}+\epsilon}) e_{\vec{0}}^{S^*} \otimes e_{\vec{l}}^*, $$
	$$ a(\mathfrak{t}).f_3 = f_3^{S^*} \otimes a(\mathfrak{t}).e_{\vec{0}} = \sideset{}{_{\vec{l} \leq -n(\mathfrak{t})}} \sum O_{\epsilon}(K^{-\frac{\Norm[-n(\mathfrak{t}) - \vec{l}]}{2}+\epsilon}) f_3^{S^*} \otimes e_{\vec{l}}. $$
	By \cite[Lemma 3.18 (4)]{Wu5} or simply \cite[Lemma 4.4 (1)]{Wu2}, we have
	$$ \widetilde{\Intw}_{\frac{1}{2}}^{(k)} (e_{\vec{0}}^{S^*} \otimes e_{\vec{l}}^*) = O_{\F,\epsilon}(K^{-\Norm[\vec{l}]+\epsilon}) \mathbbm{1}_{\vec{l} \neq \vec{0}} \cdot e_{\vec{0}}^{S^*} \otimes e_{\vec{l}}^*. $$
	Although the $\gp{K}_{S^*}$-isotypic vectors $e_{\vec{l}}$ and $e_{\vec{l}}^*$ belong to different representations, Proposition \ref{LocBaseRel} (4-2) implies that their restriction to $\gp{K}_{S^*}$ are the same real function. We deduce that
	$$ \Proj_{\gp{K}}(\widetilde{\Intw}_{\frac{1}{2}}^{(k)}\left( a(\mathfrak{t}).f_2 \right) \cdot a(\mathfrak{t}).f_3 \cdot \overline{f_3}) = \sideset{}{_{\vec{0} \neq \vec{l} \leq -n(\mathfrak{t})}} \sum O_{\F,\epsilon}(K^{-\Norm[\mathfrak{t}]+\epsilon}) \Proj_{\gp{K}}^{S^*}(f_3^{S^*} \overline{f_3^{S^*}}) = O_{\F,\epsilon}(K^{-\Norm[\mathfrak{t}]+\epsilon}). $$
	A similar argument applies to $\Proj_{\gp{K}}(\widetilde{\Intw}_{\frac{1}{2}}^{(k)}\left( a(\mathfrak{t}).f_2 \right) \cdot a(\mathfrak{t}).\IntwR_0 f_3 \overline{\IntwR_0 f_3})$ and we obtain (\ref{DT}).
	
\noindent (2) For the case $\omega=1$, we have
\begin{align*}
	\Reis(a(\mathfrak{t}).\eis_3^{\sharp} \overline{\eis_3^{\sharp}}) &= \norm[\zeta^*(1)]^2 \left\{ \eis^{\reg,(2)}(\frac{1}{2}, a(\mathfrak{t}).f_3 \overline{f_3}) + \frac{1}{4} \eis^{\reg}(\frac{1}{2}, \Intw_0^{(1)} a(\mathfrak{t}).f_3 \overline{\Intw_0^{(1)} f_3}) \right. \\
	&\left. + \frac{1}{2} \eis^{\reg,(1)}(\frac{1}{2}, a(\mathfrak{t}).f_3 \overline{\Intw_0^{(1)} f_3}) + \frac{1}{2} \eis^{\reg,(1)}(\frac{1}{2}, \Intw_0^{(1)} a(\mathfrak{t}).f_3 \overline{f_3}) \right\}.
\end{align*}
	Hence, the extra difficulty is the analysis of $\Intw_0^{(1)}$, given in the next lemma. It follows that
	$$ \extNorm{\IntwR_0^{(1)} f_3} \ll_{\epsilon} \Cond(\pi)^{\epsilon}, \quad \IntwR_0^{(1)} (f_3^{S^*} \otimes e_{\vec{l}}) = O_{\epsilon}(K^{\epsilon}) \cdot \IntwR_0^{S^*,(1)} f_3^{S^*} \otimes e_{\vec{l}}. $$
	The argument of (1) can thus be easily adapted, using inequalities like
	$$ \extnorm{\Proj_{\gp{K}}^{S^*}(\IntwR_0^{S^*,(1)} f_3^{S^*} \overline{f_3^{S^*}})} \leq \Norm[\IntwR_0^{S^*,(1)} f_3^{S^*}] \cdot \Norm[f_3^{S^*}], \quad \extnorm{\Proj_{\gp{K}}^{S^*}(\IntwR_0^{S^*,(1)} f_3^{S^*} \overline{\IntwR_0^{S^*,(1)} f_3^{S^*}})} = \Norm[\IntwR_0^{S^*,(1)} f_3^{S^*}]^2. $$
\end{proof}
\begin{lemma}
	Decompose $\pi(1,1) = \pi_{\infty} \otimes (\otimes_{\vp < \infty} \pi_{\vp})$ and let $\Sob_*$ be a Sobolev norm system involving only the differential operators of $\gp{K}_{\infty}$. Write $\Casimir_{\gp{K},\infty}$ for the Casimir element of $\gp{K}_{\infty}$.
\begin{itemize}
	\item[(1)] For any $f_{\infty} \in \pi_{\infty}^{\infty}$, $C = (C_v)_v \in \F_{\infty}^{\times}$ we have
	$$ \extNorm{ \IntwR_{0,\infty}^{(1)} f_{\infty} } \ll_{\epsilon} \extNorm{ \Casimir_{\gp{K},\infty} f_{\infty} }^{\epsilon} \cdot \extNorm{ f_{\infty} }^{1-\epsilon}, \quad \extNorm{ \Casimir_{\gp{K},\infty} a(C).f_{\infty} } \ll_{\F} \sideset{}{_{v \mid \infty}}\prod \max\{ \norm[C_v]_v^2, \norm[C_v]_v^{-2} \} \cdot \Sob_2(f_{\infty}); $$
	\item[(2)] For $\vec{n} = (n_{\vp})_{\vp}$ with $n_{\vp} \geq 0$ \& $n_{\vp} \neq 0$ for only finitely many $\vp$, and $e_{\vec{n}} = \otimes_{\vp} e_{n_{\vp}} $ with local elements defined in Section \ref{BaseGNV}, we have
	$$ \IntwR_{0,\fin}^{(1)} e_{\vec{n}} = O(\sideset{}{_{\vp}} \sum n_{\vp} \log q_{\vp} ) \cdot e_{\vec{n}}. $$
\end{itemize}
\label{IntwDerEst}
\end{lemma}
\begin{proof}
	(2) is a direct consequence of \cite[Lemma 4.4 (1)]{Wu2}, which admits a similar version at $\infty$. Namely, if we write $e_{\vec{m}}$ for $\gp{K}_{\infty}$-isotypic unitary vectors with the natural numeration given in \cite[\S 4.2]{Wu2}, then we have
	$$ \IntwR_{0,\infty}^{(1)} e_{\vec{m}} = O(\log \lambda_{\vec{m}}) \cdot e_{\vec{m}}, $$
	where $\lambda_{\vec{m}}$ is the eigenvalue of $e_{\vec{m}}$ w.r.t. $\Casimir_{\gp{K},\infty}$. Writing $f_{\infty} = \sideset{}{_{\vec{m}}} \sum a_{\vec{m}} e_{\vec{m}}$ with $a_{\vec{m}} \in \ag{C}$, we get
\begin{align*}
	\extNorm{ \IntwR_{0,\fin}^{(1)} f_{\infty} }^2 &= \extNorm{ \sum a_{\vec{m}} \IntwR_{0,\fin}^{(1)} e_{\vec{m}} }^2 = \sum \norm[a_{\vec{m}}]^2 \Norm[\IntwR_{0,\fin}^{(1)} e_{\vec{m}}]^2 \\
	&\ll_{\epsilon} \sum \norm[a_{\vec{m}}]^2 \lambda_{\vec{m}}^{2\epsilon} \leq \left( \sum \norm[a_{\vec{m}}]^2 \lambda_{\vec{m}}^2 \right)^{\epsilon} \left( \sum \norm[a_{\vec{m}}]^2 \right)^{1-\epsilon},
\end{align*}
	proving the first inequality in (1). The second inequality is elementary.
\end{proof}

\begin{lemma}
	In (\ref{ReguTerm}), we have
	$$ \extnorm{\int_{[\PGL_2]}^{\reg} a(\mathfrak{t}).\eis_3^{\sharp} \cdot \overline{\eis_3^{\sharp}} \cdot \Reis(a(\mathfrak{t}).\eis_2^* \overline{\eis_2^*})} \ll_{\F, \epsilon} (\Cond(\pi) K)^{\epsilon} K^{-\Norm[\mathfrak{t}]}. $$
\label{ReguTermBd2}
\end{lemma}
\begin{proof}
	The proof is similar to that of Lemma \ref{ReguTermBd1}. By $\Intw_s f_2 = \lambda_{\F}(s-1/2)$, we easily obtain
\begin{align*}
	\Reis(a(\mathfrak{t}).\eis_2^* \overline{\eis_2^*}) &= \norm[\Lambda_{\F}^*]^2 \cdot \left\{ \eis^{\reg,(2)}(\frac{1}{2}, a(\mathfrak{t}).f_2) + \frac{1}{2} \overline{\lambda_{\F}^{(1)}(-\frac{1}{2})} \eis^{\reg,(1)}(\frac{1}{2}, a(\mathfrak{t}).f_2) \right. \\
	&\left. + \frac{1}{2} \eis^{\reg,(1)}(\frac{1}{2}, \Intw_0^{(1)} a(\mathfrak{t}).f_2) + \frac{1}{4} \overline{\lambda_{\F}^{(1)}(-\frac{1}{2})} \eis^{\reg}(\frac{1}{2}, \Intw_0^{(1)} a(\mathfrak{t}).f_2) \right\}.
\end{align*}
	We then apply Proposition \ref{TPFReguTerm2} to treat each term of
	$$ \int_{[\PGL_2]}^{\reg} a(\mathfrak{t}).\eis_3^{\sharp} \cdot \overline{\eis_3^{\sharp}} \cdot \eis^{\reg,(n)}(\frac{1}{2}, a(\mathfrak{t}).f_2), n=1,2; \quad \int_{[\PGL_2]}^{\reg} a(\mathfrak{t}).\eis_3^{\sharp} \cdot \overline{\eis_3^{\sharp}} \cdot \eis^{\reg,(n)}(\frac{1}{2}, \Intw_0^{(1)} a(\mathfrak{t}).f_2), n=0,1. $$
	Combining (\ref{Regu3PerD}), Lemma \ref{ReguNAMainBd} (3), Lemma \ref{ReguAmpBd} (3) and Lemma \ref{ReguAMainBd} (3) we get
	$$ \extnorm{ \frac{\partial^n R}{\partial s^n}^{\hol}(\frac{1}{2}; a(\mathfrak{t}).\eis_3^{\sharp} \cdot \overline{\eis_3^{\sharp}}; a(\mathfrak{t}).f_2) }, \text{ resp. } \extnorm{ \frac{\partial^n R}{\partial s^n}^{\hol}(\frac{1}{2}; a(\mathfrak{t}).\eis_3^{\sharp} \cdot \overline{\eis_3^{\sharp}}; \Intw_0^{(1)} a(\mathfrak{t}).f_2) } \ll_{\F,\epsilon} (\Cond(\pi)K)^{\epsilon} K^{-\Norm[\mathfrak{t}]}. $$
	Most of the remaining terms have already been treated in the proof of Lemma \ref{ReguTermBd1}, except
	$$ \extnorm{ \Proj_{\gp{K}}(\Intw_0^{(1)} a(\mathfrak{t}).f_3 \overline{\Intw_0 f_3}) } \ll_{\F,\epsilon} (\Cond(\pi)K)^{\epsilon} K^{-\Norm[\mathfrak{t}]/2}, $$
which follows from an analogue of Lemma \ref{IntwDerEst} for $\pi(1,\omega^{-1})$ together with the argument given in Lemma \ref{ReguTermBd1} (2). One may find such technical analysis of $\Intw_0^{(1)}$ (in fact explicit formula for $\Intw_s$) in \cite{Wu5}.
\end{proof}

	\subsection{Degenerate Term}
	
\begin{lemma}
	The contribution of (\ref{DgnTerm}) is
	$$ \extnorm{ \int_{[\PGL_2]}^{\reg} \Reis(a(\mathfrak{t}).\eis_2^* \overline{\eis_2^*}) \cdot \Reis(a(\mathfrak{t}).\eis_3^{\sharp} \overline{\eis_3^{\sharp}}) } \ll_{\epsilon} K^{-\Norm[\mathfrak{t}]+\epsilon}. $$
\label{DgnTermBd}
\end{lemma}
\begin{proof}
	By \cite[Theorem 2.4]{Wu2}, the desired bound follows from (\ref{DT}), which is already proved.
\end{proof}

\section{Complements}

	\subsection{Bases for Generalized New Vectors}
	\label{BaseGNV}

	We restrict to a local $\vp$-adic field $\F$ in this subsection. We assume the cardinality of the residue field is $q$, and fix a uniformizer $\varpi$. Recall that the subspace of ``generalized new vectors'' in a (unitary) admissible irreducible representation $\pi$ of $\GL_2(\F)$ consists of the vectors invariant by $\gp{B}_1(\vo)$; the level $n$ subspace of generalized new vectors consists of the vectors invariant by $\gp{K}_1[\vp^n]$, where
	$$ \gp{B}_1(\vo) = \begin{pmatrix} \vo^{\times} & \vo \\ 0 & 1 \end{pmatrix}, \quad \gp{K}_1[\vp^n] = \begin{pmatrix} \vo^{\times} & \vo \\ \vp^n & 1+\vp^n \end{pmatrix}. $$
	Three bases of the subspace of generalized new vectors arise naturally. Their multual relations are our concern in this subsection.

\noindent \textbf{Basis 1:} Let $e_0$ be a unitary new vector of $\pi$. $\{ e_0, a(\varpi^{-1})e_0, \cdots, a(\varpi^{-k})e_0 \}$ is a (normal) basis of the level $\cond(\pi)+k$ subspace for $k \in \ag{N}$.

\noindent \textbf{Basis 2:} Applying Gramm-Schmidt to Basis 1, we get an ortho-normal basis of the level $\cond(\pi)+k$ subspace, denoted by $\{ e_0, e_1, \cdots, e_k \}$.

\noindent \textbf{Basis 3:} In either of the two case
\begin{itemize}
	\item $\pi=\pi(1,\omega)$ with $\omega$ unitary (hence $\cond(\pi)=\cond(\omega)$),
	\item $\pi$ is principal spherical and realized in the induced model,
\end{itemize}
	we denote by $D_k$ the function, whose restriction to $\gp{K}$ is supported in $\gp{K}_0[\vp^{\cond(\pi)+k}]$, is invariant by $\gp{K}_1[\vp^{\cond(\pi)+k}]$, and takes value $\Vol(\gp{K}_0[\vp^{\cond(\pi)+k}])^{-1/2}$ at $1$. The set $\{ D_0, D_1, \cdots, D_k \}$ is also a (normal) basis of the level $\cond(\pi)+k$ subspace for any $k \in \ag{N}$.

\begin{proposition}
(1) If $\pi$ is such that $L(s,\pi)=1$, then basis 1 and basis 2 coincide with each other.

\noindent (2) If $\pi \simeq \mathrm{St}_{\chi}$ is Steinberg with unramified twist, then basis 1 and basis 2 are related by
	$$ e_n = (1-q^{-2})^{-1/2} \cdot \{ a(\varpi^{-n}).e_0 - \chi(\varpi)q^{-1} a(\varpi^{-(n-1)}).e_0 \}. $$

\noindent (3) If $\cond(\omega) > 0, \pi=\pi(1,\omega)$, then basis 1 and basis 3 coincide with each other. Their relation to basis 2 is given by
	$$ e_0=D_0, \quad e_n = (1-q^{-1})^{-1/2} (D_n - q^{-1/2}D_{n-1}), \forall n \geq 1; $$
	$$ D_n = (1-q^{-1})^{1/2} \sum_{k=0}^n q^{-k/2} e_{n-k}, \forall n \geq 1. $$
	Moreover, the dimension $d_n$ of the $\gp{K}$-representation generated by $e_n$ (which is irreducible) is
	$$ d_0 = q^c(1+q^{-1}), \quad d_n = q^{n+c}(1-q^{-2}), n \geq 1 \quad \text{where } c = \cond(\omega). $$

\noindent (4-1) If $\pi$ is spherical with Satake parameters $\alpha_1, \alpha_2$, then basis 1 and basis 2 are related by
	$$ e_1 = c_1^{-1/2} \cdot \left\{ a(\varpi^{-1}).e_0 - \frac{q^{-1/2}}{1+q^{-1}}(\bar{\alpha}_1 + \bar{\alpha}_2) e_0 \right\}, $$
	$$ e_n = c^{-1/2} \left\{ a(\varpi^{-n}).e_0 - q^{-1/2}(\bar{\alpha}_1 + \bar{\alpha}_2) a(\varpi^{-(n-1)}).e_0 + q^{-1} \bar{\alpha}_1 \bar{\alpha}_2 a(\varpi^{-(n-2)}).e_0 \right\}, \forall n \geq 2, $$
	$$ \text{with }c_1=1-\frac{q^{-1}\norm[\alpha_1+\alpha_2]^2}{(1+q^{-1})^2} \asymp_{\theta} 1, \quad c = 1-q^{-2} - \frac{q^{-1}-q^{-2}-q^{-3}}{(1+q^{-1})^2}\norm[\alpha_1+\alpha_2]^2 \asymp_{\theta} 1; $$
	$$ a(\varpi^{-1}).e_0 = c_1^{1/2}e_1 + \frac{q^{-1/2}}{1+q^{-1}}(\bar{\alpha}_1 + \bar{\alpha}_2) e_0, $$
\begin{align*}
	a(\varpi^{-n}).e_0 &= \sum_{k=0}^{n-2} q^{-\frac{k}{2}} \frac{\bar{\alpha}_1^{k+1}-\bar{\alpha}_2^{k+1}}{\bar{\alpha}_1 - \bar{\alpha}_2} c^{\frac{1}{2}}e_{n-k} + q^{-\frac{n-1}{2}} \frac{\bar{\alpha}_1^n-\bar{\alpha}_2^n}{\bar{\alpha}_1 - \bar{\alpha}_2} c_1^{\frac{1}{2}}e_1 \\
	&+ q^{-\frac{n}{2}} \left\{ \frac{\bar{\alpha}_1^{n+1}- \bar{\alpha}_2^{n+1}}{\bar{\alpha}_1 - \bar{\alpha}_2} - \frac{\bar{\alpha}_1 + \bar{\alpha}_2}{q+1} \frac{\bar{\alpha}_1^n- \bar{\alpha}_2^n}{\bar{\alpha}_1 - \bar{\alpha}_2} \right\} e_0.
\end{align*}

\noindent (4-2) If $\pi$ is moreover principal, then their relations to basis 3 are given by
\begin{itemize}
	\item Basis 1 $\Leftrightarrow$ Basis 3:
	$$ e_0 = D_0, \quad a(\varpi^{-n}).e_0 = \alpha_2^{-n}q^{-n/2} D_0 + \frac{1-\alpha_1\alpha_2^{-1}q^{-1}}{(1+q^{-1})^{1/2}} \sum_{k=1}^n \alpha_1^{-k}\alpha_2^{k-n}q^{-(n-k)/2} D_k, \forall n \geq 1; $$
	$$ D_n = \frac{\alpha_1^n(1+q^{-1})^{1/2}}{1-\alpha_1\alpha_2^{-1}q^{-1}} \cdot \{ a(\varpi^{-n}).e_0 - \alpha_2^{-1}q^{-1/2} a(\varpi^{-(n-1)}).e_0 \}, \forall n \geq 1. $$
	\item Basis 2 $\Leftrightarrow$ Basis 3:
	$$ e_1 = (1+q^{-1})^{1/2} \cdot \{ D_1 - (q+1)^{-1/2} D_0 \}, \quad e_n = (1-q^{-1})^{-1/2} \cdot \{ D_n - q^{-1/2}D_{n-1} \}, \forall n \geq 2; $$
	$$ D_n = (1-q^{-1})^{1/2} \sum_{k=0}^{n-2} q^{-k/2} e_{n-k} + q^{-(n-1)/2} (1+q^{-1})^{-1/2}e_1 + q^{-(n-1)/2} (q+1)^{-1/2} e_0, \forall n \geq 1. $$
\end{itemize}
	Moreover, the dimension $d_n$ of the $\gp{K}$-representation generated by $e_n$ is
	$$ d_0 = 1, \quad d_1=q, \quad d_n = q^n(1-q^{-2}), n \geq 2. $$
\label{LocBaseRel}
\end{proposition}
\begin{proof}
	The proof is very computational. We only give hints for the fastest way we have found.

\noindent (1) \& (2) Use the description of the Kirillov new vector given in \cite[Table 1]{FMP17}.

\noindent (3) The first assertion follows by direct computation. The second uses again \cite[Table 1]{FMP17} or direct computation of $\Pairing{D_n}{D_m}$ below in (4-2). The dimension formula follows by noticing and evaluating at $1$, up to a complex number with norm $1$, that
	$$ e_n(\kappa) = d_n^{1/2} \Pairing{\kappa.e_n}{e_n}, \forall \kappa \in \gp{K}. $$

\noindent (4-1) We first play with MacDonald's formula \cite[Proposition 4.6.6]{Bu98}, from which we easily deduce that
	$$ e_n' = a(\varpi^{-n}).e_0 - q^{-1/2}(\bar{\alpha}_1 + \bar{\alpha}_2)a(\varpi^{-{n-1}}).e_0 + q^{-1} \bar{\alpha}_1 \bar{\alpha}_2 a(\varpi^{-(n-2)}).e_0, n \geq 2 $$
is orthogonal to $a(\varpi^{-k}).e_0$ for $0 \leq k \leq n-2$, since $\Pairing{a(\varpi^{-n}).e_0}{e_0}$ is of the form $C_1\bar{\alpha}_1^k q^{-k/2} + C_2\bar{\alpha}_2^k q^{-k/2}$ with $C_1,C_2$ constants. The verification that it is also orthogonal to $a(\varpi^{-{n-1}}).e_0$ uses the fact that $\pi$ is unitary, i.e. either $\norm[\alpha_1]=\norm[\alpha_2]=1$ or $\alpha_1 \bar{\alpha}_2=1$. Hence $e_n'$ is proportional to $e_n$ and the formula for $c$ follows easily from $\Norm[e_n']^2 = \Pairing{e_n'}{a(\varpi^{-n}).e_0}$. In order to invert the relations, we write $f_n = a(\varpi^{-n}).e_0, \sigma_1 = q^{-1/2}(\bar{\alpha}_1 + \bar{\alpha}_2), \sigma_2 = q^{-1}\bar{\alpha}_1 \bar{\alpha}_2$ and introduce the formal series
\begin{align*}
	\sum_{n=0}^{\infty} e_n' X^n &= f_0 + (f_1 - \frac{\sigma_1}{1+q^{-1}} f_0)X + \sum_{n=2}^{\infty} (f_n-\sigma_1 f_{n-1} + \sigma_2 f_{n-2})X^n \\
	&= \left( \sum_{n=0}^{\infty} f_n X^n \right) \cdot \left( 1- \sigma_1 X + \sigma_2 X^2 \right) + \frac{\sigma_1}{q+1} f_0 X,
\end{align*}
from which we get and conclude by
	$$ \sum_{n=0}^{\infty} f_n X^n = \left( \sum_{n=0}^{\infty} (\bar{\alpha}_1 q^{-1/2} X)^n \right) \left( \sum_{n=0}^{\infty} (\bar{\alpha}_2 q^{-1/2} X)^n \right) \left( \sum_{n=0}^{\infty} e_n' X^n - \frac{\sigma_1}{q+1} e_0 X \right). $$

\noindent (4-2) For the first relation, we evaluate $a(\varpi^{-n}).e_0$ at $n_-(\varpi^k), k=0,\dots, n$ and use
	$$ \begin{pmatrix} 1 & \\ \varpi^k & 1 \end{pmatrix} \begin{pmatrix} \varpi^{-n} & \\ & 1 \end{pmatrix} = \begin{pmatrix} \varpi^{-k} & \varpi^{-n} \\ & \varpi^{k-n} \end{pmatrix} \begin{pmatrix} 0 & -1 \\ 1 & \varpi^{n-k} \end{pmatrix}. $$
	For the second, we apply Gramm-Schmidt to $D_0,D_1,\dots$ using
	$$ \Pairing{D_m}{D_n} = \left( \Vol(\gp{K}_0[\vp^{\cond(\pi)+\max(m,n)}]) / \Vol(\gp{K}_0[\vp^{\cond(\pi)+\min(m,n)}]) \right)^{1/2}. $$
	The dimension formula follows the same way as in the proof of (3).
\end{proof}
\begin{corollary}
	Let $\pi$ be unitary spherical with Satake paramter $\alpha_1,\alpha_2$. Let $\Proj_n$ denote the orthogonal projection onto the $\gp{K}_0[\vp^n]$-invariant subspace of $\pi$. Then we have
	$$ \Proj_{n-k} (a(\varpi^{-n}).e_0) = q^{-\frac{k}{2}} \frac{\bar{\alpha}_1^{k+1} - \bar{\alpha}_2^{k+1}}{\bar{\alpha}_1 - \bar{\alpha}_2} a(\varpi^{-(n-k)}).e_0 - q^{-\frac{k+1}{2}} \bar{\alpha}_1\bar{\alpha}_2 \frac{\bar{\alpha}_1^k - \bar{\alpha}_2^k}{\bar{\alpha}_1 - \bar{\alpha}_2} a(\varpi^{-(n-k-1)}).e_0,  0 \leq k \leq n-1; $$
	$$ \Proj_0 (a(\varpi^{-n}).e_0) = q^{-\frac{n}{2}} \left\{ \frac{\bar{\alpha}_1^{n+1}- \bar{\alpha}_2^{n+1}}{\bar{\alpha}_1 - \bar{\alpha}_2} - \frac{\bar{\alpha}_1 + \bar{\alpha}_2}{q+1} \frac{\bar{\alpha}_1^n- \bar{\alpha}_2^n}{\bar{\alpha}_1 - \bar{\alpha}_2} \right\} e_0. $$
\label{ProjCal}
\end{corollary}
\begin{proof}
	Proposition \ref{LocBaseRel} (4-1) gives
	$$ \Proj_{n-1} a(\varpi^{-n}).e_0 = q^{-1/2}(\bar{\alpha}_1 + \bar{\alpha}_2) a(\varpi^{-(n-1)}).e_0 - q^{-1} \bar{\alpha}_1 \bar{\alpha}_2 a(\varpi^{-(n-2)}).e_0, n \geq 2; $$
	$$ \Proj_1 a(\varpi^{-1}).e_0 = \frac{q^{-1/2}}{1+q^{-1}}(\bar{\alpha}_1 + \bar{\alpha}_2) e_0. $$
	Suppose we have
	$$ \Proj_{n-k} a(\varpi^{-n}).e_0 = q^{-k/2} A_k(\bar{\alpha}_1,\bar{\alpha}_2) a(\varpi^{-(n-k)}).e_0 - q^{-(k+1)/2} B_k(\bar{\alpha}_1,\bar{\alpha}_2) a(\varpi^{n-k-1}).e_0, 0 \leq k \leq n-1. $$
	Since $\Proj_{n-k-1} \circ \Proj_{n-k} = \Proj_{n-k-1}$. we get
	$$ \begin{pmatrix} A_{k+1} \\ B_{k+1} \end{pmatrix} = \begin{pmatrix} \bar{\alpha}_1 + \bar{\alpha}_2 & -1 \\ \bar{\alpha}_1 \bar{\alpha}_2 & 0 \end{pmatrix} \begin{pmatrix} A_k \\ B_k \end{pmatrix}. $$
	Diagonalizing the matrix, we easily get the desired formula.
\end{proof}

	\subsection{Transposition Formula for Local Rankin-Selberg}

	Consider a local field $\F$, a generic representation $\pi$ of $\gp{G}=\GL_2(\F)$ with central character $\omega$, two induced representations $\pi_j=\pi(\chi_j,\chi_j')$ with $\omega \chi_1 \chi_1' \chi_2 \chi_2' = 1$. There are two ways of realizing the $\GL_2(\F)$-invariant trilinear form on $\pi^{\infty} \times \pi_1^{\infty} \times \pi_2^{\infty}$. Namely, we have
	$$ \ell_1: \pi^{\infty} \times \pi_1^{\infty} \times \pi_2^{\infty} \to \ag{C}, (e,f_1,f_2) \mapsto \int_{\gp{Z}\gp{N} \backslash \gp{G}} W_e(g) W_1(g) f_2(g) dg; $$
	$$ \ell_2: \pi^{\infty} \times \pi_1^{\infty} \times \pi_2^{\infty} \to \ag{C}, (e,f_1,f_2) \mapsto \int_{\gp{Z}\gp{N} \backslash \gp{G}} W_e(g) W_2(g) f_1(g) dg; $$
where $W_e$ resp. $W_j$ is the Whittaker function with respect to $\overline{\psi}$ resp. $\psi$ of $e$ resp. $f_j$ for $j=1,2$.
\begin{proposition}
	The two trilinear forms are related by
	$$ \ell_1 = \chi_1\chi_2'(-1)\gamma(\frac{1}{2}, \pi, \chi_1'\chi_2; \overline{\psi})^{-1} \ell_2, $$
where the gamma factor is the one appearing in the theory of $\GL_2 \times \GL_1$ (Hecke-Jacquet-Langlands).
\label{TransFLocRS}
\end{proposition}
\begin{proof}
	Write $W$ for $W_e$. Taking $\Phi_j \in \Sch(\F^2)$ such that
	$$ f_j(g) = f_{\Phi_j}(g) = \chi_j(\det g) \norm[\det g]^{\frac{1}{2}} \int_{\F^{\times}} \Phi_j((0,t)g) \chi_j (\chi_j')^{-1}(t) \norm[t] d^{\times}t, $$
we can proceed as \cite[\S 8.2]{J09}
\begin{align*}
	\ell_1 &= \int_{\gp{N} \backslash \gp{G}} W(g) W_1(g) \Phi_2((0,1)g) \chi_2(\det g) \norm[\det g]^{\frac{1}{2}} dg \\
	&= \int_{\gp{G}} W(g) f_{\Phi_1}(wg) \Phi_2((0,1)g) \chi_2(\det g) \norm[\det g]^{\frac{1}{2}} dg \\
	&= \int_{\F^{\times}} \int_{\gp{G}} \Phi_1((1,0)g) W(a(t^{-1})g) \Phi_2((0,1)g) \chi_1\chi_2(\det g) \norm[\det g] (\chi_1'\chi_2)^{-1}(t) dg d^{\times}t \\
	&= \int_{\gp{G}} \left( \int_{\F^{\times}} W(a(t)g) \chi_1'\chi_2(t) d^{\times}t \right) \Phi_1((1,0)g) \Phi_2((0,1)g) \chi_1\chi_2(\det g) \norm[\det g] dg.
\end{align*}
	The expression of $\ell_2$ is similar. Applying the local functional equation to the inner integral and making the change of variables $g \mapsto a(-1)w^{-1}g$, we conclude.
\end{proof}

	\subsection{Some Regularized Triple Product Formulas}
	\label{FERTPF}
	
	All our regularized triple products are in the singular case, so that neither \cite{Za82} nor \cite[\S 4.4]{MV10} (especially \cite[\S 4.4.3]{MV10}) apply. \cite[Theorem 2.7]{Wu2} has set an example of such analysis at the singular points. We need two more variants of it. We only give the proof of the first proposition as a recall on the techniques of \cite{Wu9, Wu2} and omit the other one.
\begin{proposition}
	Let $1 \neq \omega$ be a non-trivial Hecke character. Let $f_1,f_2 \in \pi(1,1)$ and $f_3 \in \pi(\omega, \omega^{-1})$. For any $n \in \ag{N}$ and $\omega^2 =1$ resp. $\omega^2 \neq 1$,
	$$ \int_{[\PGL_2]}^{\reg} \eis^*(0,f_1) \cdot \eis^*(0,f_2) \cdot \eis^{\reg,(n)}(\frac{1}{2},f_3) \quad \text{resp.} \quad \int_{[\PGL_2]}^{\reg} \eis^*(0,f_1) \cdot \eis^*(0,f_2) \cdot \eis^{(n)}(\frac{1}{2},f_3) $$
is equal to the generalized Rankin-Selberg value
	$$ \left( \frac{\partial^n R}{\partial s^n} \right)^{\hol}(\frac{1}{2}, \eis^*(0,f_1) \cdot \eis^*(0,f_2); f_3). $$
\label{TPFReguTerm}
\end{proposition}
\begin{proof}
	Let $\Reis(f_1,f_2)$ be the $\intL^2$-residue of $\eis^*(0,f_1) \cdot \eis^*(0,f_2)$ and $\varphi := \eis^*(0,f_1) \cdot \eis^*(0,f_2) - \Reis(f_1,f_2)$. In the case $\omega^2=1$, we are reduced to computing
	$$ \int_{[\PGL_2]}^{\reg} \varphi \cdot \eis^{\reg,(n)}(\frac{1}{2},f_3) + \int_{[\PGL_2]}^{\reg} \Reis(f_1,f_2) \cdot \eis^{\reg,(n)}(\frac{1}{2},f_3). $$
	By \cite[Proposition 2.6 (2)]{Wu2}, the first term is equal to the generalized Rankin-Selberg value plus
	$$ \lambda_{\F}^{(n)}(0) \cdot \Proj_{\gp{K}}(f_3 \otimes \omega^{-1}) \cdot \int_{[\PGL_2]} \varphi \otimes \omega = \lambda_{\F}^{(n)}(0) \cdot \Proj_{\gp{K}}(f_3 \otimes \omega^{-1}) \cdot \int_{[\PGL_2]}^{\reg} \eis^*(0,f_1) \cdot \eis^*(0,f_2 \otimes \omega), $$
which is vanishing by \cite[Lemma 3.1]{Wu9}, while other terms are $0$. The second term is also vanishing by \cite[Theorem 2.4 (1)]{Wu2}. In the case $\omega^2 \neq 1$, we proceed similarly using \cite[Proposition 2.6 (1)]{Wu2}.
\end{proof}
\begin{proposition}
	Let $1 \neq \omega$ be a non-trivial Hecke character. Let $f_1 \in \pi(1,\omega^{-1}), f_2 \in \pi(1,\omega)$ and $f_3 \in \pi(1, 1)$. For any $n \in \ag{N}$
	$$ \int_{[\PGL_2]}^{\reg} \eis(0,f_1) \cdot \eis(0,f_2) \cdot \eis^{\reg,(n)}(\frac{1}{2},f_3) $$
is equal to the sum of the generalized Rankin-Selberg value
	$$ \left( \frac{\partial^n R}{\partial s^n} \right)^{\hol}(\frac{1}{2}, \eis(0,f_1) \cdot \eis(0,f_2); f_3) $$
and a weighted sum of the following terms with weights depending only on $\F$ ($\lambda_{\F}(s)$)
\begin{itemize}
	\item $\Proj_{\gp{K}}(f_1 f_2) \Proj_{\gp{K}}(f_3)$, $\Proj_{\gp{K}}(\Intw_0 f_1 \Intw_0 f_2) \Proj_{\gp{K}}(f_3)$, $\Proj_{\gp{K}}(\Intw_0^{(1)} f_1 \cdot \Intw_0 f_2) \Proj_{\gp{K}}(f_3)$;
	\item $\Proj_{\gp{K}}(f_1 f_2 \cdot \widetilde{\Intw}_{1/2}^{(l)} f_3)$, $\Proj_{\gp{K}}(\Intw_0 f_1 \Intw_0 f_2 \cdot \widetilde{\Intw}_{1/2}^{(l)} f_3)$ for $0 \leq l \leq n+1$.
\end{itemize}
\label{TPFReguTerm2}
\end{proposition}

\section{Appendix: A Fourth Moment Bound for $\GL_2$}

	\subsection{A Gap in Our Previous Work}
	
	The proof of the bound of certain fourth moment $L$-functions for $\GL_2$, namely \cite[Theorem 6.6]{Wu14}, was sketched without giving details of the relevant test functions. We explained that \cite[Theorem 6.6]{Wu14} follows from the main theorem of \cite{MV10} without amplification. In fact, what follows directly is an equivalent truncated version of \cite[Theorem 6.6]{Wu14}, which we are going to formalize.
		
\begin{definition}
	For $C \gg 1$ and any integral ideal $\idlN$ of $\F$, we define $\mathcal{A}_0(C, \idlN; \omega)$ to be the set of all cuspidal representations $\pi$ in $\intL_0^2(\GL_2, \omega)$ such that 
\begin{itemize}
	\item for every $v \mid \infty$, the analytic conductor $\Cond(\pi_v) \leq C$;
	\item for every finite place $\vp < \infty$, the logarithmic conductor $\cond(\pi_{\vp}) \leq \mathrm{ord}_{\vp}(\idlN)$.
\end{itemize}
	We define the limit
	$$ \mathcal{A}_0(\infty,\idlN; \omega) := \sideset{}{_{C > 0}} \bigcup \mathcal{A}_0(C,\idlN; \omega). $$
	If $\omega = \mathbbm{1}$ is trivial, we write $\mathcal{A}_0(C, \idlN) = \mathcal{A}_0(C, \idlN; \mathbbm{1})$ and $\mathcal{A}_0(\infty, \idlN) = \mathcal{A}_0(\infty, \idlN; \mathbbm{1})$.
\end{definition}
\begin{remark}
	Each $\mathcal{A}_0(C, \idlN; \omega)$ is a finite family. In the rest of this section, we consider the family $\mathcal{A}_0(C, \idlN) \otimes \norm_{\A}^{i\tau}$ for $\tau \in \R, \norm[\tau] \gg 1$, whose members are of central character $\omega = \norm_{\A}^{2i\tau}$. We will always assume $C \gg (1+\norm[\tau])^2$.
\end{remark}

\begin{theorem}
	
	For any $C \gg (1+\norm[\tau])^2$ and any $\epsilon > 0$, we have 
	$$ \sideset{}{_{\pi \in \mathcal{A}_0(C,\idlN)}} \sum \extnorm{L(1/2+i\tau, \pi)}^4 \ll_{\epsilon} (C^{r_1+r_2} \Nr(\idlN))^{1+\epsilon}. $$
\label{4MMBd}
\end{theorem}
\begin{remark}
	It is clear that Theorem \ref{4MMBd} implies \cite[Theorem 6.6]{Wu14}. In the application to \cite[Corollary 6.7]{Wu14}, only the exponent of $\Nr(\idlN)$ is important.
\end{remark}

	We have chosen $f_2 \in \pi(1,1)$ to be the spherical vector taking value $1$ on $\gp{K}$ and taken 
	$$ \eis^*(0,f_2) = \left. \Lambda_{\F}(1+2s) \eis(s,f_2) \right|_{s=0}. $$ 
	We will modify our choice of $f_3 \in \pi(1, \omega^{-1}) = \pi(1, \norm_{\A}^{-2i\tau})$ according to $C, \idlN$ and take
	$$ \eis^{\sharp}(0,f_3) = \left. L(1+2s, \norm_{\A}^{2i\tau}) \eis(s,f_3) \right|_{s=0} = \zeta_{\F}(1+2i\tau) \eis(0,f_3). $$
	For every $\pi \in \mathcal{A}_0(C,\idlN)$, we will modify our choice of $\varphi \in \pi \otimes \norm_{\A}^{i\tau}$ via modified choices of $W_{\varphi,v}$. Recall our period decomposition (\ref{RSbeforeCS}) in the modified setting
	$$ \int_{[\PGL_2]} \varphi \cdot \eis^*(0,f_2) \cdot \eis^{\sharp}(0,f_3) = L(\frac{1}{2}+i\tau,\pi)^2 \cdot \prod_v \ell_v(W_{\varphi,v}, f_{2,v}, f_{3,v}), $$
	where the local terms are modified as
\begin{align*}
	&\ell_v(\cdots) = \int_{\F_v^{\times} \times \gp{K}_v} W_{\varphi,v}(a(y)\kappa) W_{f_2,v}^*(a(-y)\kappa) f_{3,v}(\kappa) \norm[y]_v^{-\frac{1}{2}} d^{\times}y d\kappa, \quad v \mid \infty; \\
	&\ell_{\vp}(\cdots) = \frac{\zeta_{\vp}(1+2i\tau)}{L(1/2+i\tau,\pi_{\vp})^2} \int_{\F_{\vp}^{\times} \times \gp{K}_{\vp}} W_{\varphi,\vp}(a(y)\kappa) W_{f_2,\vp}^*(a(-y)\kappa) f_{3,\vp}(\kappa) \norm[y]_{\vp}^{-\frac{1}{2}} d^{\times}y d\kappa, \quad \vp < \infty;
\end{align*}
so that $\ell_{\vp}=1$ for all but finitely many $\vp$.

	\subsection{Modified Local Choices \& Bound}
	
		\subsubsection{Archimedean Places}
		
	We modify our choices in \S \ref{ChoicesA} as follows. Given a bump function $\phi$ with supported contained in a small compact neighborhood of $0$, say in $\{ x \in \F_v: \norm[x] \leq \delta_0 \}$, we construct
\begin{equation} 
	f_0(\kappa) = \norm[d]_v^{-2i\tau} \int_{\F_v^{\times}} \phi(ct)\phi(dt-1) \norm[t]_v d^{\times}t, \quad \kappa = \begin{pmatrix} a & b \\ c & d \end{pmatrix} \in \gp{K}_v. 
\label{4MLocChoiceInvArch}
\end{equation}
	We adjust $\phi$ so that $f_0$ is unitary. It is clear that $f_0 \in \pi(1,\norm_v^{-2i\tau})$ and
	$$ f_0(\kappa) \neq 0 \text{ for some } \kappa = \begin{pmatrix} a & b \\ c & d \end{pmatrix} \in \gp{K}_v \quad \Rightarrow \quad \norm[c] \leq \delta_0/(1-\delta_0), $$
	hence the support of $f_0$ is a small neighborhood $U$ of $1 \in \gp{K}_v$. We then choose $t_v \in \F_v^{\times}$ such that $\norm[t_v] = C^{1+\epsilon}$, and define
\begin{equation} 
	f_{3,v} = a(t_v).f_0. 
\label{4MLocChoiceArch}
\end{equation}
	Specify $W_{\varphi,v}$ by taking $W_{\varphi,v}(a(y))$ to be a fixed smooth function $\delta_v(y)$ with support in a compact neighborhood of $1$ in $\F_v^{\times}$, invariant by $\ag{C}^{1}$ if $\F_v=\ag{C}$, such that
	$$ \int_{\F_v^{\times}} \delta_v(y) W_2^*(a(-y)) \norm[y]_v^{-\frac{1}{2}} d^{\times}y \gg 1, \quad \int_{\F_v^{\times}} \norm[\delta_v(y)]^2 d^{\times}y = 1. $$
	The proof of Lemma \ref{LocLowerBdA} works through and gives
\begin{equation}
	\extnorm{ \ell_v(W_{\varphi,v}, f_{2,v}, f_{3,v}) } \gg \norm[t_v]^{-1/2} = C^{-(1+\epsilon)/2}.
\label{4MLocLowerBdA}
\end{equation}

		\subsubsection{Finite Places}
		
	We modify our choices in \S \ref{ChoicesNA} as follows. Write $n = \mathrm{ord}_{\vp}(\idlN)$. If $n > 0$, then take $f_{3,\vp} \in \pi(1, \norm_{\vp}^{-2i\tau})$ whose restriction to $\gp{K}_{\vp}$ is
	$$ \begin{pmatrix} a & b \\ c & d \end{pmatrix} \mapsto \Vol(\gp{K}_0[\vp^n])^{-1/2} \mathbbm{1}_{\gp{K}_0[\vp^{n}]}; $$
	otherwise, take $f_{3,\vp}$ to be the spherical function taking value $1$ on $\gp{K}_{\vp}$. Specify $W_{\varphi,\vp}$ to be a new vector of $\pi_{\vp} \otimes \norm_{\vp}^{i\tau}$ in the Whittaker model. Note that $\cond(\pi_{\vp} \otimes \norm_{\vp}^{i\tau}) = \cond(\pi_{\vp}) \leq n$. The proof of Lemma \ref{LocLowerBdNA} works through and gives if $n > 0$
\begin{equation}
	\extnorm{ \ell_{\vp}(W_{\varphi,\vp}, f_{2,\vp}, f_{3,\vp}) } \gg \Nr(\vp)^{-n/2} \cdot \frac{\Norm[W_{\varphi,\vp}]}{\sqrt{L(1, \pi_{\vp} \times \bar{\pi}_{\vp})}},
\label{4MLocLowerBdNA}
\end{equation}
	while if $n=0$, $\ell_{\vp}(W_{\varphi,\vp}, f_{2,\vp}, f_{3,\vp}) = 1$.
	

	\subsection{Modified Global Bounds}
	
	On the global side, we use the version without amplification. We begin by recalling the construction of the measure/operator of regularization. We choose a finite place $\vp_0 \nmid \idlN$ at which $\psi$ is unramified. In particular, $\pi_{\vp_0}$ is unramified for each $\pi \in \mathcal{A}_0(C,\idlN)$. Write and assume the smallness of $\vp_0$ (see (\ref{Smallp0}))
	$$ \varpi_0 := \varpi_{\vp_0}, \quad q_0 := q_{\vp_0} \ll (\log \Nr(\idlN))^2, \quad \gp{K}_0 := \gp{K}_{\vp_0}, \quad \omega_0 := \omega_{\vp_0}. $$ 
	Adapt Definition \ref{NormHeckeOp} as follows. Define the Hecke operators for $n \in \ag{N}$
	$$ T_0(n) := \int_{\gp{K}_0^2} \rpR_{\norm_{\A}^{-2i\tau}}^{\mathcal{A}}(\kappa_1 a(\varpi_0^n) \kappa_2) d\kappa_1 d\kappa_2; $$
	the eigenvalues of $T_0(1)$ on spherical vectors in $\pi(\norm_{\vp_0}^{1/2},\norm_{\vp_0}^{-1/2-2i\tau})$ resp. $\pi(\norm_{\vp_0}^{1/2-2i\tau},\norm_{\vp_0}^{-1/2})$
	$$ \lambda_0(0) := \frac{q_0^{-1/2}}{1+q_0^{-1}} \left( q_0^{-1/2} + q_0^{1/2+2i\tau} \right), \quad \text{resp.} \quad \tilde{\lambda}_0(0) := \frac{q_0^{-1/2}}{1+q_0^{-1}} \left( q_0^{-1/2+2i\tau} + q_0^{1/2} \right); $$
	and finally the \emph{operator/measure of regularization}
	$$ \sigma_0 := (T_0(1) - \lambda_0(0))^2 (T_0(1) - \tilde{\lambda}_0(0))^2. $$
	Write $\eis_2^* = \eis^*(0,f_2)$ resp. $\eis_3^{\sharp} = \eis^{\sharp}(0,f_3)$. Lemma \ref{RDRegProd} shows that $\sigma_0 (\eis_2^* \eis_3^{\sharp}) \in \mathcal{A}(\GL_2, \norm_{\A}^{-2i\tau})$ is of rapid decay (see Proposition \ref{DualHeckeOp}).
	
	In this subsection, we propose to bound
	$$ \Norm[\sigma_0 (\eis_2^* \eis_3^{\sharp})]_2. $$
	By the definition of adjoint measure (see Proposition \ref{DualHeckeOp}), we have
	$$ \Norm[\sigma_0 (\eis_2^* \eis_3^{\sharp})]_2^2 = \int_{[\PGL_2]} \sigma_0(\eis_2^* \eis_3^{\sharp}) \cdot \overline{ \sigma_0(\eis_2^* \eis_3^{\sharp}) } = \int_{[\PGL_2]} \sigma_0^*\sigma_0(\eis_2^* \eis_3^{\sharp}) \cdot \overline{ \eis_2^* \eis_3^{\sharp} }. $$
	From Proposition \ref{HeckeRel}, we deduce that
	$$ \sigma_0^*\sigma_0 = \sideset{}{_{j=0}^8} \sum b_j T_0(j)^* $$
	for some $b_j \in \C$ with $\norm[b_j] \ll 1$. Hence we can apply the $\gp{K}$-invariance property of the regularized integrals \cite[Proposition 2.27 (2)]{Wu9} and rewrite
	$$ \Norm[\sigma_0 (\eis_2^* \eis_3^{\sharp})]_2^2 = \sideset{}{_{j=0}^8} \sum b_j \int_{[\PGL_2]}^{\reg} T_0(j)^*(\eis_2^* \eis_3^{\sharp}) \cdot \overline{(\eis_2^* \eis_3^{\sharp}) } = \sideset{}{_{j=0}^8} \sum b_j \int_{[\PGL_2]}^{\reg} a(\varpi_0^{-j}).(\eis_2^* \eis_3^{\sharp}) \cdot \overline{(\eis_2^* \eis_3^{\sharp}) }. $$
	We re-arrange the last integral as
\begin{align}
	&\quad \int_{[\PGL_2]}^{\reg} a(\varpi_0^{-j}).\left( \eis_2^* \eis_3^{\sharp} \right) \cdot \overline{ \left( \eis_2^* \eis_3^{\sharp} \right) } \nonumber \\
	&= \int_{[\PGL_2]} \left( a(\varpi_0^{-j}).\eis_2^* \cdot \overline{\eis_2^*} - \Reis(a(\varpi_0^{-j}).\eis_2^* \overline{\eis_2^*}) \right) \cdot \left( a(\varpi_0^{-j}).\eis_3^{\sharp} \cdot \overline{\eis_3^{\sharp}} - \Reis(a(\varpi_0^{-j}).\eis_3^{\sharp} \overline{\eis_3^{\sharp}}) \right) \label{4MRegTerm} \\
	&+ \int_{[\PGL_2]}^{\reg} a(\varpi_0^{-j}).\eis_2^* \cdot \overline{\eis_2^*} \cdot \Reis(a(\varpi_0^{-j}).\eis_3^{\sharp} \overline{\eis_3^{\sharp}}) + \int_{[\PGL_2]}^{\reg} a(\varpi_0^{-j}).\eis_3^{\sharp} \cdot \overline{\eis_3^{\sharp}} \cdot \Reis(a(\varpi_0^{-j}).\eis_2^* \overline{\eis_2^*}) \label{4MReguTerm} \\
	&- \int_{[\PGL_2]}^{\reg} \Reis(a(\varpi_0^{-j}).\eis_2^* \overline{\eis_2^*}) \cdot \Reis(a(\varpi_0^{-j}).\eis_3^{\sharp} \overline{\eis_3^{\sharp}}). \label{4MDgnTerm}
\end{align}
	We refer to (\ref{4MRegTerm}) resp. (\ref{4MReguTerm}) resp. (\ref{4MDgnTerm}) as the \emph{regular term} resp. \emph{regularized terms} resp. \emph{degenerate term}. The smallness of $q_0$, implying that any power of it is $\ll_{\epsilon} \Nr(\idlN)^{\epsilon}$ hence negligible, allows us to basically ignore the contribution at $\vp_0$ in the estimation.
	
\begin{lemma}
	The regular resp. regularized resp. degenerate term(s) can be bounded as
	$$ (C^{r_1+r_2}\Nr(\idlN))^{\epsilon}. $$
	Consequently, we have $\Norm[\sigma_0 (\eis_2^* \eis_3^{\sharp})]_2 \ll_{\epsilon} (C^{r_1+r_2}\Nr(\idlN))^{\epsilon}$.
\label{4MNormBd}
\end{lemma}
\begin{proof}
	The proof is analogous to the proof of our main theorem. We will give details for the regular term, to illustrate the similarity. The technical estimations will occupy the rest of the subsection, where some modifications/variations will be emphasized. In particular, the possible circular reasoning using \cite[Corollary 6.7]{Wu14} should be removed.
\end{proof}
		
\begin{remark}
	For convenience of notations, we shall write $\idlN_{\vp} := \vp^{\mathrm{ord}_{\vp}(\idlN)}$ in the sequel.
\end{remark}
			
	The treatment of the regular term is parallel to \S \ref{MainRT}. We expand (\ref{4MRegTerm}) by the Plancherel formula for $\intL^2(\GL_2, \mathbbm{1})$, inspecting the invariance of $a(\varpi_0^{-j}).\eis_2^* \cdot \overline{\eis_2^*}$, as
\begin{align*}
	&\quad \sum_{\pi \in \mathcal{A}_0(\infty,\vp_0^j)} \sum_{\varphi \in \Bas(\pi)} \left( \int_{[\PGL_2]} a(\varpi_0^{-j}).\eis_2^* \cdot \overline{\eis_2^*} \cdot \overline{\varphi} \right) \cdot \left( \int_{[\PGL_2]} \varphi \cdot a(\varpi_0^{-j}).\eis_3^{\sharp} \cdot \overline{\eis_3^{\sharp}} \right) \\
	&+ \sum_{\xi} \sum_{\Phi \in \Bas(\xi,\xi^{-1})} \int_{-\infty}^{\infty} \left( \int_{[\PGL_2]}^{\reg} a(\varpi_0^{-j}).\eis_2^* \cdot \overline{\eis_2^*} \cdot \overline{\eis(i\tau_1,\Phi)} \right) \cdot \left( \int_{[\PGL_2]}^{\reg} \eis(i\tau_1,\Phi) \cdot a(\varpi_0^{-j}).\eis_3^{\sharp} \cdot \overline{\eis_3^{\sharp}} \right) \frac{d\tau_1}{4\pi} \\
	&+ \frac{1}{\Vol([\PGL_2])} \sum_{\chi^2 = 1} \left( \int_{[\PGL_2]}^{\reg} a(\varpi_0^{-j}).\eis_2^* \cdot \overline{\eis_2^*} \cdot \overline{\chi \circ \det} \right) \cdot \left( \int_{[\PGL_2]}^{\reg} \chi \circ \det \cdot a(\varpi_0^{-j}).\eis_3^{\sharp} \cdot \overline{\eis_3^{\sharp}} \right),
\end{align*}
	where $\xi$ resp. $\chi$ is subject to the conditions 
	$$ \cond(\xi_{\vp}) \text{ resp. } \cond(\chi_{\vp}) = 0, \quad \forall \vp \neq \vp_0, \quad \cond(\xi_{\vp_0}) \text{ resp. } \cond(\chi_{\vp_0}) \leq j/2. $$
	We refer to the first resp. second resp. third line as the cuspidal resp. Eisenstein resp. one-dimensional contribution, similar to (\ref{CuspRegTerm}) - (\ref{OneDRegTerm}). We put $f_3' := a(\varpi_0^{-j})f_3$.
	
	For the cuspidal contribution, we apply (\ref{CuspPerD}) to decompose
	$$ \int_{[\PGL_2]} \varphi \cdot a(\varpi_0^{-j}).\eis_3^{\sharp} \cdot \overline{\eis_3^{\sharp}} = L(\frac{1}{2},\pi) L(\frac{1}{2}+2i\tau,\pi) \cdot \prod_v \ell_v(W_{\varphi,v}, \overline{f_{3,v}}, f_{3,v}'). $$
	At $\vp_0 \neq \vp < \infty$, Lemma \ref{LocUpperBdNA} (1) (with $\cond(\omega_{\vp}) = 0$) gives
	$$ \extnorm{ \ell_{\vp}(W_{\varphi,\vp}, \overline{f_{3,\vp}}, f_{3,\vp}') } \ll \Nr(\idlN_{\vp})^{-1/2+\theta} \cdot \frac{\Norm[W_{\varphi,\vp}]}{\sqrt{L(1,\pi_{\vp} \times \overline{\pi}_{\vp})}}; $$
	at $\vp = \vp_0$, Lemma \ref{LocAmpBd} (1) gives
	$$ \extnorm{ \ell_{\vp}(W_{\varphi,\vp}, \overline{f_{3,\vp}}, f_{3,\vp}') } \ll q_0^{-j/2} \cdot \frac{\Norm[W_{\varphi,\vp}]}{\sqrt{L(1,\pi_{\vp} \times \overline{\pi}_{\vp})}}; $$
	at $v \mid \infty$, Lemma \ref{LocUpperBdA} (1) gives
	$$ \extnorm{ \ell_v(W_{\varphi,v}, \overline{f_{3,v}}, f_{3,v}') } \ll_{\epsilon} C^{-1/2} \left( C/(1+\norm[\tau]) \right)^{\theta+\epsilon} \cdot S_d(W_{\varphi,v}). $$
	Unlike \S \ref{MainRT}, we bound the $L$-values by the convex bound because we are proving \cite[Corollary 6.7]{Wu14} hence can not apply it. Taking into account \cite{HL94} and \cite[Lemma 3]{BH10} giving bounds for $L(1,\pi \times \overline{\pi})$, we get
	$$ \extnorm{ \int_{[\PGL_2]} \varphi \cdot a(\varpi_0^{-j}).\eis_3^{\sharp} \cdot \overline{\eis_3^{\sharp}} } \ll_{\epsilon} \left( (1+\norm[\tau])/C \right)^{(r_1+r_2)(1/2-\theta-\epsilon)} \Nr(\idlN)^{-1/2+\theta} q_0^{j\epsilon} S_{d+1/2}(\varphi). $$
	Inserting Weyl's law and applying Cauchy-Schwarz, we get for some $d' \gg d+1/2$
\begin{align*}
	&\quad \sum_{\pi \in \mathcal{A}_0(\infty,\vp_0^j)} \sum_{\varphi \in \Bas(\pi)} \extnorm{ \int_{[\PGL_2]} a(\varpi_0^{-j}).\eis_2^* \cdot \overline{\eis_2^*} \cdot \overline{\varphi} } \cdot S_{d+1/2}(\varphi) \\
	&\ll \sum_{\pi \in \mathcal{A}_0(\infty,\vp_0^j)} \sum_{\varphi \in \Bas(\pi)} \extnorm{ \Pairing{a(\varpi_0^{-j}).\eis_2^* \cdot \overline{\eis_2^*}}{\Delta_{\infty}^{d'} \varphi}} \cdot \lambda_{\varphi,\infty}^{-(d'-d-1/2)} \\
	&\ll_{\epsilon} S_{d'}(a(\varpi_0^{-j}).\eis_2^* \cdot \overline{\eis_2^*} - \Reis(a(\varpi_0^{-j}).\eis_2^* \overline{\eis_2^*})) \cdot q_0^{j(1+\epsilon)}.
\end{align*}
	\cite[Theorem 5.4]{Wu2} applies to bound the above Sobolev norm as a power of $j \log q_0$. We obtain a bound for the cuspidal contribution as
\begin{align}
	&\quad \sum_{\pi \in \mathcal{A}_0(\infty,\vp_0^j)} \sum_{\varphi \in \Bas(\pi)} \left| \int_{[\PGL_2]} a(\varpi_0^{-j}).\eis_2^* \cdot \overline{\eis_2^*} \cdot \overline{\varphi} \right| \cdot \left| \int_{[\PGL_2]} \varphi \cdot a(\varpi_0^{-j}).\eis_3^{\sharp} \cdot \overline{\eis_3^{\sharp}} \right| \nonumber \\
	&\ll_{\epsilon} \left( (1+\norm[\tau])/C \right)^{(r_1+r_2)(1/2-\theta-\epsilon)} \Nr(\idlN)^{-1/2+\theta+\epsilon}. \label{4MRTCusp}
\end{align}

	For the Eisenstein contribution, we apply (\ref{EisPerD}) to decompose
\begin{align*} 
	&\quad \int_{[\PGL_2]}^{\reg} \eis(i\tau_1,\Phi) \cdot a(\varpi_0^{-j}).\eis_3^{\sharp} \cdot \overline{\eis_3^{\sharp}} \\
	&= \frac{L(\frac{1}{2}+i\tau_1, \xi)^2 L(\frac{1}{2}+i(\tau_1+2\tau), \xi) L(\frac{1}{2}+i(\tau_1-2\tau), \xi)}{L(1+2i\tau_1, \xi^2)} \prod_v \ell_v(i\tau_1; \Phi_v, \overline{f_{3,v}}, f_{3,v}').
\end{align*}
	At $\vp_0 \neq \vp < \infty$, Lemma \ref{LocUpperBdNA} (2) gives
	$$ \extnorm{ \ell_{\vp}(i\tau_1; \Phi_{\vp}, \overline{f_{3,\vp}}, f_{3,\vp}') } \ll_{\epsilon} \Nr(\idlN_{\vp})^{-1/2+\epsilon}; $$
	at $\vp = \vp_0$, Lemma \ref{LocAmpBd} (2) gives
	$$ \extnorm{ \ell_{\vp}(i\tau_1; \Phi_{\vp}, \overline{f_{3,\vp}}, f_{3,\vp}') } \ll q_0^{-j/2}; $$
	at $v \mid \infty$, Lemma \ref{LocUpperBdA} (2) gives
	$$ \extnorm{ \ell_v(i\tau_1; \Phi_v, \overline{f_{3,v}}, f_{3,v}') } \ll_{\epsilon} C^{-1/2+\epsilon} S_d(\Phi_{v,i\tau_1}). $$
	We apply the convex bounds for $L(s,\xi)$ and Siegel's lower bound for $L(1+2i\tau_1, \xi^2)$ to get
	$$ \extnorm{ \int_{[\PGL_2]}^{\reg} \eis(i\tau_1,\Phi) \cdot a(\varpi_0^{-j}).\eis_3^{\sharp} \cdot \overline{\eis_3^{\sharp}} } \ll_{\epsilon} \left( (1+\norm[\tau])/C \right)^{(r_1+r_2)(1/2-\epsilon)} \Nr(\idlN)^{-1/2} q_0^{j\epsilon} S_{d+1/2}(\Phi_{i\tau_1}). $$
	Applying Cauchy-Schwarz, we get for some $d' \gg d+1/2$
\begin{align*}
	&\quad \sum_{\xi} \sum_{\Phi \in \Bas(\xi,\xi^{-1})} \int_{-\infty}^{\infty} \left| \int_{[\PGL_2]}^{\reg} a(\varpi_0^{-j}).\eis_2^* \cdot \overline{\eis_2^*} \cdot \overline{\eis(i\tau_1,\Phi)} \right| \cdot S_{d+1/2}(\Phi_{i\tau_1}) \frac{d\tau_1}{4\pi} \\
	&\ll \sum_{\xi} \sum_{\Phi \in \Bas(\xi,\xi^{-1})} \int_{-\infty}^{\infty} \extnorm{ \Pairing{a(\varpi_0^{-j}).\eis_2^* \cdot \overline{\eis_2^*} - \Reis(a(\varpi_0^{-j}).\eis_2^* \overline{\eis_2^*}) }{\Delta_{\infty}^{d'}\eis(i\tau_1,\Phi)} } \cdot \lambda_{\Phi_{i\tau_1}, \infty}^{-(d'-d-1/2)} \frac{d\tau_1}{4\pi} \\
	&\ll_{\epsilon} S_{d'}(a(\varpi_0^{-j}).\eis_2^* \cdot \overline{\eis_2^*} - \Reis(a(\varpi_0^{-j}).\eis_2^* \overline{\eis_2^*})) \cdot \left( \sum_{\xi} \sum_{\Phi \in \Bas(\xi,\xi^{-1})} \int_{-\infty}^{\infty} \lambda_{\Phi_{i\tau_1}, \infty}^{-2(d'-d-1/2)} \frac{d\tau_1}{4\pi} \right)^{1/2} \\
	&\ll_{\epsilon} S_{d'}(a(\varpi_0^{-j}).\eis_2^* \cdot \overline{\eis_2^*} - \Reis(a(\varpi_0^{-j}).\eis_2^* \overline{\eis_2^*})) \cdot q_0^{j(1+\epsilon)/2}.
\end{align*}
	\cite[Theorem 5.4]{Wu2} applies to bound the above Sobolev norm as a power of $j \log q_0$. We obtain a bound for the Eisenstein contribution as
\begin{align}
	&\quad \sum_{\xi} \sum_{\Phi \in \Bas(\xi,\xi^{-1})} \int_{-\infty}^{\infty} \left| \int_{[\PGL_2]}^{\reg} a(\varpi_0^{-j}).\eis_2^* \cdot \overline{\eis_2^*} \cdot \overline{\eis(i\tau_1,\Phi)} \right| \cdot \left| \int_{[\PGL_2]}^{\reg} \eis(i\tau_1,\Phi) \cdot a(\varpi_0^{-j}).\eis_3^{\sharp} \cdot \overline{\eis_3^{\sharp}} \right| \frac{d\tau_1}{4\pi} \nonumber \\
	&\ll_{\epsilon} \left( (1+\norm[\tau])/C \right)^{(r_1+r_2)(1/2-\epsilon)} \Nr(\idlN)^{-1/2+\epsilon}. \label{4MRTEis}
\end{align}

	Applying Remark \ref{OneDimProjRk} with $\omega=1$, we see that
	$$ \int_{[\PGL_2]}^{\reg} a(\varpi_0^{-j}).\eis_2^* \cdot \overline{\eis_2^*} \cdot \overline{\chi \circ \det} $$
is non-vanishing only if $\chi=1$. Denote by $T(\vp_0^j)$ the Hecke operator 
	$$ T(\vp_0^j) := \int_{\gp{K}_0^2} \rpR_{\mathbbm{1}}^{\mathcal{A}}(\kappa_1 a(\varpi_0^j) \kappa_2) d\kappa_1 d\kappa_2 = T(\vp_0^j)^*. $$
The $\gp{K}$-invariance of regularized integrals \cite[Proposition 2.27]{Wu9} implies
	$$ \extnorm{ \int_{[\PGL_2]}^{\reg} a(\varpi_0^{-j}).\eis_2^* \cdot \overline{\eis_2^*} } = \extnorm{ \int_{[\PGL_2]}^{\reg} T(\vp_0^j)^*.\eis_2^* \cdot \overline{\eis_2^*} } \ll_{\epsilon} q_0^{-\frac{j}{2}+\epsilon} \extnorm{ \int_{[\PGL_2]}^{\reg} \eis_2^* \cdot \overline{\eis_2^*} } \ll_{\F,\epsilon} q_0^{-\frac{j}{2}+\epsilon}. $$
	A similar argument leads to
	$$ \extnorm{ \int_{[\PGL_2]}^{\reg} a(\varpi_0^{-j}).\eis_3^{\sharp} \cdot \overline{\eis_3^{\sharp}} } \ll_{\epsilon} q_0^{-\frac{j}{2}+\epsilon} \extnorm{ \int_{[\PGL_2]}^{\reg} \eis_3^{\sharp} \cdot \overline{\eis_3^{\sharp}} }. $$
	By (\ref{OneDimProj}), Lemma \ref{LocUpperBdNA} (3), Lemma \ref{LocUpperBdA} (3), together with $\extnorm{\zeta_{\F}(1\pm 2i\tau)} \ll_{\epsilon} (1+\norm[\tau])^{\epsilon}$, we get
\begin{align} 
	&\quad \frac{1}{\Vol([\PGL_2])} \sum_{\chi^2 = 1} \left| \int_{[\PGL_2]}^{\reg} a(\varpi_0^{-j}).\eis_2^* \cdot \overline{\eis_2^*} \cdot \overline{\chi \circ \det} \right| \cdot \left| \int_{[\PGL_2]}^{\reg} \chi \circ \det \cdot a(\varpi_0^{-j}).\eis_3^{\sharp} \cdot \overline{\eis_3^{\sharp}} \right| \nonumber \\
	&\ll_{\epsilon} q_0^{-j} (q_0 C^{r_1+r_2}\Nr(\idlN))^{\epsilon}. \label{4MRTOneD}
\end{align}

	The relevant bound in Lemma \ref{4MNormBd} follows from (\ref{4MRTCusp}), (\ref{4MRTEis}) and (\ref{4MRTOneD}).
\begin{remark}
	Note that the one-dimensional contribution is the main term, to which amplification should apply in order to break the convexity bound.
\end{remark}

	\subsection{Proof of Main Bound}
	
	We are ready to prove Theorem \ref{4MMBd}. We have constructed two Eisenstein series $\eis_2^*$ and $\eis_3^{\sharp}$, and chosen $\varphi \in \pi \otimes \norm_{\A}^{i\tau}$ for each $\pi \in \mathcal{A}(C,\idlN)$. Note that Lemma \ref{RDRegProd} (1) is still valid, i.e., $\varphi$ is an eigenvector for $\sigma_0^{\vee}$ whose eigenvalue $R_0$ satisfies $\norm[R_0] \asymp 1$. The local lower bounds (\ref{4MLocLowerBdA}) and (\ref{4MLocLowerBdNA}) readily imply
\begin{align*} 
	&\quad \extnorm{ \int_{[\PGL_2]} \varphi \cdot \sigma_0(\eis_2^* \cdot \eis_3^{\sharp}) } = \extnorm{ \int_{[\PGL_2]} \sigma_0^{\vee}(\varphi) \cdot \eis_2^* \cdot \eis_3^{\sharp} } = \extnorm{ R_0 \int_{[\PGL_2]} \varphi \cdot \eis_2^* \cdot \eis_3^{\sharp} } \\
	&\gg \extnorm{ \int_{[\PGL_2]} \varphi \cdot \eis_2^* \cdot \eis_3^{\sharp} } \gg_{\epsilon} \extnorm{L(1/2+i\tau, \pi)}^2 (C\Nr(\idlN))^{-1/2-\epsilon}. 
\end{align*}
	Square the above equality. Sum over $\varphi$ and apply Lemma \ref{4MNormBd}. We get and conclude by
\begin{align*}
	\sideset{}{_{\pi \in \mathcal{A}_0(C,\idlN)}} \sum \extnorm{L(1/2+i\tau, \pi)}^4 (C\Nr(\idlN))^{-1-2\epsilon} &\ll_{\epsilon} \sideset{}{_{\pi \in \mathcal{A}_0(C,\idlN)}} \sum \extnorm{ \int_{[\PGL_2]} \varphi \cdot \sigma_0(\eis_2^* \cdot \eis_3^{\sharp}) }^2 \\
	&\leq \extNorm{ \sigma_0(\eis_2^* \cdot \eis_3^{\sharp}) }_2^2 \ll_{\epsilon} (C\Nr(\idlN))^{\epsilon}.
\end{align*}

\section*{Acknowledgement}
	
	The preparation of this paper is completed during the stay of the author at the math department of BUAA, in the MTA R\'enyi Int\'ezet Lend\"ulet Automorphic Research Group at Alf\'ed Renyi Institute in Hungary, in TAN at EPFL and in IMS at the National University of Singapore. The author would like to thank all these institutes for their hospitality. The author would also like to thank the anonymous referee for careful reading and useful suggestions, which improves the readability of the paper.

\bibliographystyle{acm}
	
\bibliography{mathbib}

\address{\quad \\ Han WU \\ School of Mathematical Sciences \\ Queen Mary University of London \\ Mile End Road \\ London, E1 4NS \\ United Kingdom \\ wuhan1121@yahoo.com}	
	
\end{document}